\documentclass[12pt,twoside]{article}
\usepackage[utf8]{inputenc}
\usepackage[T1]{fontenc}
\usepackage[english]{babel}
\usepackage{amsmath, amssymb, amsthm}            
\usepackage{amstext, amsfonts, a4}
\usepackage{graphicx}
\usepackage{hyperref}
\usepackage{pifont}
\usepackage[svgnames]{xcolor}
\usepackage{color}
\usepackage{stmaryrd}
\usepackage{enumitem}
\usepackage{tikz}
\usepackage{comment}
\usepackage{mathrsfs}
\usepackage{pgfplots}
\pgfplotsset{compat=1.15}
\usetikzlibrary{arrows}

\theoremstyle{plain}
	\newtheorem{Theo}{Theorem}[section] 
	\newtheorem{Prop}[Theo]{Proposition}        %même compteur que pour les théorèmes      
	\newtheorem{Lem}[Theo]{Lemma}            %etc...
	\newtheorem{Cor}[Theo]{Corollary}
	
	\newtheorem{Conj}[Theo]{Conjecture}

\theoremstyle{definition}
	\newtheorem{Def}[Theo]{Definition}
        
	\newtheorem{Nota}[Theo]{Notation}

\theoremstyle{remark}
	\newtheorem{Rema}[Theo]{Remark}

\def\NN{{\mathbb N}}    %naturels
\def\RR{{\mathbb R}}    %réels
\def\CC{{\mathbb C}}    %complexes
\def\HH{{\mathbb H}^2}    % espace hyperbolique
\newcommand{\Hyp}{{\mathbb H}}

\def\SL{{SL_{2}(\mathbb R)}}

\newcommand{\Int}{\mbox{Int}}
\newcommand{\Vol}{\mbox{Vol}}

\def\GIntij{{|\alpha_i \cap \beta_j|}}
\def\Int{{\mathrm{Int}}}
\def\XX{{X_n}}
\def\SS{{\mathcal{S}_n}}
\def\TT{{\mathcal{T}_n}}
\def\KVol{{\mathrm{KVol}}}

\title{Algebraic intersection, lengths and Veech surfaces}
\author{\textsc{Julien Boulanger}}
\begin{document}
\maketitle

%\begin{abstract}
%In this paper, we continue the study of intersections of closed curves on translation surfaces, initiated in \cite{CKM} and \cite{CKMcras} for a family of arithmetic Veech surfaces and \cite{BLM22} for a family of non-arithmetic Veech surfaces. Namely, we investigate the question of maximizing the algebraic intersection between two curves of given lengths, and we focus on the case of translation surfaces in the Teichm\"uller disk of the regular $n$-gons for even $n$.
%\end{abstract}

\section{Introduction}
In this paper, we continue the study of intersections of closed curves on translation surfaces, initiated in \cite{CKM} and \cite{CKMcras} for a family of arithmetic Veech surfaces and \cite{BLM22} for a family of non-arithmetic Veech surfaces. Namely, we investigate the question of maximizing the algebraic intersection between two curves of given lengths. A suitable way of quantifying this is to consider the following quantity, defined for any closed oriented surface $X$ with a Riemannian metric $g$ (possibly with singularities):
\begin{equation}
\label{eq:KVol}
\mathrm{KVol}(X): = \mathrm{Vol}(X,g)\cdot \sup_{\alpha,\beta} \frac{\mathrm{Int} (\alpha,\beta)}{l_g (\alpha) l_g (\beta)},
\end{equation}
%$$
%\sup_{\alpha,\beta} \frac{\Int (\alpha,\beta)}{l_g (\alpha) l_g (\beta)},
%$$
where the supremum ranges over all piecewise smooth closed curves $\alpha$ and $\beta$ in $X$, $\Int$ denotes the algebraic intersection, and $l_g(\cdot)$ denotes the length with respect to the Riemannian metric (it is readily seen that multiplying by the volume Vol$(X,g)$ makes the quantity invariant by rescaling the metric $g$). As shown by Massart-Müetzel~\cite{MM}, this function is well defined and finite.

Though KVol is a close cousin of the systolic volume $\mathrm{SysVol}(X) = \sup_{\alpha} \frac{\mathrm{Vol}(X)}{l_g(\alpha)^2}$, it is difficult to compute on a given surface outside the case of a flat torus (where KVol $= 1$). In this context, Cheboui, Kessi and Massart~\cite{CKM,CKMcras} initiated the study of KVol on translation surfaces, which are instances of flat surfaces with conical singularities. More precisely,~\cite{CKM} provides estimates of KVol on the Teichm\"uller curves associated with a family of arithmetic translation surfaces $(X,\omega)$. In Boulanger--Lanneau--Massart~\cite{BLM22}, we give a closed formula for KVol on the Teichm\"uller disk of the double regular $n$-gon translation surface for $n \geq 5$ odd. In this paper, we continue the study of KVol in the Teichm\"uller disk of translation surfaces coming from regular polygons. Namely, we deal with the case of the regular $n$-gon for $n \geq 8$ even. Although similar to the case of the double regular $n$-gon for odd $n$, the case of the regular $n$-gon for even $n$ requires a more careful study.\newline

Given an even integer $n \geq 8$, we denote by $\XX$ the translation surface made from a regular $n$-gon by identifying its parallel opposite sides by translations. The resulting surface has a unique conical singularity if $n \equiv 0 \mod 4$ and two distinct conical singularities if $n \equiv 2 \mod 4$. This surface can also be obtained by the unfolding construction of Katok--Zemlyakov \cite{KZ} from a triangle of angles $(\frac{\pi}{2}, \frac{\pi}{n}, \frac{(n/2-1)\pi}{n})$ and is one of the original Veech surfaces (see \cite{Veech}). As it will be recalled in Section \ref{sec:background}, the Veech group of $\XX$ has a \emph{staircase model} $\mathcal{S}$ in its Teichm\"uller disk whose Veech group $\Gamma_n$ is an index two subgroup of the Hecke group of order $n$. In particular, the Teichm\"uller curve associated to $\XX$ can be identified with $\HH / \Gamma_n$ and has a fundamental domain $\TT$ depicted in Figure \ref{geodesics_kPhi}, where $\Phi := \Phi_{n} = 2 \cos( \frac{\pi}{n})$. 

In this paper, we study KVol in the Teichm\"uller disk of $\XX$, and we give a formula to compute KVol on any surface of $\TT$ for $n \equiv 0 \mod 4$. If $n \equiv 2 \mod 4$, the fact that $\XX$ has two singularities makes it more difficult to compute KVol. In this latter case, our methods provide upper bounds on KVol in the Teichm\"uller disk of $\XX$.

\paragraph{The case $n \equiv 0 \mod 4$.}
Our main result holds in the case where $n \equiv 0 \mod 4$, so that the regular $n$-gon as a single singularity, and can be stated as:

\begin{Theo}\label{theo:main}
Let $n \geq 8$ such that $n \equiv 0 \mod 4$. Given $d,d' \in \RR \cup \{ \infty \} \simeq \partial \HH$, let $\gamma_{d,d'}$ denote the geodesic in the hyperbolic plane $\HH$ having $d$ and $d'$ as endpoints, and define:
$$
\mathcal{G}_{max} = \bigcup_{k \in \NN^* \cup \{\infty\}} \gamma_{\infty, \pm \frac{1}{k\Phi}}
$$
(with the convention $\frac{1}{\infty} = 0$). \newline

Let $X = M \cdot \mathcal{S}_n$ be a surface in the Teichm\"uller disk of $\XX$, obtained from the staircase model $\mathcal{S}$ by applying a matrix $M = \begin{pmatrix} a & b \\ c & d \end{pmatrix} \in SL_2(\RR)$. Then, we have:
\begin{equation}\label{Formule_1}
\KVol(X) = K_0 \cdot \frac{1}{\cosh(\mathrm{dist}_{\HH}(\frac{di+b}{ci+a}, \Gamma_n \cdot \mathcal{G}_{max}))}
\end{equation}
Where $K_0 > 0$ is an explicit constrant which only depends on $n$ and $\mathrm{dist}_{\HH}$ denotes the hyperbolic distance. \newline

In particular, KVol is bounded on the Teichm\"uller disk of the regular $n$-gon, and
\begin{enumerate}[label=(\roman*)]
\item the maximum of KVol is achieved for surfaces represented by images of elements of $\mathcal{G}_{max}$ under the group $\Gamma_n$,
\item the minimum of KVol is achieved, uniquely, at $\XX$.
\end{enumerate}
\end{Theo}

%\begin{figure}[h]
%\center
%\definecolor{ffzztt}{rgb}{1,0.6,0.2}
%\definecolor{qqzzqq}{rgb}{0,0.6,0}
%\definecolor{qqqqff}{rgb}{0,0,1}
%\definecolor{ffqqqq}{rgb}{1,0,0}
%\begin{tikzpicture}[line cap=round,line join=round,>=triangle 45,x=0.4cm,y=0.4cm]
%\clip(-10,-0.1) rectangle (10,15);
%\draw [line width=2pt,color=ffqqqq] (-3,0)-- (-7.242640687119286,4.2426406871192865);
%\draw [line width=2pt,color=qqqqff] (-7.242640687119286,4.2426406871192865)-- (-7.242640687119284,10.242640687119286);
%\draw [line width=2pt,color=qqzzqq] (-7.242640687119284,10.242640687119286)-- (-3,14.485281374238571);
%\draw [line width=2pt,color=ffzztt] (-3,14.485281374238571)-- (3,14.48528137423857);
%\draw [line width=2pt,color=ffqqqq] (3,14.48528137423857)-- (7.242640687119286,10.242640687119284);
%\draw [line width=2pt,color=qqqqff] (7.242640687119286,10.242640687119284)-- (7.242640687119285,4.242640687119285);
%\draw [line width=2pt,color=qqzzqq] (7.242640687119285,4.242640687119285)-- (3,0);
%\draw [line width=2pt,color=ffzztt] (3,0)-- (-3,0);

%\end{tikzpicture}
%\caption{The octagon with parallel sides identified.}
%\label{Octogone}
%\end{figure}

%\section{The geodesics $(\infty, \frac{1}{k\Phi})$ and their images by the Veech group}
%\begin{Prop}
%For any $k \geq 1$, the geodesic $(\infty, \frac{1}{k\Phi})$ is the closest to the $n$-gon $\XX$ among all its images by the Veech group.
%\end{Prop}
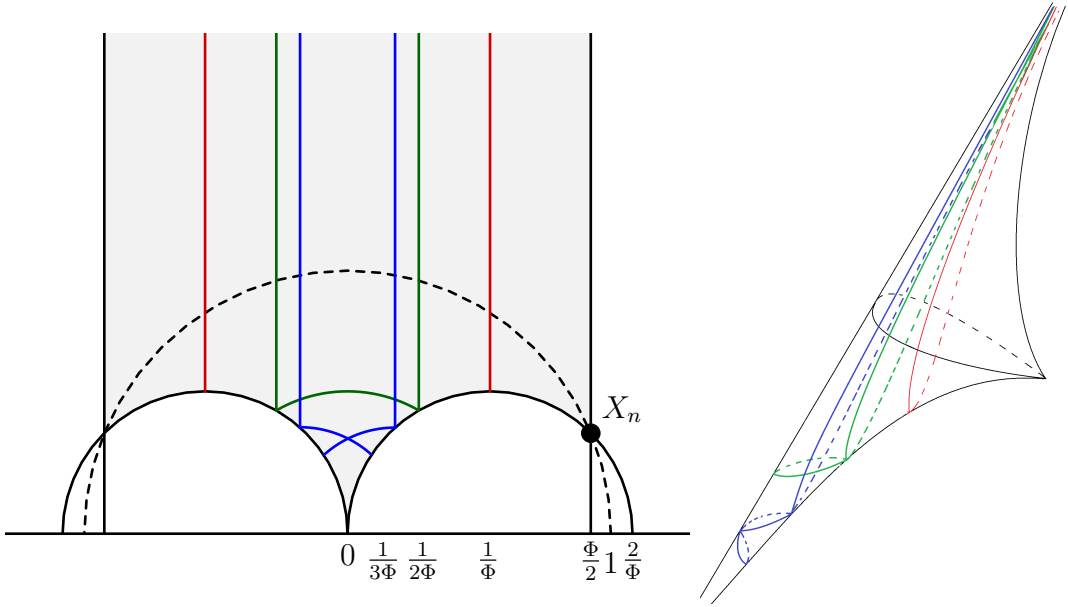
\begin{figure}[h]
\center
\definecolor{qqqqff}{rgb}{0,0,1}
\definecolor{qqwuqq}{rgb}{0,0.39215686274509803,0}
\definecolor{ccqqqq}{rgb}{0.8,0,0}
\begin{tikzpicture}[line cap=round,line join=round,>=triangle 45,x=1cm,y=1cm, scale = 3.5]
\clip(-1.3,-0.3) rectangle (1.3,1.9);
\fill [gray!10](-0.923879532511,0) rectangle (0.923879532511,1.95);
\filldraw[fill=white] (1.0823922002,0) -- (1.0823922002,0) arc (0:180:0.54119610014) (0,0) -- cycle;
\filldraw[fill=white] (0,0) -- (0,0) arc (0:180:0.5411961001) (1.0823922002,0)  -- cycle;
\draw [shift={(0,0)},line width=1pt, dash pattern=on 3pt off 3pt]  plot[domain=0:3.141592653589793,variable=\t]({1*1*cos(\t r)+0*1*sin(\t r)},{0*1*cos(\t r)+1*1*sin(\t r)});
\draw [line width=1pt] (0.9238795325112867,0) -- (0.9238795325112867,2.34877816657346);
\draw [line width=1pt] (-0.9238795325112867,0) -- (-0.9238795325112867,2.34877816657346);
\draw [shift={(0.541196100146197,0)},line width=1pt]  plot[domain=0:3.141592653589793,variable=\t]({1*0.541196100146197*cos(\t r)+0*0.541196100146197*sin(\t r)},{0*0.541196100146197*cos(\t r)+1*0.541196100146197*sin(\t r)});
\draw [shift={(-0.541196100146197,0)},line width=1pt]  plot[domain=0:3.141592653589793,variable=\t]({1*0.541196100146197*cos(\t r)+0*0.541196100146197*sin(\t r)},{0*0.541196100146197*cos(\t r)+1*0.541196100146197*sin(\t r)});
\draw [line width=1pt,color=ccqqqq] (-0.541196100146197,0.541196100146197) -- (-0.541196100146197,2.34877816657346);
\draw [line width=1pt,color=ccqqqq] (0.541196100146197,0.541196100146197) -- (0.541196100146197,2.34877816657346);
\draw [line width=1pt,domain=-2.2811437913952344:2.305963049326692] plot(\x,{(-0-0*\x)/2});
\draw (-0.07,0) node[anchor=north west] {$0$};
\draw (0.84,0.01) node[anchor=north west] {$\frac{\Phi}{2}$};
%\draw (0.9684922620546951,0.05000025094971709) node[anchor=north west] {$1$};
\draw (1,0.01) node[anchor=north west] {$\frac{2}{\Phi}$};
\draw (0.93,-0.03) node[anchor=north west] {$1$};
\draw (0.45,0.01) node[anchor=north west] {$\frac{1}{\Phi}$};
\draw [line width=1pt,color=qqwuqq] (0.2705980500730985,0.4686895711556737) -- (0.2705980500730985,2.34877816657346);
\draw [line width=1pt,color=qqwuqq] (-0.27059805007309845,0.4686895711556738) -- (-0.27059805007309845,2.34877816657346);
\draw [shift={(0,0)},line width=1pt,color=qqwuqq]  plot[domain=1.0471975511965976:2.0943951023931957,variable=\t]({1*0.541196100146197*cos(\t r)+0*0.541196100146197*sin(\t r)},{0*0.541196100146197*cos(\t r)+1*0.541196100146197*sin(\t r)});
\draw (0.18,0.01) node[anchor=north west] {$\frac{1}{2\Phi}$};
\draw (0.03,0.01) node[anchor=north west] {$\frac{1}{3\Phi}$};
\draw [shift={(0.18039870004873235,0)},line width=1pt,color=qqqqff]  plot[domain=1.5707963267948966:2.3055723661905114,variable=\t]({1*0.40338375636156015*cos(\t r)+0*0.40338375636156015*sin(\t r)},{0*0.40338375636156015*cos(\t r)+1*0.40338375636156015*sin(\t r)});
\draw [shift={(-0.18039870004873235,0)},line width=1pt,color=qqqqff]  plot[domain=0.8360202873992818:1.5707963267948966,variable=\t]({1*0.40484317796423347*cos(\t r)+0*0.40484317796423347*sin(\t r)},{0*0.40484317796423347*cos(\t r)+1*0.40484317796423347*sin(\t r)});
\draw [line width=1pt,color=qqqqff] (0.18039870004873235,0.40338375636156015) -- (0.18039870004873235,2.34877816657346);
\draw [line width=1pt,color=qqqqff] (-0.1803987000487323,0.40338375636156015) -- (-0.1803987000487323,2.34877816657346);
%\draw (0.1116750389585722,0.05522474621249833) node[anchor=north west] {$\frac{1}{3\Phi}$};
%\draw (0,1) node[anchor=south west] {$\SS$};
\draw (0.924,0.382) node[anchor=south west] {$\XX$};
%\draw [fill=black] (0,1) circle (1pt);
\draw [fill=black] (0.924,0.382) circle (1pt);
\end{tikzpicture}
\includegraphics[width=5cm]{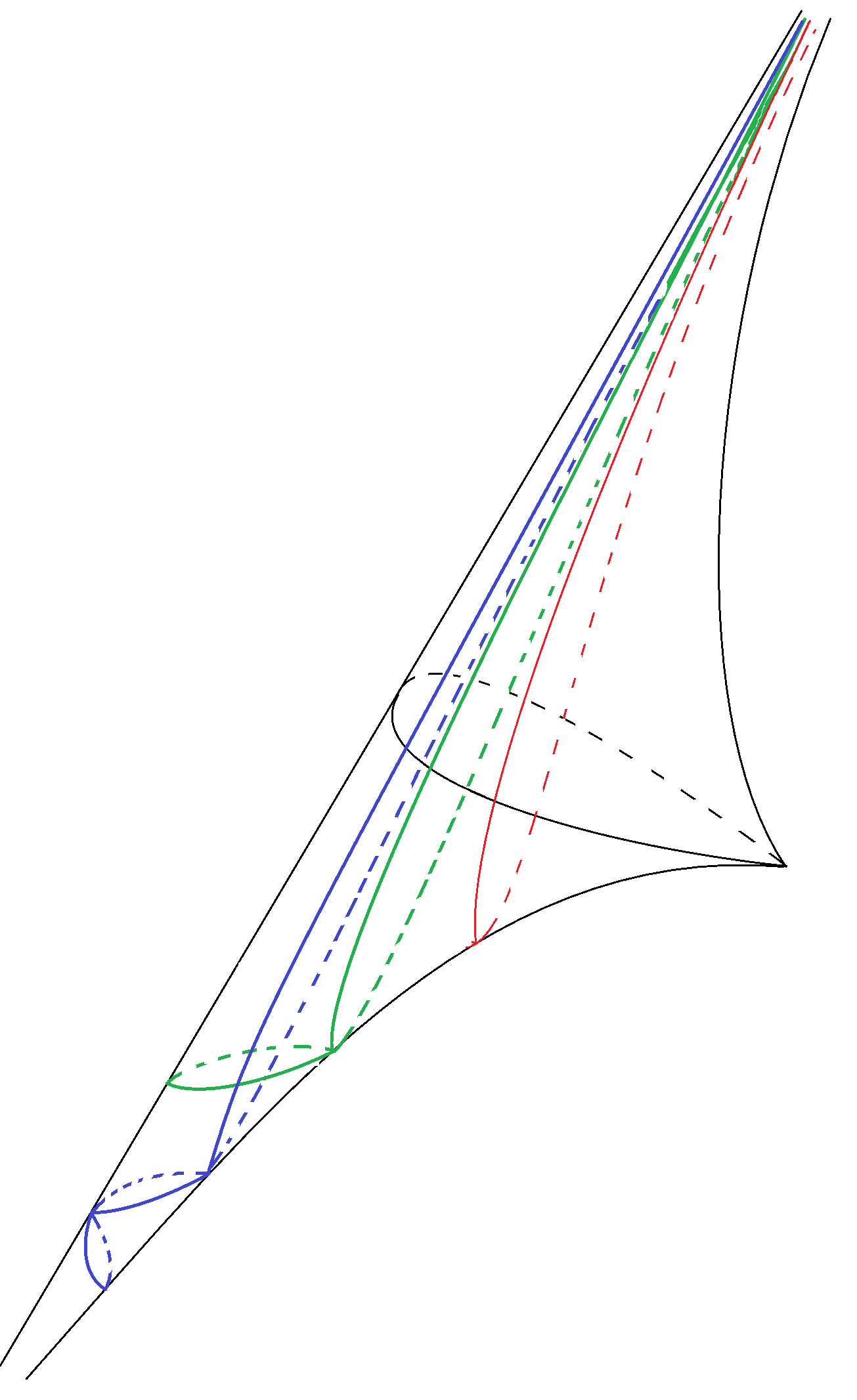}
\caption{The geodesics $\gamma_{\infty, \frac{1}{k\Phi}}$ for $k=1,2,3$ and their images by the Veech group intersecting the fundamental domain $\TT_n$. On the right, the same geodesics on the surface $\HH / \Gamma_n$.}
\label{geodesics_kPhi}
\end{figure}

Specifically when $X = \XX$ is the regular $n$-gon, the result of Theorem \ref{theo:main} can be stated as follows\footnote{Although stated in this order, we first prove Theorem \ref{theo:KVol_4m} and then use it to prove Theorem \ref{theo:main}}:

\begin{Theo}\label{theo:KVol_4m}
Let $n \geq 8$, $n \equiv 0 \mod 4$. Let $l_0$ be the length of the side of the $n$-gon\footnote{Which is also the systolic length of the resulting surface $\XX$.}. For any pair of saddle connections $\alpha, \beta$, we have 
\begin{equation}\label{K4m}
\frac{\mathrm{Int}(\alpha,\beta)}{l(\alpha)l(\beta)} \leq \frac{1}{l_0^2}.
\end{equation}
Moreover, equality is achieved if and only if $\alpha$ and $\beta$ are distinct sides of the $n$-gon.
\end{Theo}

Using that the volume of a regular $n$-gon of unit side is $\frac{n}{4 \tan \frac{\pi}{n}}$, we get:
\begin{Cor}
For any $n \geq 8$ such that $n \equiv 0 \mod 4$, we have:
\[
\KVol(\XX) = \frac{n}{4 \tan(\frac{\pi}{n})}.
\]
\end{Cor}

Notice that although $\XX$ has minimal KVol in its Teichm\"uller disk, it is not a local minimum for KVol in its stratum $\mathcal{H} (2g-2)$. Indeed, the regular $n$-gon is the polygon with $n$ sides of the same length that have maximal volume. In particular, any other such polygon close to the regular $n$-gon will have a smaller volume, and KVol will still be realized by pairs of sides of the corresponding $n$-gon, hence will be smaller.

\paragraph{The case $n \equiv 2 \mod 4$.}
If $n \equiv 2 \mod 4$, the resulting translation surface $\XX$ has two conical singularities, so that saddle connections may not be closed curves anymore, and simple closed geodesics could be homologous to the union of several non-closed saddle connections in different directions. For this reason, we do not have a closed formula for KVol in the Teichm\"uller disk of the regular $n$-gon. However, we show:

\begin{Theo}\label{theo:main2}
For $n \geq 10$ with $n \equiv 2 \mod 4$, KVol is bounded on the Teichm\"uller disk of the regular $n$-gon.
\end{Theo}

An explicit bound is given in Corollary \ref{cor:boundedness}. This result contrasts with examples of squared tiled translation surfaces with multiple singularities having unbounded KVol on their Teichm\"uller disk. It is the first example of a translation surface with more than one singularity where we can show boundedness on the Teichm\"uller disk. In fact, we give an explicit boundedness criterion in the Teichm\"uller disk of a Veech surface:

\begin{Theo}\label{theo:boundedness_criterion_0}
$\KVol$ is bounded on the Teichm\"uller disk of a Veech surface $X$ if and only if there are no intersecting closed curves $\eta$ and $\xi$ on $X$ such that $\eta = \eta_1 \cup \cdots \cup \eta_k$ and $\xi = \xi_1 \cup \cdots \cup \xi_l$ are unions of parallel saddle connections (that is all saddle connections $\eta_1, \dots , \eta_k,\xi_1, \dots, \xi_l$ have the same direction).
\end{Theo}
This criterion generalises Proposition 3.2 of \cite{BLM22} to the case of translation surfaces with several singularities.\newline

Finally, concerning the regular $n$-gon itself, the proof of Theorem \ref{theo:KVol_4m} extends and gives:
\begin{Theo}\label{theo:4m+2}
Let $n \geq 10$ with $n \equiv 2 \mod 4$. Let $l_0$ be the length of the side of the $n$-gon. For any pair of closed curves $\alpha, \beta$ on $\XX$, we have 
\begin{equation*}\label{K4m+2}
\frac{Int(\alpha,\beta)}{l(\alpha)l(\beta)} < \frac{1}{l_0^2}.
\end{equation*}
\end{Theo}
Notice that in this case, we show that the inequality is strict. In fact, it is unclear whether the best possible ratio is achieved and, if so, by which closed curves. Our guess would be the following:

\begin{Conj}
Let $n \geq 10$ with $n \equiv 2 \mod 4$. For any pair of closed curves $\alpha, \beta$ on $\XX$, we have
\begin{equation*}\label{K4m+2}
\frac{Int(\alpha,\beta)}{l(\alpha)l(\beta)} \leq \frac{1}{2 l_0^2},
\end{equation*}
where $l_0$ is the length of the side of the $n$-gon. Moreover, equality is achieved if and only if $\alpha$ and $\beta$ are twice intersecting pairs of sides of the $n$-gon, as in Figure \ref{fig:example_10_gon} for the decagon.
In particular:
\[
\KVol(\XX) = \frac{n}{8 \tan(\frac{\pi}{n})}
\]
\end{Conj}

\begin{figure}
\center
\definecolor{ccqqqq}{rgb}{0.8,0,0}
\definecolor{qqwuqq}{rgb}{0,0.39215686274509803,0}
\definecolor{uuuuuu}{rgb}{0.26666666666666666,0.26666666666666666,0.26666666666666666}
\begin{tikzpicture}[line cap=round,line join=round,>=triangle 45,x=1cm,y=1cm,scale=1.2]
\clip(-2.06666666666667,-0.1244444444444426) rectangle (3.737777777777777,3.248888888888885);
\draw [line width=1pt,color=qqwuqq] (0,0)-- (1,0);
\draw [line width=1pt,color=ccqqqq] (1,0)-- (1.809016994374947,0.5877852522924729);
\draw [line width=1pt] (1.809016994374947,0.5877852522924729)-- (2.118033988749895,1.5388417685876261);
\draw [line width=1pt] (2.118033988749895,1.5388417685876261)-- (1.8090169943749475,2.4898982848827793);
\draw [line width=1pt] (1.8090169943749475,2.4898982848827793)-- (1,3.0776835371752527);
\draw [line width=1pt] (1,3.0776835371752527)-- (0,3.0776835371752527);
\draw [line width=1pt] (0,3.0776835371752527)-- (-0.809016994374947,2.4898982848827798);
\draw [line width=1pt,color=qqwuqq] (-0.809016994374947,2.4898982848827798)-- (-1.1180339887498945,1.5388417685876268);
\draw [line width=1pt] (-1.1180339887498945,1.5388417685876268)-- (-0.8090169943749475,0.5877852522924734);
\draw [line width=1pt,color=ccqqqq] (-0.8090169943749475,0.5877852522924734)-- (0,0);
\draw [color=qqwuqq](0.25,0.4) node[anchor=north west] {$\alpha$};
\draw [color=ccqqqq](-0.9,0.3) node[anchor=north west] {$\beta$};
\draw [color=ccqqqq](1.4,0.3) node[anchor=north west] {$\beta$};
\draw [color=qqwuqq](-1.4,2.25) node[anchor=north west] {$\alpha$};
\begin{scriptsize}
\draw [fill=uuuuuu] (0,0) circle (2.5pt);
\draw [color=black] (1,0)-- ++(-2.5pt,-2.5pt) -- ++(5pt,5pt) ++(-5pt,0) -- ++(5pt,-5pt);
\draw [fill=uuuuuu] (1.809016994374947,0.5877852522924729) circle (2.5pt);
\draw [color=black] (2.118033988749895,1.5388417685876261)-- ++(-2.5pt,-2.5pt) -- ++(5pt,5pt) ++(-5pt,0) -- ++(5pt,-5pt);
\draw [fill=uuuuuu] (1.8090169943749475,2.4898982848827793) circle (2.5pt);
\draw [color=black] (1,3.0776835371752527)-- ++(-2.5pt,-2.5pt) -- ++(5pt,5pt) ++(-5pt,0) -- ++(5pt,-5pt);
\draw [fill=uuuuuu] (0,3.0776835371752527) circle (2.5pt);
\draw [color=black] (-0.809016994374947,2.4898982848827798)-- ++(-2.5pt,-2.5pt) -- ++(5pt,5pt) ++(-5pt,0) -- ++(5pt,-5pt);
\draw [fill=uuuuuu] (-1.1180339887498945,1.5388417685876268) circle (2.5pt);
\draw [color=black] (-0.8090169943749475,0.5877852522924734)-- ++(-2.5pt,-2.5pt) -- ++(5pt,5pt) ++(-5pt,0) -- ++(5pt,-5pt);
\end{scriptsize}
\end{tikzpicture}
\caption{In this example, the curves $\alpha$ and $\beta$ (oriented such that the resulting curve has a well defined orientation) intersect twice.}
\label{fig:example_10_gon}
\end{figure}
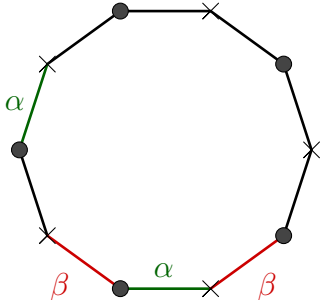

\paragraph{Strategy of proof.}
Although stated in the reversed order, the proof of Theorem \ref{theo:main} rely on Theorem \ref{theo:KVol_4m}. The proof of both Theorem \ref{theo:KVol_4m} and Theorem \ref{theo:4m+2} uses a subdivision method developped in \cite{BLM22} in the case of the double $n$-gon for odd $n$, which we extend to the case of the regular $n$-gons for even $n$: given two saddle connections, we can estimate simulteneously their lengths and their intersection by subdivising each saddle connection into smaller well-chosen segments. \newline

Then, we prove Theorem \ref{theo:main} by studying how the ratio $\frac{\Int(\alpha,\beta)}{l(\alpha) l(\beta)}$ varies as the surface $X$ and the saddle connections $\alpha$ and $\beta$  moves under the action of a matrix of $SL_2(\RR)$. The proof of both Theorem \ref{theo:main} and Theorem \ref{theo:main2} uses the fact that there are only two cylinder decompositions up to the action of $SL_2(\RR)$ on each surface of the Teichm\"uller disk of the regular $n$-gon. 

Further, although Theorem \ref{theo:main} is similar to Theorem 1.1 of \cite{BLM22} stated for the double $n$-gon when $n$ is odd, the main difference is that the maximum of KVol on the Teichm\"uller disk of the regular $n$-gon is achieved along an infinite set of geodesics instead of a single geodesic for the Teichm\"uller disk of the double regular $n$-gon. This has to do with the fact that the staircase model associated with the double regular $n$-gon has a pair of intersecting systoles, giving a big KVol, while there is only one systole for the staircase model associated with the regular $n$-gon. In particular, the supremum in the definition of KVol on the staircase model associated with the regular $n$-gon ($n \geq 8$, $n \equiv 0 \mod 4$) is achieved when $\alpha$ and $\beta$ are respectively the systole and the second shortest closed curve (which is perpendicular to the systole, see Figure \ref{staircase_model_0}), but it is also realised as the limit when $k$ goes to infinity of the ratio $\frac{\Int(\alpha,\beta_k)}{l(\alpha)l(\beta_k)}$, where $\alpha$ is the systole and $\beta_k$ is a saddle connection winding $k$ times around the smallest vertical cylinder of $\SS$ (and hence intersecting $k+1$ times $\alpha$, counting one singular intersection).

\paragraph{Organization of the paper.} We recall in Section \ref{sec:background} useful results about translation surfaces and their Veech groups, and we describe the staircase model associated with the regular $n$-gon. In Section \ref{sec:4m}, we compute explicitly KVol on the regular $n$-gon using elementary geometry, showing both Theorem \ref{theo:KVol_4m} and Theorem \ref{theo:4m+2}. Next, we provide in Section \ref{sec:directions} %Proposition \ref{prop:etude_K}
key estimates that allows to understand the function KVol on the Teichm\"uller disk. Using the knowledge of KVol on the regular $n$-gon ($n \equiv 0 \mod 4$), we prove Theorem \ref{theo:main} in Section \ref{sec:extension_teichmuller}. Finally, we show the boundedness criterion for KVol on Teichm\"uller disks (Theorem \ref{theo:boundedness_criterion_0}) in Section \ref{sec:4m+2} and we apply this criterion to the regular $n$-gon ($n \equiv 2 \mod 4$) in order to deduce Theorem \ref{theo:main2}.

\paragraph{Ackowledgements.}
I would like to thank Erwan Lanneau and Daniel Massart for their constant support and for enlightening discussions, as well as helpful comments on preliminary versions of this document.

\section{Background}\label{sec:background}
\subsection{Translation surfaces and their Veech groups}
We start with a quick review of useful notions. We encourage the reader to check out the surveys \cite{survey_Zorich}, \cite{survey_wright} and \cite{survey_massart} for a general introduction to translation surfaces.

A \emph{translation surface} $(X,\omega)$ is a real compact genus $g$ surface $X$ with an atlas $\omega$ such that all transition functions are translations except on a finite set of singularities $\Sigma$, along with a distinguished direction. In fact, it can also be seen as a surface obtained from a finite collection of polygons embedded in $\CC$ by gluing pairs of parallel opposite sides by translation. The resulting surface has a flat metric and a finite number of conical singularities. With this description, the moduli space of translation surfaces can be thought of as the space of all translation surfaces up to cut and paste.

The action of $GL_2^+(\RR)$ on polygons induces an action on the moduli space of translation surfaces. The orbit of a given translation surface is called its \emph{Teichm\"uller disk} and its stabilizer is called the \emph{Veech group} and is often denoted $SL(X)$. W.A.~Veech showed in \cite{Veech} that the latter are discrete subgroups of $\SL$. In particular, the Teichm\"uller disk of a translation surface can be identified with $\HH / SL(X)$.

\subsection{The regular $n$-gon and its staircase model}
Given $n \geq 4$ even, one can contruct a translation surface by identifying parallel opposite sides of a regular $n$-gon. If $n \equiv 0 \mod 4$ (and $n \geq 8$), the resulting surface has a single singularity, and in particular, every edge corresponds to a closed curve on the surface. However, if $n \equiv 2$ mod $4$ (and $n \geq 10$), the resulting surface has two singularities, so that sides are no longer closed curves.

As described in \cite{Monteil} and depicted in Figures \ref{cut_and_paste} and \ref{staircase_model_0}, the regular $n$-gon $\XX$ has a staircase shaped model $\SS$ in its Teichm\"uller disk.

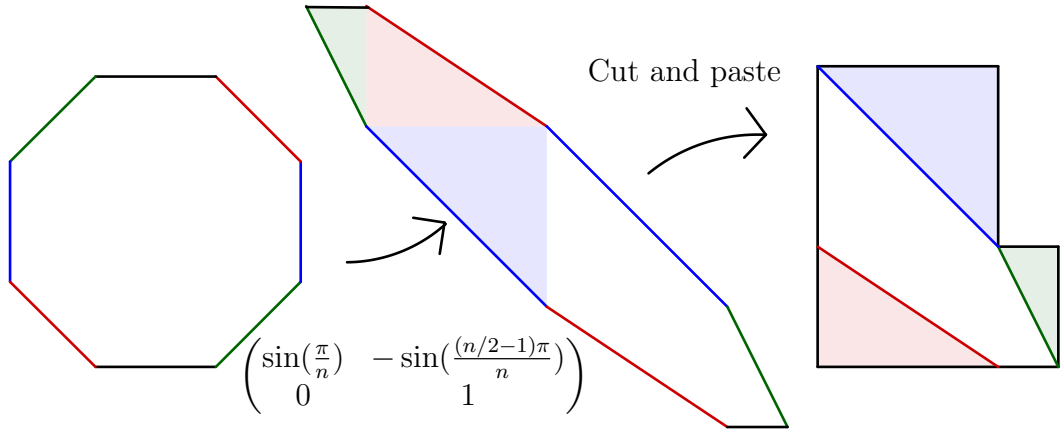
\begin{figure}[h]
\center
\definecolor{ccqqqq}{rgb}{0.8,0,0}
\definecolor{qqqqff}{rgb}{0,0,1}
\definecolor{qqwuqq}{rgb}{0,0.39215686274509803,0}
\begin{tikzpicture}[line cap=round,line join=round,>=triangle 45,x=1cm,y=1cm,scale=0.8]
\clip(-9.2,-1.1) rectangle (8.6,6.2);
\fill[line width=1pt,color=qqqqff,fill=qqqqff,fill opacity=0.1] (4.5,5) -- (7.5,5) -- (7.5,2) -- cycle;
\fill[line width=1pt,color=qqwuqq,fill=qqwuqq,fill opacity=0.1] (7.5,2) -- (8.5,2) -- (8.5,0) -- cycle;
\fill[line width=1pt,color=ccqqqq,fill=ccqqqq,fill opacity=0.1] (4.5,2) -- (4.5,0) -- (7.5,0) -- cycle;
\fill[line width=1pt,color=ccqqqq,fill=ccqqqq,fill opacity=0.1] (-3,4) -- (-3,6) -- (0,4) -- cycle;
\fill[line width=1pt,color=qqwuqq,fill=qqwuqq,fill opacity=0.1] (-4,6) -- (-3,6) -- (-3,4) -- cycle;
\fill[line width=1pt,color=qqqqff,fill=qqqqff,fill opacity=0.1] (-3,4) -- (0,4) -- (0,1) -- cycle;
\draw [line width=1pt] (4.5,0)-- (8.5,0);
\draw [line width=1pt] (8.5,0)-- (8.5,2);
\draw [line width=1pt] (8.5,2)-- (7.5,2);
\draw [line width=1pt] (7.5,2)-- (7.5,5);
\draw [line width=1pt] (7.5,5)-- (4.5,5);
\draw [line width=1pt] (4.5,5)-- (4.5,0);
\draw [line width=1pt,color=qqwuqq] (7.5,2)-- (8.5,0);
\draw [line width=1pt,color=qqqqff] (7.5,2)-- (4.5,5);
\draw [line width=1pt,color=ccqqqq] (7.5,0)-- (4.5,2);
\draw [line width=1pt,color=ccqqqq] (-3,6)-- (0,4);
\draw [line width=1pt,color=qqwuqq] (-3,4)-- (-4,6);
\draw [line width=1pt,color=qqqqff] (0,1)-- (-3,4);
\draw [line width=1pt,color=qqqqff] (0,4)-- (3,1);
\draw [line width=1pt,color=ccqqqq] (0,1)-- (3,-1);
\draw [line width=1pt,color=qqwuqq] (3,1)-- (4,-1);
\draw [line width=1pt] (3,-1)-- (4,-1);
\draw [line width=1pt] (-3.992,5.984)-- (-2.963076923076923,5.975384615384615);
\draw [line width=1pt] (-7.5,0)-- (-5.5,0);
\draw [line width=1pt,color=qqwuqq] (-5.5,0)-- (-4.09,1.414213562373095);
\draw [line width=1pt,color=qqqqff] (-4.09,1.414213562373095)-- (-4.09,3.4142135623730945);
\draw [line width=1pt,color=ccqqqq] (-4.09,3.4142135623730945)-- (-5.5,4.82842712474619);
\draw [line width=1pt] (-5.5,4.82842712474619)-- (-7.5,4.82842712474619);
\draw [line width=1pt,color=qqwuqq] (-7.5,4.82842712474619)-- (-8.92,3.4142135623730954);
\draw [line width=1pt,color=qqqqff] (-8.92,3.4142135623730954)-- (-8.92,1.4142135623730956);
\draw [line width=1pt,color=ccqqqq] (-8.92,1.4142135623730956)-- (-7.5,0);
\draw [shift={(-3.34,4.22)},line width=1pt]  plot[domain=4.720453321691457:5.445810878213226,variable=\t]({1*2.480080643850115*cos(\t r)+0*2.480080643850115*sin(\t r)},{0*2.480080643850115*cos(\t r)+1*2.480080643850115*sin(\t r)});
\draw [shift={(3.5,1)},line width=1pt]  plot[domain=1.521884320645702:2.2550911184560465,variable=\t]({1*2.863424523188973*cos(\t r)+0*2.863424523188973*sin(\t r)},{0*2.863424523188973*cos(\t r)+1*2.863424523188973*sin(\t r)});
\draw [line width=1pt] (3.64,3.86)-- (3.28,3.48);
\draw [line width=1pt] (3.64,3.86)-- (3.32,4.3);
\draw [line width=1pt] (-1.7,2.4)-- (-2.22,2.48);
\draw [line width=1pt] (-1.7,2.4)-- (-1.78,1.88);
\draw (-5.3,0.7) node[anchor=north west] {$\begin{pmatrix} \sin(\frac{\pi}{n}) & -\sin(\frac{(n/2-1)\pi}{n}) \\ 0 & 1 \end{pmatrix}$};
\draw (0.5,5.3) node[anchor=north west]  {Cut and paste};
\end{tikzpicture}
\caption{From the regular octagon to its staircase model.}
\label{cut_and_paste}
\end{figure}

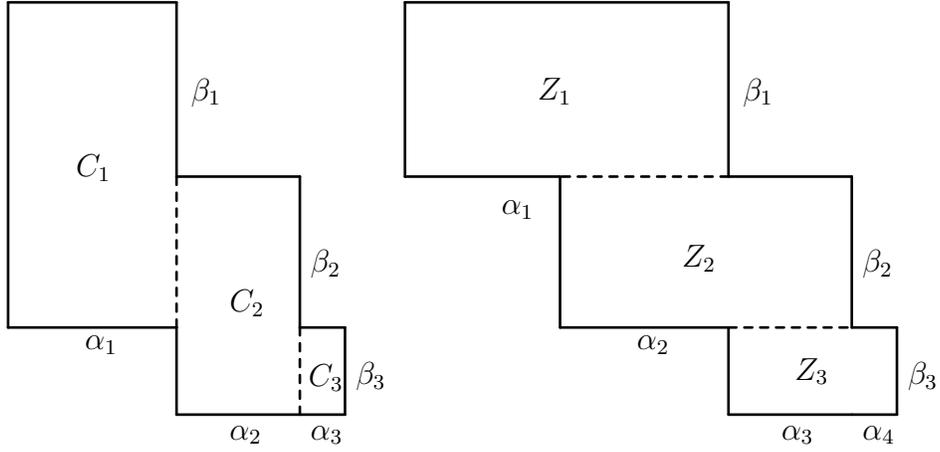
\begin{figure}[h]
\center
%\begin{tikzpicture}[line cap=round,line join=round,>=triangle 45,x=1cm,y=1cm, scale=0.9]
%\clip(-3.5,-1) rectangle (2,5);
%\draw [line width=1pt] (0,0)-- (1,0);
%\draw [line width=1pt] (1,0)-- (1,1.8477590650225735);
%\draw [line width=1pt] (1,1.8477590650225735)-- (0,1.8477590650225735);
%\draw [line width=1pt] (0,1.8477590650225735)-- (0,4.460884994775325);
%\draw [line width=1pt] (0,4.460884994775325)-- (-2.414213562373095,4.460884994775325);
%\draw [line width=1pt] (-2.414213562373095,4.460884994775325)-- (-2.414213562373095,0);
%\draw [line width=1pt] (-2.414213562373095,0)-- (0,0);
%\draw [line width=1pt,dash pattern=on 3pt off 3pt] (0,0)-- (0,1.8477590650225735);
%\draw [line width=1pt,dash pattern=on 3pt off 10pt] (-2.414213562373095,1.8477590650225735)-- (0,1.8477590650225735);
%\draw (0.1,0) node[anchor=north west] {$\alpha_2$};
%\draw (-1.5,0) node[anchor=north west] {$\alpha_1$};
%\draw (1,1.3) node[anchor=north west] {$\beta_2$};
%\draw (0,3.5) node[anchor=north west] {$\beta_1$};
%\draw (0.1,1.3) node[anchor=north west] {$C_2$};
%\draw (-1.6,2.3) node[anchor=north west] {$C_1$};
%\end{tikzpicture}
\begin{tikzpicture}[line cap=round,line join=round,>=triangle 45,x=1cm,y=1cm, scale = 0.6]
\clip(-7,-1) rectangle (2,10);
\draw [line width=1pt] (0,0)-- (1,0);
\draw [line width=1pt] (1,0)-- (1,1.9318516525781366);
\draw [line width=1pt] (1,1.9318516525781366)-- (0,1.9318516525781366);
\draw [line width=1pt] (0,1.9318516525781366)-- (0,5.277916867529369);
\draw [line width=1pt] (0,5.277916867529369)-- (-2.7320508075688776,5.277916867529369);
\draw [line width=1pt] (-2.7320508075688776,0)-- (0,0);
\draw [line width=1pt,dash pattern=on 3pt off 3pt] (0,0)-- (0,1.9318516525781366);
\draw (-1.8,0) node[anchor=north west] {$\alpha_2$};
\draw (-5,2) node[anchor=north west] {$\alpha_1$};
\draw (0,3.9) node[anchor=north west] {$\beta_2$};
\draw (-2.65,7.7) node[anchor=north west] {$\beta_1$};
%\draw [line width=1pt,dash pattern=on 3pt off 10pt] (0,1.9318516525781366)-- (-2.7320508075688776,1.9318516525781368);
\draw [line width=1pt] (-2.7320508075688776,1.9318516525781368)-- (-6.464101615137755,1.9318516525781366);
\draw [line width=1pt] (-6.464101615137755,1.9318516525781366)-- (-6.464101615137755,9.141620172685643);
\draw [line width=1pt] (-6.464101615137755,9.141620172685643)-- (-2.7320508075688767,9.141620172685643);
\draw [line width=1pt] (-2.7320508075688767,9.141620172685643)-- (-2.7320508075688776,5.277916867529369);
\draw [line width=1pt,dash pattern=on 3pt off 3pt] (-2.7320508075688776,1.9318516525781368)-- (-2.7320508075688776,5.277916867529369);
\draw [line width=1pt] (-2.7320508075688776,1.9318516525781368)-- (-2.7320508075688776,0);
%\draw [line width=1pt,dash pattern=on 3pt off 10pt] (-6.464101615137756,5.277916867529369)-- (-2.7320508075688776,5.277916867529369);
\draw (0,0) node[anchor=north west] {$\alpha_3$};
\draw (1,1.4) node[anchor=north west] {$\beta_3$};
\draw (-0.05,1.35) node[anchor=north west] {$C_3$};
\draw (-1.8,3) node[anchor=north west] {$C_2$};
\draw (-5.2,6) node[anchor=north west] {$C_1$};
\end{tikzpicture}
\begin{tikzpicture}[line cap=round,line join=round,>=triangle 45,x=1cm,y=1cm, scale = 0.6]
\clip(-10,-1) rectangle (2,10);
\draw [line width=1pt] (0,0)-- (1,0);
\draw [line width=1pt] (1,0)-- (1,1.9318516525781366);
\draw [line width=1pt] (1,1.9318516525781366)-- (0,1.9318516525781366);
\draw [line width=1pt] (0,1.9318516525781366)-- (0,5.277916867529369);
\draw [line width=1pt] (0,5.277916867529369)-- (-2.7320508075688776,5.277916867529369);
\draw [line width=1pt] (-2.7320508075688776,0)-- (0,0);
\draw [line width=1pt,dash pattern=on 3pt off 3pt] (-2.7320508075688776,5.277916867529369)-- (-6.464101615137755,5.277916867529369);
\draw (-1.8,0) node[anchor=north west] {$\alpha_3$};
\draw (-5,2) node[anchor=north west] {$\alpha_2$};
\draw (0,3.9) node[anchor=north west] {$\beta_2$};
\draw (-2.65,7.7) node[anchor=north west] {$\beta_1$};
%\draw [line width=1pt,dash pattern=on 3pt off 10pt] (0,1.9318516525781366)-- (-2.7320508075688776,1.9318516525781368);
\draw [line width=1pt] (-2.7320508075688776,1.9318516525781368)-- (-6.464101615137755,1.9318516525781366);
\draw [line width=1pt] (-6.464101615137755,1.9318516525781366)-- (-6.464101615137755,5.277916867529369);
\draw [line width=1pt] (-9.9,5.277916867529369)-- (-6.464101615137755,5.277916867529369);
\draw [line width=1pt] (-9.9,5.277916867529369)-- (-9.9,9.141620172685643);
\draw [line width=1pt] (-9.9,9.141620172685643)-- (-2.7320508075688767,9.141620172685643);
\draw [line width=1pt] (-2.7320508075688767,9.141620172685643)-- (-2.7320508075688776,5.277916867529369);
\draw [line width=1pt,dash pattern=on 3pt off 3pt] (-2.7320508075688776,1.9318516525781368)-- (0,1.9318516525781368);
\draw [line width=1pt] (-2.7320508075688776,1.9318516525781368)-- (-2.7320508075688776,0);
%\draw [line width=1pt,dash pattern=on 3pt off 10pt] (-6.464101615137756,5.277916867529369)-- (-2.7320508075688776,5.277916867529369);
\draw (0,0) node[anchor=north west] {$\alpha_4$};
\draw (1,1.4) node[anchor=north west] {$\beta_3$};
\draw (-1.5,1.5) node[anchor=north west] {$Z_3$};
\draw (-4,4) node[anchor=north west] {$Z_2$};
\draw (-7.2,7.7) node[anchor=north west] {$Z_1$};
\draw (-8,5) node[anchor=north west] {$\alpha_1$};
\end{tikzpicture}
\caption{The staircase model associated with the regular $n$-gon for $n=12$ on the left and $n=14$ on the right.}
\label{staircase_model_0}
\end{figure}

With the notations of Figure \ref{staircase_model_0}, the lengths of the sides are given by:

$$
l(\alpha_i) = \sin(\frac{(n/2-2i+1)\pi}{n}) \text{ and } l(\beta_j) = \sin(\frac{(n/2-2j+2)\pi}{n})
$$
Where $i,j \in \llbracket 1,\frac{n}{4} \rrbracket$ for $n \equiv 0 \mod 4$, and $i \in \llbracket 1,\frac{n-2}{4}+1 \rrbracket$ while $j \in \llbracket 1, \frac{n-2}{4} \rrbracket$ for $n \equiv 2 \mod 4$. In particular, the modulus of each vertical cylinder $C_i$ is $\Phi = 2\cos(\frac{\pi}{n})$ as well as the modulus of each horizontal cylinder $Z_j$, except the horizontal cylinder $Z_1$ if $n \equiv 0 \mod 4$ (resp. the vertical cylinder $C_1$ if $n \equiv 2 \mod 4$) having a modulus of $\Phi / 2$. \newline

In both cases, the Veech group $\Gamma_n$ associated with the staircase model contains the horizontal twist $T_H = \begin{pmatrix} 1 & \Phi \\ 0 & 1 \end{pmatrix}$ and the vertical twist $T_V = \begin{pmatrix} 1 & 0 \\ \Phi & 1 \end{pmatrix}$. In fact, it is shown in \cite{Veech} that $\Gamma_n$ is generated by $T_H$ and $T_V$. In particular, the Teichm\"uller disk of the regular $n$-gon, identified with $\HH / \Gamma_n$ has a fundamental domain $\TT$ as depicted in Figure \ref{Domaine_fondamental}.

It will sometimes be convenient to include orientation-reversing elements in the definition of the Veech group. The matrix $R = \begin{pmatrix} 1 & 0 \\ 0 & -1 \end{pmatrix}$ gives an affine orientation-reversing diffeomorphism of the staircase model, and in fact $T_H, T_V$ and $R$ generate the "non-oriented" Veech group, which we will denote by $\Gamma_n^{\pm}$.

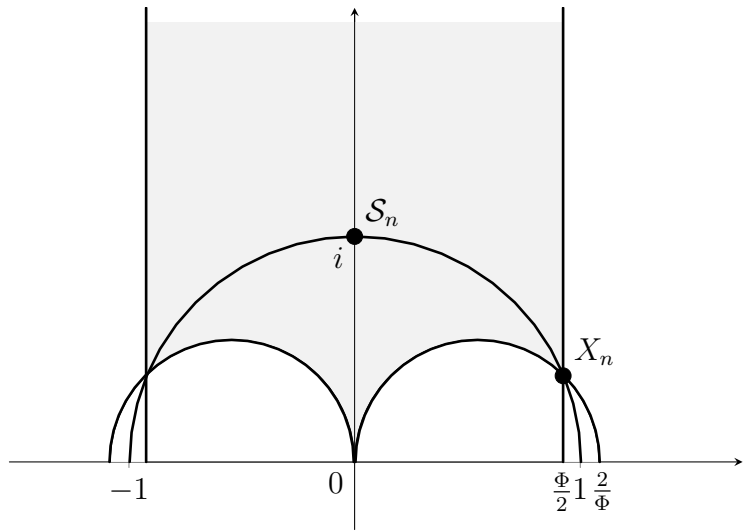
\begin{figure}[h]
\center
\begin{tikzpicture}[line cap=round,line join=round,>=triangle 45,x=3cm,y=3cm]
\begin{axis}[
x=3cm,y=3cm,
axis lines=middle,
xmin=-1.5299745127701394,
xmax=1.719501040204969,
ymin=-0.3,
ymax=2.013798818627857,
xtick={-1,0,1},
ytick={0},]
\clip(-1.5299745127701394,-0.3) rectangle (1.719501040204969,2.013798818627857);
\fill [gray!10](-0.923879532511,0) rectangle (0.923879532511,1.95);
\filldraw[fill=white] (1.0823922002,0) -- (1.0823922002,0) arc (0:180:0.54119610014) (0,0) -- cycle;
\filldraw[fill=white] (0,0) -- (0,0) arc (0:180:0.5411961001) (1.0823922002,0)  -- cycle;
\draw [line width=1pt] (-0.9238795325112867,0) -- (-0.9238795325112867,2.013798818627857);
\draw [line width=1pt] (0.9238795325112867,0) -- (0.9238795325112867,2.013798818627857);
\draw [line width=0.3pt] (-0.9238,0) -- (0.9238795325112867,0);
\draw [line width=0.3pt] (0,0) -- (0,2);
\draw [shift={(-2.075,-0.3)},line width=1pt]  plot[domain=0:3.141592653589793,variable=\t]({1*0.541196100146197*cos(\t r)+0*0.541196100146197*sin(\t r)},{0*0.541196100146197*cos(\t r)+1*0.541196100146197*sin(\t r)});
\draw [shift={(-0.985,-0.3)},line width=1pt]  plot[domain=0:3.141592653589793,variable=\t]({1*0.541196100146197*cos(\t r)+0*0.541196100146197*sin(\t r)},{0*0.541196100146197*cos(\t r)+1*0.541196100146197*sin(\t r)});
\draw [shift={(-1.527,-0.3)},line width=1pt]  plot[domain=0:3.141592653589793,variable=\t]({1*1*cos(\t r)+0*1*sin(\t r)},{0*1*cos(\t r)+1*1*sin(\t r)});
\draw (0.82,0) node[anchor=north west] {$\frac{\Phi}{2}$};
\draw (1,0) node[anchor=north west] {$\frac{2}{\Phi}$};
\draw (0,0) node[anchor=north east] {$0$};
\draw (0,1) node[anchor=north east] {$i$};
\draw (0,1) node[anchor=south west] {$\SS$};
\draw (0.924,0.382) node[anchor=south west] {$\XX$};
\draw [fill=black] (0,1) circle (3pt);
\draw [fill=black] (0.924,0.382) circle (3pt);
\end{axis}
\end{tikzpicture}
\caption{The fundamental domain of the Teichm\"uller disk of $\XX$.}
\label{Domaine_fondamental}
\end{figure}

\section{KVol on the regular $n$-gon, $n \equiv 0 \mod 4$.}\label{sec:4m}
In this section, we give an elementary proof of Theorem \ref{theo:KVol_4m} and Theorem \ref{theo:4m+2}. The main idea of the proof is to subdivide saddle connections $\alpha$ and $\beta$ into smaller (non-closed) segments where we can control both the lengths and the intersections.

Given saddle connections $\alpha$ and $\beta$, we start by defining a notion of sector for the direction of $\alpha$ (resp. $\beta$) which tells how to subdivide the saddle connection $\alpha$ (resp. $\beta$) into segments (\S\S 3.1 \& 3.2). Then, we study these segments separately and show that they are all longer than the side of the regular $n$-gon (\S 3.2). Finally, we study the possible intersection for each pair of segments (\S 3.3). Setting aside two particular cases that are studied separately, each pair of segments $(\alpha_i,\beta_j)$ intersect at most once, giving the desired ratio. To conclude the proof, it remains to take into account the possible singular intersections. This can be acomplished by paying closer attention to the intersections and the lengths.

\subsection{Sectors and separatrix diagrams}
The directions of the diagonals of the regular $n$-gon subdivide the set of possible directions into $n$ sectors of angle $\frac{\pi}{n}$, as in Figure \ref{secteurs_octogone} for the octagon.

In each sector there is an associated \emph{transition diagram} which encodes the possible sequence of intersections of sides for a line whose direction lies in this sector. For example, a geodesic with direction in the sector $\Sigma_0$ of Figure \ref{secteurs_octogone} has to intersect $e_2$ before and after each intersection with $e_1$ while it can intersect either $e_2$ or $e_3$ before and after each intersection with $e_0$. In this example, the sector $\Sigma_0$ gives the following transition diagram:
\begin{equation*}
 e_1 \leftrightarrows e_2 \leftrightarrows e_0\leftrightarrows e_3 \circlearrowleft
\end{equation*}

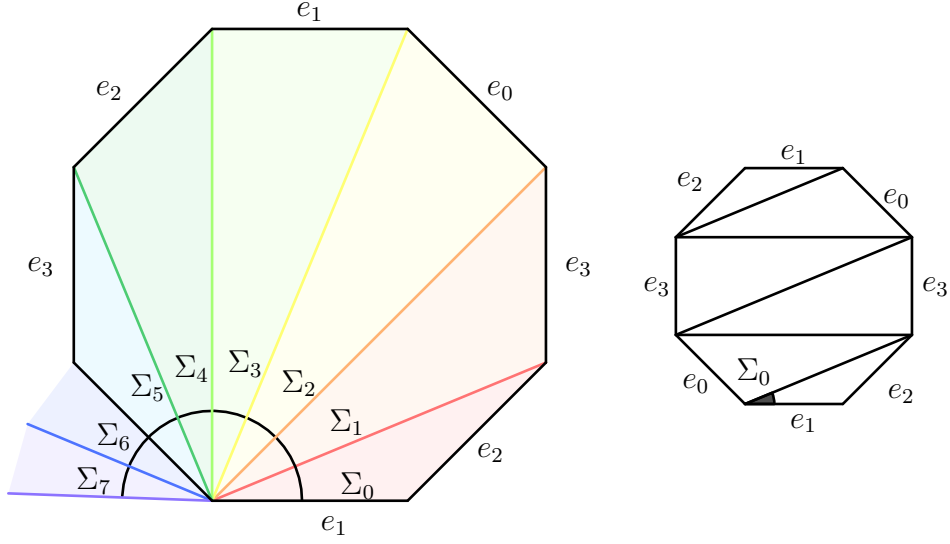
\begin{figure}[h]
\center
\definecolor{ffacyf}{rgb}{0.55,0.45,1} %Sigma_7
\definecolor{zdaqzb}{rgb}{0.3,0.45,1} %Sigma_6
\definecolor{yfeayc}{rgb}{0.3,0.8,1} %Sigma_5
\definecolor{zvbcff}{rgb}{0.3,0.8,0.45} %Sigma_4
\definecolor{drauff}{rgb}{0.65,1,0.45} %Sigma_3
\definecolor{ffzqfw}{rgb}{1,1,0.45} %Sigma_2
\definecolor{ffdcbu}{rgb}{1,0.7,0.45} %Sigma_1
\definecolor{euzsva}{rgb}{1,0.45,0.45} %Sigma_0
\definecolor{ccqqqq}{rgb}{0,0,0}
\definecolor{qqqqff}{rgb}{0,0,0}
\definecolor{ffvvqq}{rgb}{0,0,0}
\definecolor{qqwuqq}{rgb}{0,0,0}
\begin{tikzpicture}[line cap=round,line join=round,>=triangle 45,x=1.3cm,y=1.3cm]
\clip(-3.5,-0.5) rectangle (7,5.2);
\fill[line width=1pt,color=euzsva,fill=euzsva,fill opacity=0.1] (-1,0) -- (1,0) -- (2.4142135623730954,1.4142135623730951) -- cycle;
\fill[line width=1pt,color=ffdcbu,fill=ffdcbu,fill opacity=0.1] (-1,0) -- (2.4142135623730954,1.4142135623730951) -- (2.4142135623730954,3.4142135623730936) -- cycle;
\fill[line width=1pt,color=ffzqfw,fill=ffzqfw,fill opacity=0.1] (-1,0) -- (2.4142135623730954,3.4142135623730936) -- (1,4.828427124746191) -- cycle;
\fill[line width=1pt,color=drauff,fill=drauff,fill opacity=0.1] (-1,0) -- (1,4.828427124746191) -- (-1,4.828427124746191) -- cycle;
\fill[line width=1pt,color=zvbcff,fill=zvbcff,fill opacity=0.1] (-1,0) -- (-1,4.828427124746191) -- (-2.4142135623730954,3.4142135623730945) -- cycle;
\fill[line width=1pt,color=yfeayc,fill=yfeayc,fill opacity=0.1] (-1,0) -- (-2.4142135623730954,3.4142135623730945) -- (-2.414213562373096,1.4142135623730963) -- cycle;
\fill[line width=1pt,color=zdaqzb,fill=zdaqzb,fill opacity=0.1] (-1,0) -- (-2.414213562373096,1.4142135623730963) -- (-2.8864976604412904,0.7850695654254252) -- cycle;
\fill[line width=1pt,color=ffacyf,fill=ffacyf,fill opacity=0.1] (-2.8864976604412904,0.7850695654254252) -- (-1,0) -- (-3.083673377740735,0.07523698314742658) -- cycle;
\draw [line width=1pt,color=euzsva] (2.4142135623730954,1.4142135623730951)-- (-1,0);
\draw [shift={(-1,0)},line width=1pt]  plot[domain=0:3.088835092859723,variable=\t]({1*0.92*cos(\t r)+0*0.92*sin(\t r)},{0*0.92*cos(\t r)+1*0.92*sin(\t r)});
\draw (2.5,2.6) node[anchor=north west] {$e_3$};
\draw (1.6,0.7) node[anchor=north west] {$e_2$};
\draw (0,-0.05) node[anchor=north west] {$e_1$};
\draw (1.7,4.4) node[anchor=north west] {$e_0$};
\draw (0,4.8) node[above] {$e_1$};
\draw (-2.3,4.4) node[anchor=north west] {$e_2$};
\draw (-3,2.6) node[anchor=north west] {$e_3$};
\draw (0.2,0.4) node[anchor=north west] {$\Sigma_0$};
\draw [line width=1pt,color=ffdcbu] (2.4142135623730954,3.4142135623730936)-- (-1,0);
\draw [line width=1pt,color=ffzqfw] (1,4.828427124746191)-- (-1,0);
\draw [line width=1pt,color=drauff] (-1,4.828427124746191)-- (-1,0);
\draw [line width=1pt,color=zvbcff] (-1,0)-- (-2.4142135623730954,3.4142135623730945);
\draw [line width=1pt,color=zdaqzb] (-2.8864976604412904,0.7850695654254252)-- (-1,0);
\draw [line width=1pt,color=ffacyf] (-1,0)-- (-3.083673377740735,0.07523698314742658);
\draw (0.09,1.05) node[anchor=north west] {$\Sigma_1$};
\draw (-0.4,1.45) node[anchor=north west] {$\Sigma_2$};
\draw (-0.95,1.65) node[anchor=north west] {$\Sigma_3$};
\draw (-1.5,1.6) node[anchor=north west] {$\Sigma_4$};
\draw (-1.95,1.4) node[anchor=north west] {$\Sigma_5$};
\draw (-2.2949705085429573,0.9) node[anchor=north west] {$\Sigma_6$};
\draw (-2.5,0.45) node[anchor=north west] {$\Sigma_7$};
\draw [line width=1pt,color=qqwuqq] (-1,0)-- (1,0);
\draw [line width=1pt,color=ffvvqq] (1,0)-- (2.414213562373095,1.414213562373095);
\draw [line width=1pt,color=qqqqff] (2.414213562373095,1.414213562373095)-- (2.414213562373095,3.4142135623730945);
\draw [line width=1pt,color=ccqqqq] (2.414213562373095,3.4142135623730945)-- (1,4.82842712474619);
\draw [line width=1pt,color=qqwuqq] (1,4.82842712474619)-- (-1,4.82842712474619);
\draw [line width=1pt,color=ffvvqq] (-1,4.82842712474619)-- (-2.414213562373095,3.4142135623730954);
\draw [line width=1pt,color=qqqqff] (-2.414213562373095,3.4142135623730954)-- (-2.4142135623730954,1.4142135623730956);
\draw [line width=1pt,color=ccqqqq] (-2.4142135623730954,1.4142135623730956)-- (-1,0);
\draw [shift={(4.45,0.99)},line width=1pt,fill=black,fill opacity=0.7] (0,0) -- (0:0.3) arc (0:22.5:0.3) -- cycle;
\draw [line width=1pt] (4.45,0.99)-- (5.45,0.99);
\draw [line width=1pt] (5.45,0.99)-- (6.157106781186547,1.6971067811865472);
\draw [line width=1pt] (6.157106781186547,1.6971067811865472)-- (6.157106781186547,2.697106781186547);
\draw [line width=1pt] (6.157106781186547,2.697106781186547)-- (5.45,3.4042135623730942);
\draw [line width=1pt] (5.45,3.4042135623730942)-- (4.45,3.4042135623730947);
\draw [line width=1pt] (4.45,3.4042135623730947)-- (3.742893218813453,2.6971067811865472);
\draw [line width=1pt] (3.742893218813453,2.6971067811865472)-- (3.7428932188134527,1.6971067811865477);
\draw [line width=1pt] (3.7428932188134527,1.6971067811865477)-- (4.45,0.99);
\draw [line width=1pt] (4.45,0.99)-- (6.157106781186547,1.6971067811865472);
\draw [line width=1pt] (3.7428932188134527,1.6971067811865477)-- (6.157106781186547,2.697106781186547);
\draw [line width=1pt] (3.7428932188134527,1.6971067811865477)-- (6.157106781186547,1.6971067811865472);
\draw [line width=1pt] (3.742893218813453,2.6971067811865472)-- (6.157106781186547,2.697106781186547);
\draw [line width=1pt] (3.742893218813453,2.6971067811865472)-- (5.45,3.4042135623730942);
\draw (4.25,1.57) node[anchor=north west] {$\Sigma_0$};
\draw (5.8,1.3) node[anchor=north west] {$e_2$};
\draw (4.7,3.75) node[anchor=north west] {$e_1$};
\draw (5.75,3.3) node[anchor=north west] {$e_0$};
\draw (3.3,2.4) node[anchor=north west] {$e_3$};
\draw (3.64,3.47) node[anchor=north west] {$e_2$};
\draw (6.15,2.4) node[anchor=north west] {$e_3$};
\draw (3.7,1.4) node[anchor=north west] {$e_0$};
\draw (4.8,1) node[anchor=north west] {$e_1$};
\end{tikzpicture}
\caption{The directions of the diagonals of the octagon divide the set of directions into $8$ sectors. On the right, diagonals corresponding to the sector $\Sigma_0$.}
\label{secteurs_octogone}
\end{figure}

In this configuration we say that $e_1$ is \emph{sandwiched}. The side $e_1$ can only be preceded and followed by the (adjacent) side $e_2$. More generally, for a given sector $\Sigma$ on the regular $n$-gon we have a separatrix diagram of the form 
\begin{equation*}
 e_{\sigma(0)} \leftrightarrows e_{\sigma(1)} \leftrightarrows \cdots \leftrightarrows  e_{\sigma(n/2-2)} \leftrightarrows e_{\sigma(n/2-1)} \circlearrowleft
\end{equation*}
where $\sigma = \sigma_\Sigma \in \mathfrak{S}_{n/2}$ is a given permutation. We say that the side $e_{\sigma(0)}$ is sandwiched by the side $e_{\sigma(1)}$ in sector $\Sigma$. Note that it is sufficient to know $\sigma(0)$ and $\sigma(1)$ to tell the sector, as the sector is defined by the directions of the side $e_{\sigma(0)}$ and the diagonal $e_{\sigma(0)} + e_{\sigma(1)}$. See \cite{SU} for further details on transition diagrams of regular $n$-gons.

\subsection{Construction of the subdivision.}
Let $\alpha$ be an oriented saddle connections on the regular $n$-gon. Assume $\alpha$ is not a diagonal of the $n$-gon, so that it has a well defined sector $\Sigma_{\alpha}$ which corresponds to a transition diagram given by the permutation $\sigma_{\alpha}$. We cut $\alpha$  each time it intersects a non-sandwiched side of the $n$-gon. This gives a decomposition into non-closed segments $\alpha = \alpha_1 \cup \cdots \cup \alpha_k$ where each segment is either (see Figure~\ref{types_segments}):
\begin{itemize}
\item A non-sandwiched segment which goes from a side of the $n$-gon to another non-adjacent side of the $n$-gon.
\item A sandwiched segment, with extremities on the side $e' = e_{\sigma_{\alpha}(1)}$, intersecting the sandwiched side $e = e_{\sigma_{\alpha}(0)}$ on its interior. Such segments are made of one piece going from $e'$ to $e$ and another piece going from $e$ to $e'$. We say that such a sandwiched segment has type $e' \to e \to e'$.
\item An initial or terminal segment $\alpha_1$ or $\alpha_k$. Such segments will be considered as non-sandwiched segments.
\end{itemize}

If $\alpha$ is a diagonal or a side of the $n$-gon, we set $k=1$ and $\alpha = \alpha_1$.

\begin{figure}[h]
\center
\definecolor{qqwuqq}{rgb}{0,0.39215686274509803,0}
\definecolor{ccqqqq}{rgb}{0.8,0,0}
\begin{tikzpicture}[line cap=round,line join=round,>=triangle 45,x=1cm,y=1cm]
\clip(-3,-0.5) rectangle (3,5.3);
\draw [line width=1pt] (-1,0)-- (1,0);
\draw [line width=1pt] (1,0)-- (2.414213562373095,1.414213562373095);
\draw [line width=1pt] (2.414213562373095,1.414213562373095)-- (2.414213562373095,3.4142135623730945);
\draw [line width=1pt] (2.414213562373095,3.4142135623730945)-- (1,4.82842712474619);
\draw [line width=1pt] (1,4.82842712474619)-- (-1,4.82842712474619);
\draw [line width=1pt] (-1,4.82842712474619)-- (-2.414213562373095,3.4142135623730954);
\draw [line width=1pt] (-2.414213562373095,3.4142135623730954)-- (-2.4142135623730954,1.4142135623730956);
\draw [line width=1pt] (-2.4142135623730954,1.4142135623730956)-- (-1,0);
\draw [line width=1pt,dash pattern= on 3pt off 3pt, color=ccqqqq] (-1.43,0.43)-- (2.414213562373095,2.58);
\draw [line width=2pt,color=qqwuqq] (0,0)-- (1.34,0.34);
\draw [line width=2pt,color=qqwuqq] (0,4.828427124746191)-- (-1.3242135623730953,4.504213562373094);
\draw (-0.25,0) node[anchor=north west] {$e_1$};
\draw (1.7,0.7) node[anchor=north west] {$e_2$};
\draw (-0.4,5.3) node[anchor=north west] {$e_1$};
\draw (-2.4,4.5) node[anchor=north west] {$e_2$};
\draw (1.7,4.5) node[anchor=north west] {$e_0$};
\draw (2.4,2.5) node[anchor= north west] {$e_3$};
\draw (-2.4,0.7) node[anchor=north west] {$e_0$};
\draw (-3,2.5) node[anchor= north west] {$e_3$};
\end{tikzpicture}
\caption{Example of a non-sandwiched segment (in dashed red) and a sandwiched segment (in bold green). The sandwiched segment is of type $e_2 \to e_1 \to e_2$.}
\label{types_segments}
\end{figure}
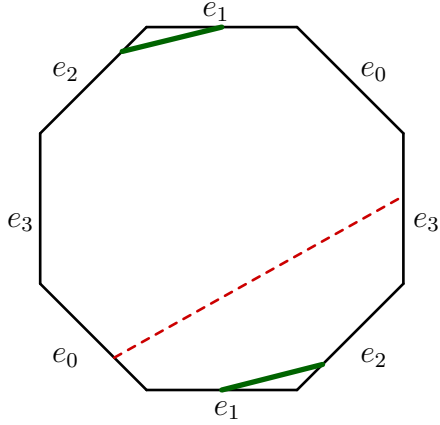

By construction, we have:
\begin{Lem}\label{Lemme_longueurs}
For every $i$, we have $l(\alpha_i) \geq l_0$ and equality holds if and only if $k=1$ and $\alpha$ is a side of the regular $n$-gon.
\end{Lem}
\begin{proof}
A non-sandwiched segment has length at least $l_0$ since each segment going from a side to another non-adjacent side has length at least $l_0$ in the regular $n$-gon. \newline
Now, take a sandwiched segment and assume up to a rotation that it is of type $e_2 \to e_1 \to e_2$, as in Figure \ref{types_segments}. Since angles between the sides $e_1$ and $e_2$ are obtuse, the sandwiched segment has a length bigger than the length of $e_1$, that is $l_0$.
\end{proof}

\subsection{Study of the intersections}
In this section, we investigate the possible intersections between two distinct saddle connections $\alpha$ and $\beta$, and show:

\begin{equation}\label{eq:intersection_general}
\frac{|\alpha \cap \beta|+1}{l(\alpha)l(\beta)} \leq \frac{1}{l_0^2}
\end{equation}
where $|\alpha \cap \beta|$ is the cardinal of the set of intersection points without counting the singularities. This set is finite as $\alpha$ and $\beta$ are distinct saddle connections.

As before, we decompose both $\alpha = \alpha_1 \cup \cdots \cup \alpha_k$ and $\beta = \beta_1 \cup \cdots \cup \beta_l$ into sandwiched and non-sandwiched segments. We start with the case where either $\alpha$ or $\beta$ are sides of the $n$-gon.

\paragraph{If $\alpha$ (resp. $\beta$) is a side of the $n$-gon.}
In this case, notice that:
\begin{enumerate}
\item Between each non-singular intersection with $\alpha$, there is at least one sandwiched segment or one non-sandwiched segment of $\beta$, giving a length greater than $l_0$.
\item Further, after the last non singular intersection with $\alpha$, there is a least one non-sandwiched segment, giving a length greater than $l_0$. 
\item If there are no non-singular intersection with $\alpha$, then $\beta$ has length at least $l_0$ anyway (with equality if and only if $\beta$ is another side of the regular $n$-gon).
\end{enumerate}
In particular, $ l(\beta) \geq (|\alpha \cap \beta|+1) l_0$ and hence:
\begin{equation*}
\frac{|\alpha \cap \beta|+1}{l(\alpha)l(\beta)} \leq \frac{1}{l_0^2}
\end{equation*}
with equality if and only if both $\alpha$ and $\beta$ are sides of the regular $n$-gon.

%\paragraph{The case of diagonals.} If both $\alpha$ and $\beta$ are diagonals of the $n$-gon, then Inequality \eqref{eq:intersection_general} holds with equality if and only if both $\alpha$ and $\beta$ are sides of the $n$-gon. Indeed:
%\begin{itemize}
%\item If $\alpha$ (or $\beta$) is a side, then $|\alpha \cap \beta|=0$ as there is no non singular intersection, and Inequality \eqref{eq:intersection_general} holds if and only if $l(\alpha) = l(\beta) = l_0$.
%\item Else, both $\alpha$ and $\beta$ are diagonals but not sides of the $n$-gon. Then, we use the fact that diagonals of the regular $n$-gon have length at least $\Phi l_0$ to obtain that $l(\alpha)l(\beta) \geq \Phi^2 l_0^2 > 2 l_0^2$ and, since $|\alpha \cap \beta| \leq 1$
%$$\frac{|\alpha \cap \beta| + 1}{l(\alpha)l(\beta)} \leq \frac{2}{\Phi^2 l_0^2} < \frac{1}{l_0^2}.$$
%\end{itemize}

\paragraph{The other cases.} In the rest of this section, we assume that $\alpha$ and $\beta$ are not both sides of the regular $n$-gon. As such, there is at least one segment $\alpha_i$ of $\alpha$ for which the inequaltity $l(\alpha_i) > l_0$ is strict, and hence $l(\alpha)l(\beta) > kl \cdot l_0^2$. Further, up to a small deformation of $\alpha$, we can assume there are no intersection on the sides of the $n$-gon (it is possible to do it while keeping each $\alpha_i$ straight and in the same sector). In this setting, all non-singular intersections between $\alpha$ and $\beta$ correspond to an intersection of two segments $\alpha_i$ and $\beta_j$ and we have:
\begin{equation*}
|\alpha \cap \beta| \leq \sum_{i,j} |\alpha_i \cap \beta_j|
\end{equation*}
where $\GIntij$ is the number of intersection points between the (non-closed) curves $\alpha_i$ and $\beta_j$. The latter quantity is finite as long as $\alpha$ and $\beta$ are assumed to be distinct. Notice that if $n \equiv 0 \mod 4$, there is a single singularity so that $\alpha$ and $\beta$ are automatically closed curves and we can define their algebraic intersection, and hence:
\begin{equation}\label{eq:intersections}
\Int(\alpha,\beta) \leq (\sum_{i,j} |\alpha_i \cap \beta_j|) +1
\end{equation}
where the added $+1$ corresponds to the possible singular intersection.\newline

Now, remark that if both $\alpha_i$ and $\beta_j$ are not sandwiched, then $\GIntij \leq 1$. More generally, we have:

\begin{Lem}\label{lem:cases_2_intersections}
For any $i,j$, $\GIntij \leq 2$. Further, the case $\GIntij = 2$ is possible only if:
\begin{enumerate}[label = (\roman*)]
\item Either $\alpha_i$ is a sandwiched segment of type $e' \to e \to e'$ and $\beta_j$ is a sandwiched segment of type $e \to e' \to e$.
\item Or, up to permutation of $\alpha$ and $\beta$, $\alpha_i$ is a sandwiched segment of type $e' \to e \to e'$, and $\beta_j$ is a long segment, that is the endpoints of $\beta_j$ lie in $e_{\sigma_{\beta} (2m-1)}$ or $e_{\sigma_{\beta}(2m-2)}$ or at the two vertices contained in both $e_{\sigma_{\beta} (n/2-1)}$ and $e_{\sigma_{\beta}(n/2-2)}$. Further, $\{e,e'\} = \{e_{\sigma_{\beta}(n/2-1)}, e_{\sigma_{\beta}(n/2-2)}\}$.
\end{enumerate}
These two configurations are depicted in Figure~\ref{fig:case_two_intersections}.
\end{Lem}

\begin{figure}[h]
\center
\definecolor{ccqqqq}{rgb}{0.8,0,0}
\definecolor{qqwuqq}{rgb}{0,0.39215686274509803,0}
\begin{tikzpicture}[line cap=round,line join=round,>=triangle 45,x=2cm,y=2cm]
\clip(-0.8,-0.2) rectangle (2,2.7);
\draw [line width=1pt] (0,0)-- (1,0);
\draw [line width=1pt] (1,0)-- (1.7071067811865475,0.7071067811865475);
\draw [line width=1pt] (1.7071067811865475,0.7071067811865475)-- (1.7071067811865475,1.7071067811865472);
\draw [line width=1pt] (1.7071067811865475,1.7071067811865472)-- (1,2.414213562373095);
\draw [line width=1pt] (1,2.414213562373095)-- (0,2.414213562373095);
\draw [line width=1pt] (0,2.414213562373095)-- (-0.7071067811865475,1.7071067811865477);
\draw [line width=1pt] (-0.7071067811865475,1.7071067811865477)-- (-0.7071067811865477,0.7071067811865478);
\draw [line width=1pt] (-0.7071067811865477,0.7071067811865478)-- (0,0);
\draw [line width=1pt,color=qqwuqq] (0.6266666666666643,0)-- (1.155555555555555,0.15555555555555506);
\draw [line width=1pt,color=qqwuqq] (-0.2611111111111077,2.153102451261988)-- (0.6266666666666645,2.414213562373095);
\draw [line width=1pt,color=ccqqqq] (0.8088888888888882,0)-- (1.3,0.3);
\draw [line width=1pt,color=ccqqqq] (0.4,2.4142)-- (-0.4,2.0142);
\begin{scriptsize}
\draw [color=qqwuqq] (0.6266666666666643,0)-- ++(-2.5pt,-2.5pt) -- ++(5pt,5pt) ++(-5pt,0) -- ++(5pt,-5pt);
\draw [fill=qqwuqq,shift={(1.155555555555555,0.15555555555555506)}] (0,0) ++(0 pt,3.75pt) -- ++(3.2475952641916446pt,-5.625pt)--++(-6.495190528383289pt,0 pt) -- ++(3.2475952641916446pt,5.625pt);
\draw [color=qqwuqq] (0.6266666666666645,2.414213562373095)-- ++(-2.5pt,-2.5pt) -- ++(5pt,5pt) ++(-5pt,0) -- ++(5pt,-5pt);
\draw [fill=qqwuqq] (-0.2611111111111077,2.153102451261988) circle (2.5pt);
\draw [fill=qqwuqq] (1.4459956700754397,0.44599567007544016) circle (2.5pt);
\draw [fill=ccqqqq] (-0.4,2.0142) circle (2.5pt);
\draw [fill=ccqqqq] (1.3,0.3) circle (2.5pt);
\draw [color=ccqqqq] (0.4,2.4142)-- ++(-2.5pt,-2.5pt) -- ++(5pt,5pt) ++(-5pt,0) -- ++(5pt,-5pt);
\draw [fill=ccqqqq,shift={(0.8088888888888882,0)}] (0,0) ++(0 pt,3.75pt) -- ++(3.2475952641916446pt,-5.625pt)--++(-6.495190528383289pt,0 pt) -- ++(3.2475952641916446pt,5.625pt);
\end{scriptsize}
\draw [color=qqwuqq](0.1,2.3) node[anchor=north west] {$\alpha_i$};
\draw [color=qqwuqq](0.6,0.4) node[anchor=north west] {$\alpha_i$};
\draw [color=ccqqqq](1,0.6) node[anchor=north west] {$\beta_j$};
\draw [color=ccqqqq](-0.2,2.2) node[anchor=north west] {$\beta_j$};
\draw (-0.6,2.3) node[anchor=north west] {$e'$};
\draw (0.35,2.65) node[anchor=north west] {$e$};
\draw (0.35,0) node[anchor=north west] {$e$};
\draw (1.35,0.5) node[anchor=north west] {$e'$};
\end{tikzpicture}
\begin{tikzpicture}[line cap=round,line join=round,>=triangle 45,x=2cm,y=2cm]
\clip(-1,-0.2) rectangle (1.8,2.7);
\draw [line width=1pt] (0,0)-- (1,0);
\draw [line width=1pt] (1,0)-- (1.7071067811865475,0.7071067811865475);
\draw [line width=1pt] (1.7071067811865475,0.7071067811865475)-- (1.7071067811865475,1.7071067811865472);
\draw [line width=1pt] (1.7071067811865475,1.7071067811865472)-- (1,2.414213562373095);
\draw [line width=1pt] (1,2.414213562373095)-- (0,2.414213562373095);
\draw [line width=1pt] (0,2.414213562373095)-- (-0.7071067811865475,1.7071067811865477);
\draw [line width=1pt] (-0.7071067811865475,1.7071067811865477)-- (-0.7071067811865477,0.7071067811865478);
\draw [line width=1pt] (-0.7071067811865477,0.7071067811865478)-- (0,0);
\draw [line width=1pt,color=qqwuqq] (0.6266666666666643,0)-- (1.155555555555555,0.15555555555555506);
\draw [line width=1pt,color=qqwuqq] (-0.2611111111111077,2.153102451261988)-- (0.6266666666666645,2.414213562373095);
\draw [line width=1pt,color=ccqqqq] (0.8088888888888882,0)-- (0.235555555555555,2.414213562373095);
\begin{scriptsize}
\draw [color=qqwuqq] (0.6266666666666643,0)-- ++(-2.5pt,-2.5pt) -- ++(5pt,5pt) ++(-5pt,0) -- ++(5pt,-5pt);
\draw [fill=qqwuqq,shift={(1.155555555555555,0.15555555555555506)}] (0,0) ++(0 pt,3.75pt) -- ++(3.2475952641916446pt,-5.625pt)--++(-6.495190528383289pt,0 pt) -- ++(3.2475952641916446pt,5.625pt);
\draw [color=qqwuqq] (0.6266666666666645,2.414213562373095)-- ++(-2.5pt,-2.5pt) -- ++(5pt,5pt) ++(-5pt,0) -- ++(5pt,-5pt);
\draw [fill=qqwuqq] (-0.2611111111111077,2.153102451261988) circle (2.5pt);
\draw [fill=qqwuqq] (1.4459956700754397,0.44599567007544016) circle (2.5pt);
\end{scriptsize}
\draw [color=qqwuqq](-0.1,2.2) node[anchor=north west] {$\alpha_i$};
\draw [color=qqwuqq](0.8,0.35) node[anchor=north west] {$\alpha_i$};
\draw [color=ccqqqq](0.6,1.4) node[anchor=north west] {$\beta_j$};
\draw (-0.6,2.3) node[anchor=north west] {$e'$};
\draw (0.35,2.65) node[anchor=north west] {$e$};
\draw (0.35,0) node[anchor=north west] {$e$};
\draw (1.35,0.5) node[anchor=north west] {$e'$};
\end{tikzpicture}
\caption{The two cases where $\GIntij = 2$.}
\label{fig:case_two_intersections}
\end{figure}
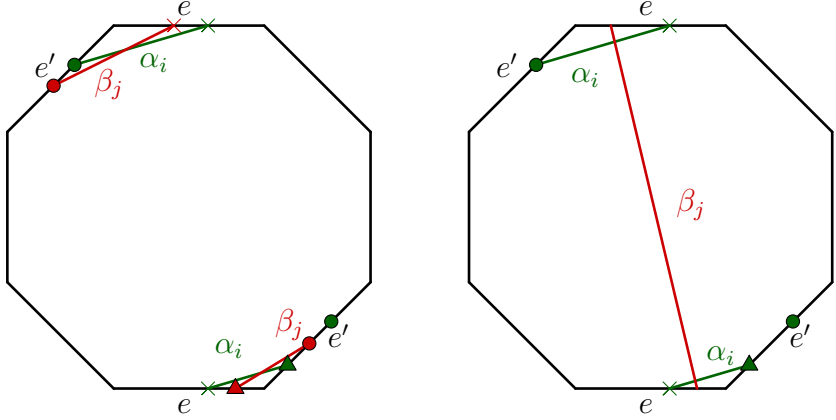

\begin{Rema}\label{rk:not_simultaneous_cases}
Notice that the two cases cannot happen simultaneously since in case $(i)$ we have $\{e,e'\} = \{e_{\sigma_{\beta}(1)}, e_{\sigma_{\beta}(0)}\}$ while in case $(ii)$ we have $\{e,e'\} = \{e_{\sigma_{\beta}(n/2-1)}, e_{\sigma_{\beta}(n/2-2)}\}$.
\end{Rema}

\begin{proof}[Proof of Lemma \ref{lem:cases_2_intersections}]
\textbf{$1^{st}$ case:} Assume that both $\alpha_i$ and $\beta_j$ are sandwiched segments. Then, the study of the intersections is exactly the same as the study of the intersections of pairs of sandwiched segments in the double $n$-gon for odd $n$, see \cite[\S 6.3]{BLM22}. Up to a rotation or a symmetry, we can assume $\alpha_i$ is sandwiched of type $e_2 \to e_1 \to e_2$.\newline

In particular, $\GIntij = 0$ unless $\beta_j$ has one of the following type:
$$
\begin{array}{llllllll}
(1)& e_0 \to e_1 \to e_0 &&&&& (2)& e_1 \to e_0 \to e_1 \\
(3)& e_1 \to e_2 \to e_1 &&&&& (4)& e_2 \to e_1 \to e_2 \\
(5)& e_2 \to e_3 \to e_2 &&&&& (6)& e_3 \to e_2 \to e_3
\end{array}
$$
(This is because $\alpha_i$ and $\beta_j$ have to share at least a common side of the $n$-gon to intersect).\newline

Similarly to the case of the double $n$-gon for odd $n$, we can show (see Figure~\ref{fig:intersections_sandwichees}) that in all cases but $(3)$ we have $\GIntij \leq 1$. In particular, the case $(3)$ is the only case where we possibly have $\GIntij = 2$ (as in the left of Figure \ref{fig:case_two_intersections}). \newline
\textbf{$2^{nd}$ case:} Up to permutation of $\alpha$ and $\beta$, $\alpha_i$ is sandwiched (say of type $e' \to e \to e'$) while $\beta_j$ is not.
In this case, we see easily that $\beta_j$ has to be a long segment whose extremities lie on the sides $e$ or $e'$ in order to get two intersections, as in the right of Figure \ref{fig:case_two_intersections}.
\end{proof}

\begin{figure}
\center
\definecolor{wwwwww}{rgb}{0.4,0.4,0.4}
\definecolor{ccqqqq}{rgb}{0.8,0,0}
\definecolor{qqwuqq}{rgb}{0,0.39215686274509803,0}
\begin{tikzpicture}[line cap=round,line join=round,>=triangle 45,x=1cm,y=1cm, scale = 1.7]
\clip(-1.7,-2) rectangle (2.5,2);
\fill[line width=1pt,color=wwwwww,fill=wwwwww,fill opacity=0.1] (-0.8090169943749475,0.5877852522924734) -- (0,0) -- (1,0) -- (1.809016994374947,0.5877852522924729) -- (2.118033988749895,1.5388417685876261) -- cycle;
\fill[line width=1pt,color=wwwwww,fill=wwwwww,fill opacity=0.1] (-1.1180339887498947,-1.5388417685876261) -- (-0.809016994374947,-0.5877852522924729) -- (0,0) -- (1,0) -- (1.8090169943749475,-0.5877852522924734) -- cycle;
\draw [line width=1pt] (-0.8090169943749475,0.5877852522924734)-- (0,0);
\draw [line width=1pt] (0,0)-- (1,0);
\draw [line width=1pt] (1,0)-- (1.809016994374947,0.5877852522924729);
\draw [line width=1pt] (1.809016994374947,0.5877852522924729)-- (2.118033988749895,1.5388417685876261);
\draw [line width=1pt] (0,0)-- (-0.809016994374947,-0.5877852522924729);
\draw [line width=1pt] (-0.809016994374947,-0.5877852522924729)-- (-1.1180339887498947,-1.5388417685876261);
\draw [line width=1pt] (1,0)-- (1.8090169943749475,-0.5877852522924734);
\draw [line width=1pt,color=ccqqqq] (-0.17235204515320687,-0.12522109059250505)-- (1.267800736216506,0.19456862389243712);
\draw (0.3,1.4) node[anchor=north west] {$\cdots$};
\draw (0.3,-0.9) node[anchor=north west] {$\cdots$};
\draw (1.3,0.3) node[anchor=north west] {$e_2$};
\draw (0.3,0) node[anchor=north west] {$e_1$};
\draw (-0.9,0.45) node[anchor=north west] {$e_0$};
\draw (1.45,0) node[anchor=north west] {$e_0$};
\draw (-1.4,-0.85) node[anchor=north west] {$e_3$};
\draw (-0.85,0.05) node[anchor=north west] {$e_2$};
\draw (1.95,1.2) node[anchor=north west] {$e_3$};
\draw [color=ccqqqq](0.6,0.5511111111111096) node[anchor=north west] {$\alpha_i$};
\draw (0.5, -1.5) node[below] {$(1)$ $e_0 \to e_1 \to e_0$};
\draw [line width=1pt,color=qqwuqq] (-0.49,0.36) -- (1.41,-0.3);
\draw [color=qqwuqq](-0.1,0.7) node[anchor=north west] {$\beta_j$};
\begin{scriptsize}
\draw [fill=qqwuqq] (-0.49,0.36) circle (2pt);
\draw [color=qqwuqq] (1.41,-0.3)-- ++(-2.5pt,-2.5pt) -- ++(5pt,5pt) ++(-5pt,0) -- ++(5pt,-5pt);
\draw [fill=black] (0,0) circle (1pt);
\draw [fill=black] (1,0) circle (1pt);
\draw [fill=black] (1.809016994374947,0.5877852522924729) circle (1pt);
\draw [fill=black] (2.118033988749895,1.5388417685876261) circle (1pt);
\draw [fill=black] (-0.8090169943749475,0.5877852522924734) circle (1pt);
\draw [fill=black] (-0.809016994374947,-0.5877852522924729) circle (1pt);
\draw [fill=black] (-1.1180339887498947,-1.5388417685876261) circle (1pt);
\draw [fill=black] (1.8090169943749475,-0.5877852522924734) circle (1pt);
\draw [fill=ccqqqq] (-0.17235204515320687,-0.12522109059250505) circle (2pt);
\draw [color=ccqqqq] (1.267800736216506,0.19456862389243712)-- ++(-2.5pt,-2.5pt) -- ++(5pt,5pt) ++(-5pt,0) -- ++(5pt,-5pt);
\draw [color=ccqqqq] (-0.541216258158441,-0.3932166284000358)-- ++(-2.5pt,-2.5pt) -- ++(5pt,5pt) ++(-5pt,0) -- ++(5pt,-5pt);
\draw [fill=ccqqqq] (1.6366649492217402,0.46256416169996784) circle (2pt);
\end{scriptsize}
\end{tikzpicture}
\begin{tikzpicture}[line cap=round,line join=round,>=triangle 45,x=1cm,y=1cm, scale = 1.7]
\clip(-1.7,-2) rectangle (2.5,2);
\fill[line width=1pt,color=wwwwww,fill=wwwwww,fill opacity=0.1] (-0.8090169943749475,0.5877852522924734) -- (0,0) -- (1,0) -- (1.809016994374947,0.5877852522924729) -- (2.118033988749895,1.5388417685876261) -- cycle;
\fill[line width=1pt,color=wwwwww,fill=wwwwww,fill opacity=0.1] (-1.1180339887498947,-1.5388417685876261) -- (-0.809016994374947,-0.5877852522924729) -- (0,0) -- (1,0) -- (1.8090169943749475,-0.5877852522924734) -- cycle;
\draw [line width=1pt] (-0.8090169943749475,0.5877852522924734)-- (0,0);
\draw [line width=1pt] (0,0)-- (1,0);
\draw [line width=1pt] (1,0)-- (1.809016994374947,0.5877852522924729);
\draw [line width=1pt] (1.809016994374947,0.5877852522924729)-- (2.118033988749895,1.5388417685876261);
\draw [line width=1pt] (0,0)-- (-0.809016994374947,-0.5877852522924729);
\draw [line width=1pt] (-0.809016994374947,-0.5877852522924729)-- (-1.1180339887498947,-1.5388417685876261);
\draw [line width=1pt] (1,0)-- (1.8090169943749475,-0.5877852522924734);
\draw [line width=1pt,color=ccqqqq] (-0.17235204515320687,-0.12522109059250505)-- (1.267800736216506,0.19456862389243712);
\draw (0.3,1.4) node[anchor=north west] {$\cdots$};
\draw (0.3,-0.9) node[anchor=north west] {$\cdots$};
\draw (1.3,0.3) node[anchor=north west] {$e_2$};
\draw (0.3,0) node[anchor=north west] {$e_1$};
\draw (-0.9,0.45) node[anchor=north west] {$e_0$};
\draw (1.45,0) node[anchor=north west] {$e_0$};
\draw (-1.4,-0.85) node[anchor=north west] {$e_3$};
\draw (-0.85,0.05) node[anchor=north west] {$e_2$};
\draw (1.95,1.2) node[anchor=north west] {$e_3$};
\draw [color=ccqqqq](0.6,0.5511111111111096) node[anchor=north west] {$\alpha_i$};
\draw (0.5, -1.5) node[below] {$(2)$ $e_1 \to e_0 \to e_1$};
\draw [line width=1pt,color=qqwuqq] (0.76,0) -- (1.26,-0.19);
\draw [line width=1pt,color=qqwuqq] (-0.55,0.4) -- (0.53,0);
\draw [color=qqwuqq](-0.1,0.7) node[anchor=north west] {$\beta_j$};
\begin{scriptsize}
\draw [fill=qqwuqq] (1.26,-0.19) circle (2pt);
\draw [fill=qqwuqq] (-0.55,0.4) circle (2pt);
\draw [color=qqwuqq] (0.53,0)-- ++(-2.5pt,-2.5pt) -- ++(5pt,5pt) ++(-5pt,0) -- ++(5pt,-5pt);
\draw [fill=qqwuqq,shift={(0.76,0)}] (0,0) ++(0 pt,1.825pt) -- ++(1.62pt,-2.81pt)--++(-3.25pt,0 pt) -- ++(1.62pt,2.81pt);
\draw [fill=black] (0,0) circle (1pt);
\draw [fill=black] (1,0) circle (1pt);
\draw [fill=black] (1.809016994374947,0.5877852522924729) circle (1pt);
\draw [fill=black] (2.118033988749895,1.5388417685876261) circle (1pt);
\draw [fill=black] (-0.8090169943749475,0.5877852522924734) circle (1pt);
\draw [fill=black] (-0.809016994374947,-0.5877852522924729) circle (1pt);
\draw [fill=black] (-1.1180339887498947,-1.5388417685876261) circle (1pt);
\draw [fill=black] (1.8090169943749475,-0.5877852522924734) circle (1pt);
\draw [fill=ccqqqq] (-0.17235204515320687,-0.12522109059250505) circle (2pt);
\draw [color=ccqqqq] (1.267800736216506,0.19456862389243712)-- ++(-2.5pt,-2.5pt) -- ++(5pt,5pt) ++(-5pt,0) -- ++(5pt,-5pt);
\draw [color=ccqqqq] (-0.541216258158441,-0.3932166284000358)-- ++(-2.5pt,-2.5pt) -- ++(5pt,5pt) ++(-5pt,0) -- ++(5pt,-5pt);
\draw [fill=ccqqqq] (1.6366649492217402,0.46256416169996784) circle (2pt);
\end{scriptsize}
\end{tikzpicture}
\begin{tikzpicture}[line cap=round,line join=round,>=triangle 45,x=1cm,y=1cm, scale = 1.7]
\clip(-1.7,-2) rectangle (2.5,2);
\fill[line width=1pt,color=wwwwww,fill=wwwwww,fill opacity=0.1] (-0.8090169943749475,0.5877852522924734) -- (0,0) -- (1,0) -- (1.809016994374947,0.5877852522924729) -- (2.118033988749895,1.5388417685876261) -- cycle;
\fill[line width=1pt,color=wwwwww,fill=wwwwww,fill opacity=0.1] (-1.1180339887498947,-1.5388417685876261) -- (-0.809016994374947,-0.5877852522924729) -- (0,0) -- (1,0) -- (1.8090169943749475,-0.5877852522924734) -- cycle;
\draw [line width=1pt] (-0.8090169943749475,0.5877852522924734)-- (0,0);
\draw [line width=1pt] (0,0)-- (1,0);
\draw [line width=1pt] (1,0)-- (1.809016994374947,0.5877852522924729);
\draw [line width=1pt] (1.809016994374947,0.5877852522924729)-- (2.118033988749895,1.5388417685876261);
\draw [line width=1pt] (0,0)-- (-0.809016994374947,-0.5877852522924729);
\draw [line width=1pt] (-0.809016994374947,-0.5877852522924729)-- (-1.1180339887498947,-1.5388417685876261);
\draw [line width=1pt] (1,0)-- (1.8090169943749475,-0.5877852522924734);
\draw [line width=1pt,color=ccqqqq] (-0.17235204515320687,-0.12522109059250505)-- (1.267800736216506,0.19456862389243712);
\draw (0.3,1.4) node[anchor=north west] {$\cdots$};
\draw (0.3,-0.9) node[anchor=north west] {$\cdots$};
\draw (1.3,0.3) node[anchor=north west] {$e_2$};
\draw (0.3,0) node[anchor=north west] {$e_1$};
\draw (-0.9,0.45) node[anchor=north west] {$e_0$};
\draw (1.45,0) node[anchor=north west] {$e_0$};
\draw (-1.4,-0.85) node[anchor=north west] {$e_3$};
\draw (-0.85,0.05) node[anchor=north west] {$e_2$};
\draw (1.95,1.2) node[anchor=north west] {$e_3$};
\draw [color=ccqqqq](0.6,0.5511111111111096) node[anchor=north west] {$\alpha_i$};
\draw (0.5, -1.5) node[below] {$(3)$ $e_1 \to e_2 \to e_1$};
\draw [line width=1pt,color=qqwuqq] (0.17,0) -- (-0.35,-0.25);
\draw [line width=1pt,color=qqwuqq] (1.46,0.33) -- (0.78,0);
\draw [color=qqwuqq](-0.1,-0.1) node[anchor=north west] {$\beta_j$};
\begin{scriptsize}
\draw [fill=qqwuqq] (-0.35,-0.25) circle (2pt);
\draw [fill=qqwuqq] (1.46,0.33) circle (2pt);
\draw [color=qqwuqq] (0.17,0)-- ++(-2.5pt,-2.5pt) -- ++(5pt,5pt) ++(-5pt,0) -- ++(5pt,-5pt);
\draw [fill=qqwuqq,shift={(0.78,0)}] (0,0) ++(0 pt,1.825pt) -- ++(1.62pt,-2.81pt)--++(-3.25pt,0 pt) -- ++(1.62pt,2.81pt);
\draw [fill=black] (0,0) circle (1pt);
\draw [fill=black] (1,0) circle (1pt);
\draw [fill=black] (1.809016994374947,0.5877852522924729) circle (1pt);
\draw [fill=black] (2.118033988749895,1.5388417685876261) circle (1pt);
\draw [fill=black] (-0.8090169943749475,0.5877852522924734) circle (1pt);
\draw [fill=black] (-0.809016994374947,-0.5877852522924729) circle (1pt);
\draw [fill=black] (-1.1180339887498947,-1.5388417685876261) circle (1pt);
\draw [fill=black] (1.8090169943749475,-0.5877852522924734) circle (1pt);
\draw [fill=ccqqqq] (-0.17235204515320687,-0.12522109059250505) circle (2pt);
\draw [color=ccqqqq] (1.267800736216506,0.19456862389243712)-- ++(-2.5pt,-2.5pt) -- ++(5pt,5pt) ++(-5pt,0) -- ++(5pt,-5pt);
\draw [color=ccqqqq] (-0.541216258158441,-0.3932166284000358)-- ++(-2.5pt,-2.5pt) -- ++(5pt,5pt) ++(-5pt,0) -- ++(5pt,-5pt);
\draw [fill=ccqqqq] (1.6366649492217402,0.46256416169996784) circle (2pt);
\end{scriptsize}
\end{tikzpicture}
\begin{tikzpicture}[line cap=round,line join=round,>=triangle 45,x=1cm,y=1cm, scale = 1.7]
\clip(-1.7,-2) rectangle (2.5,2);
\fill[line width=1pt,color=wwwwww,fill=wwwwww,fill opacity=0.1] (-0.8090169943749475,0.5877852522924734) -- (0,0) -- (1,0) -- (1.809016994374947,0.5877852522924729) -- (2.118033988749895,1.5388417685876261) -- cycle;
\fill[line width=1pt,color=wwwwww,fill=wwwwww,fill opacity=0.1] (-1.1180339887498947,-1.5388417685876261) -- (-0.809016994374947,-0.5877852522924729) -- (0,0) -- (1,0) -- (1.8090169943749475,-0.5877852522924734) -- cycle;
\draw [line width=1pt] (-0.8090169943749475,0.5877852522924734)-- (0,0);
\draw [line width=1pt] (0,0)-- (1,0);
\draw [line width=1pt] (1,0)-- (1.809016994374947,0.5877852522924729);
\draw [line width=1pt] (1.809016994374947,0.5877852522924729)-- (2.118033988749895,1.5388417685876261);
\draw [line width=1pt] (0,0)-- (-0.809016994374947,-0.5877852522924729);
\draw [line width=1pt] (-0.809016994374947,-0.5877852522924729)-- (-1.1180339887498947,-1.5388417685876261);
\draw [line width=1pt] (1,0)-- (1.8090169943749475,-0.5877852522924734);
\draw [line width=1pt,color=ccqqqq] (-0.17235204515320687,-0.12522109059250505)-- (1.267800736216506,0.19456862389243712);
\draw (0.3,1.4) node[anchor=north west] {$\cdots$};
\draw (0.3,-0.9) node[anchor=north west] {$\cdots$};
\draw (1.3,0.3) node[anchor=north west] {$e_2$};
\draw (0.3,0) node[anchor=north west] {$e_1$};
\draw (-0.9,0.45) node[anchor=north west] {$e_0$};
\draw (1.45,0) node[anchor=north west] {$e_0$};
\draw (-1.4,-0.85) node[anchor=north west] {$e_3$};
\draw (-0.85,0.05) node[anchor=north west] {$e_2$};
\draw (1.95,1.2) node[anchor=north west] {$e_3$};
\draw [color=ccqqqq](0.6,0.5511111111111096) node[anchor=north west] {$\alpha_i$};
\draw (0.5, -1.5) node[below] {$(4)$ $e_2 \to e_1 \to e_2$};
\draw [line width=1pt,color=qqwuqq] (-0.37,-0.27) -- (1.38,0.28);
\draw [color=qqwuqq](-0.1,-0.1) node[anchor=north west] {$\beta_j$};
\begin{scriptsize}
\draw [fill=qqwuqq] (-0.37,-0.27) circle (2pt);
\draw [color=qqwuqq] (1.38,0.28)-- ++(-2.5pt,-2.5pt) -- ++(5pt,5pt) ++(-5pt,0) -- ++(5pt,-5pt);
\draw [fill=black] (0,0) circle (1pt);
\draw [fill=black] (1,0) circle (1pt);
\draw [fill=black] (1.809016994374947,0.5877852522924729) circle (1pt);
\draw [fill=black] (2.118033988749895,1.5388417685876261) circle (1pt);
\draw [fill=black] (-0.8090169943749475,0.5877852522924734) circle (1pt);
\draw [fill=black] (-0.809016994374947,-0.5877852522924729) circle (1pt);
\draw [fill=black] (-1.1180339887498947,-1.5388417685876261) circle (1pt);
\draw [fill=black] (1.8090169943749475,-0.5877852522924734) circle (1pt);
\draw [fill=ccqqqq] (-0.17235204515320687,-0.12522109059250505) circle (2pt);
\draw [color=ccqqqq] (1.267800736216506,0.19456862389243712)-- ++(-2.5pt,-2.5pt) -- ++(5pt,5pt) ++(-5pt,0) -- ++(5pt,-5pt);
\draw [color=ccqqqq] (-0.541216258158441,-0.3932166284000358)-- ++(-2.5pt,-2.5pt) -- ++(5pt,5pt) ++(-5pt,0) -- ++(5pt,-5pt);
\draw [fill=ccqqqq] (1.6366649492217402,0.46256416169996784) circle (2pt);
\end{scriptsize}
\end{tikzpicture}
\begin{tikzpicture}[line cap=round,line join=round,>=triangle 45,x=1cm,y=1cm, scale = 1.7]
\clip(-1.7,-2) rectangle (2.5,2);
\fill[line width=1pt,color=wwwwww,fill=wwwwww,fill opacity=0.1] (-0.8090169943749475,0.5877852522924734) -- (0,0) -- (1,0) -- (1.809016994374947,0.5877852522924729) -- (2.118033988749895,1.5388417685876261) -- cycle;
\fill[line width=1pt,color=wwwwww,fill=wwwwww,fill opacity=0.1] (-1.1180339887498947,-1.5388417685876261) -- (-0.809016994374947,-0.5877852522924729) -- (0,0) -- (1,0) -- (1.8090169943749475,-0.5877852522924734) -- cycle;
\draw [line width=1pt] (-0.8090169943749475,0.5877852522924734)-- (0,0);
\draw [line width=1pt] (0,0)-- (1,0);
\draw [line width=1pt] (1,0)-- (1.809016994374947,0.5877852522924729);
\draw [line width=1pt] (1.809016994374947,0.5877852522924729)-- (2.118033988749895,1.5388417685876261);
\draw [line width=1pt] (0,0)-- (-0.809016994374947,-0.5877852522924729);
\draw [line width=1pt] (-0.809016994374947,-0.5877852522924729)-- (-1.1180339887498947,-1.5388417685876261);
\draw [line width=1pt] (1,0)-- (1.8090169943749475,-0.5877852522924734);
\draw [line width=1pt,color=ccqqqq] (-0.17235204515320687,-0.12522109059250505)-- (1.267800736216506,0.19456862389243712);
\draw (0.3,1.4) node[anchor=north west] {$\cdots$};
\draw (0.3,-0.9) node[anchor=north west] {$\cdots$};
\draw (1.3,0.3) node[anchor=north west] {$e_2$};
\draw (0.3,0) node[anchor=north west] {$e_1$};
\draw (-0.9,0.45) node[anchor=north west] {$e_0$};
\draw (1.45,0) node[anchor=north west] {$e_0$};
\draw (-1.4,-0.85) node[anchor=north west] {$e_3$};
\draw (-0.85,0.05) node[anchor=north west] {$e_2$};
\draw (2.05,1.2) node[anchor=north west] {$e_3$};
\draw [color=ccqqqq](0.6,0.5511111111111096) node[anchor=north west] {$\alpha_i$};
\draw (0.5, -1.5) node[below] {$(5)$ $e_2 \to e_3 \to e_2$};
\draw [line width=1pt,color=qqwuqq] (1.49,0.36) -- (1.96,1.05);
\draw [line width=1pt,color=qqwuqq] (-0.97,-1.08) -- (-0.47,-0.34);
\draw [color=qqwuqq](1.3,1) node[anchor=north west] {$\beta_j$};
\begin{scriptsize}
\draw [fill=qqwuqq] (1.96,1.05) circle (2pt);
\draw [fill=qqwuqq] (-0.97,-1.08) circle (2pt);
\draw [color=qqwuqq] (1.49,0.36)-- ++(-2.5pt,-2.5pt) -- ++(5pt,5pt) ++(-5pt,0) -- ++(5pt,-5pt);
\draw [fill=qqwuqq,shift={(-0.47,-0.34)}] (0,0) ++(0 pt,1.825pt) -- ++(1.62pt,-2.81pt)--++(-3.25pt,0 pt) -- ++(1.62pt,2.81pt);
\draw [fill=black] (0,0) circle (1pt);
\draw [fill=black] (1,0) circle (1pt);
\draw [fill=black] (1.809016994374947,0.5877852522924729) circle (1pt);
\draw [fill=black] (2.118033988749895,1.5388417685876261) circle (1pt);
\draw [fill=black] (-0.8090169943749475,0.5877852522924734) circle (1pt);
\draw [fill=black] (-0.809016994374947,-0.5877852522924729) circle (1pt);
\draw [fill=black] (-1.1180339887498947,-1.5388417685876261) circle (1pt);
\draw [fill=black] (1.8090169943749475,-0.5877852522924734) circle (1pt);
\draw [fill=ccqqqq] (-0.17235204515320687,-0.12522109059250505) circle (2pt);
\draw [color=ccqqqq] (1.267800736216506,0.19456862389243712)-- ++(-2.5pt,-2.5pt) -- ++(5pt,5pt) ++(-5pt,0) -- ++(5pt,-5pt);
\draw [color=ccqqqq] (-0.541216258158441,-0.3932166284000358)-- ++(-2.5pt,-2.5pt) -- ++(5pt,5pt) ++(-5pt,0) -- ++(5pt,-5pt);
\draw [fill=ccqqqq] (1.6366649492217402,0.46256416169996784) circle (2pt);
\end{scriptsize}
\end{tikzpicture}
\begin{tikzpicture}[line cap=round,line join=round,>=triangle 45,x=1cm,y=1cm, scale = 1.7]
\clip(-1.7,-2) rectangle (2.5,2);
\fill[line width=1pt,color=wwwwww,fill=wwwwww,fill opacity=0.1] (-0.8090169943749475,0.5877852522924734) -- (0,0) -- (1,0) -- (1.809016994374947,0.5877852522924729) -- (2.118033988749895,1.5388417685876261) -- cycle;
\fill[line width=1pt,color=wwwwww,fill=wwwwww,fill opacity=0.1] (-1.1180339887498947,-1.5388417685876261) -- (-0.809016994374947,-0.5877852522924729) -- (0,0) -- (1,0) -- (1.8090169943749475,-0.5877852522924734) -- cycle;
\draw [line width=1pt] (-0.8090169943749475,0.5877852522924734)-- (0,0);
\draw [line width=1pt] (0,0)-- (1,0);
\draw [line width=1pt] (1,0)-- (1.809016994374947,0.5877852522924729);
\draw [line width=1pt] (1.809016994374947,0.5877852522924729)-- (2.118033988749895,1.5388417685876261);
\draw [line width=1pt] (0,0)-- (-0.809016994374947,-0.5877852522924729);
\draw [line width=1pt] (-0.809016994374947,-0.5877852522924729)-- (-1.1180339887498947,-1.5388417685876261);
\draw [line width=1pt] (1,0)-- (1.8090169943749475,-0.5877852522924734);
\draw [line width=1pt,color=ccqqqq] (-0.17235204515320687,-0.12522109059250505)-- (1.267800736216506,0.19456862389243712);
\draw (0.3,1.4) node[anchor=north west] {$\cdots$};
\draw (0.3,-0.9) node[anchor=north west] {$\cdots$};
\draw (1.3,0.3) node[anchor=north west] {$e_2$};
\draw (0.3,0) node[anchor=north west] {$e_1$};
\draw (-0.9,0.45) node[anchor=north west] {$e_0$};
\draw (1.45,0) node[anchor=north west] {$e_0$};
\draw (-1.4,-0.85) node[anchor=north west] {$e_3$};
\draw (-0.85,0.05) node[anchor=north west] {$e_2$};
\draw (2.05,1.2) node[anchor=north west] {$e_3$};
\draw [color=ccqqqq](0.6,0.5511111111111096) node[anchor=north west] {$\alpha_i$};
\draw (0.5, -1.5) node[below] {$(6)$ $e_3 \to e_2 \to e_3$};
\draw [line width=1pt,color=qqwuqq] (1.99,1.14) -- (1.39,0.29);
\draw [line width=1pt,color=qqwuqq] (-0.98,-1.11) -- (-0.41,-0.3);
\draw [color=qqwuqq](1.3,1) node[anchor=north west] {$\beta_j$};
\begin{scriptsize}
\draw [fill=qqwuqq] (1.39,0.29) circle (2pt);
\draw [fill=qqwuqq] (-0.41,-0.3) circle (2pt);
\draw [color=qqwuqq] (-0.98,-1.11)-- ++(-2.5pt,-2.5pt) -- ++(5pt,5pt) ++(-5pt,0) -- ++(5pt,-5pt);
\draw [fill=qqwuqq,shift={(1.99,1.14)}] (0,0) ++(0 pt,1.825pt) -- ++(1.62pt,-2.81pt)--++(-3.25pt,0 pt) -- ++(1.62pt,2.81pt);
\draw [fill=black] (0,0) circle (1pt);
\draw [fill=black] (1,0) circle (1pt);
\draw [fill=black] (1.809016994374947,0.5877852522924729) circle (1pt);
\draw [fill=black] (2.118033988749895,1.5388417685876261) circle (1pt);
\draw [fill=black] (-0.8090169943749475,0.5877852522924734) circle (1pt);
\draw [fill=black] (-0.809016994374947,-0.5877852522924729) circle (1pt);
\draw [fill=black] (-1.1180339887498947,-1.5388417685876261) circle (1pt);
\draw [fill=black] (1.8090169943749475,-0.5877852522924734) circle (1pt);
\draw [fill=ccqqqq] (-0.17235204515320687,-0.12522109059250505) circle (2pt);
\draw [color=ccqqqq] (1.267800736216506,0.19456862389243712)-- ++(-2.5pt,-2.5pt) -- ++(5pt,5pt) ++(-5pt,0) -- ++(5pt,-5pt);
\draw [color=ccqqqq] (-0.541216258158441,-0.3932166284000358)-- ++(-2.5pt,-2.5pt) -- ++(5pt,5pt) ++(-5pt,0) -- ++(5pt,-5pt);
\draw [fill=ccqqqq] (1.6366649492217402,0.46256416169996784) circle (2pt);
\end{scriptsize}
\end{tikzpicture}
\caption{The six cases in Lemma \ref{lem:cases_2_intersections}, $(i)$.}
\label{fig:intersections_sandwichees}
\end{figure}
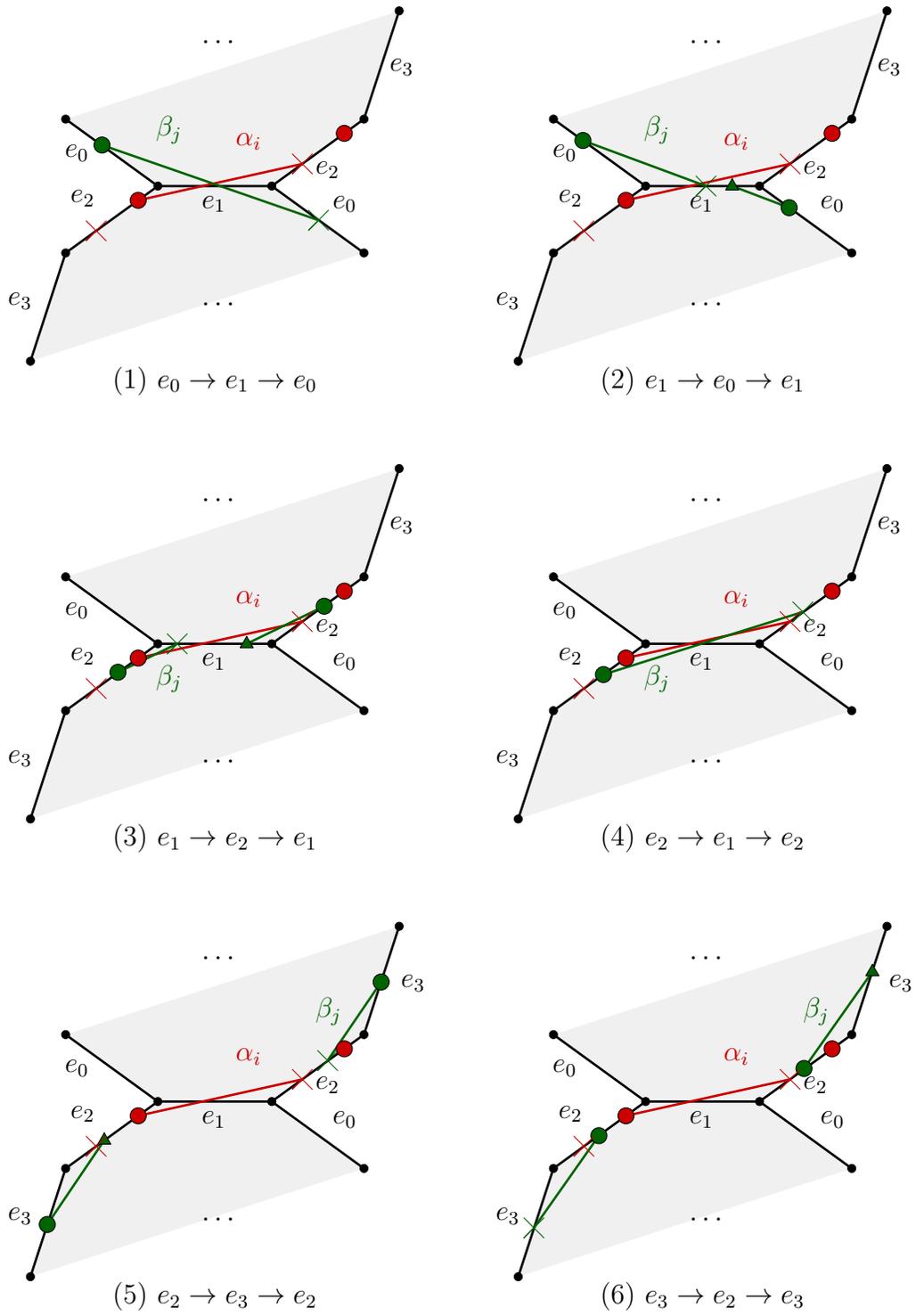

In light of Lemma \ref{lem:cases_2_intersections} and Remark \ref{rk:not_simultaneous_cases}, we distinguish three mutually exclusive cases:
\begin{enumerate}
\item[(0)] There is no configuration of type $(i)$ or $(ii)$.
\item[(i)] There exists $i,j$ such that $\alpha_i$ and $\beta_j$ are in a configuration of type $(i)$.
\item[(ii)] There exists $i,j$ such that $\alpha_i$ and $\beta_j$ are in a configuration of type $(ii)$.
\end{enumerate}

\paragraph{\underline{Case $(0)$: There is no configuration of type $(i)$ or $(ii)$.}}
In this case, we deduce from Equation \eqref{eq:intersections} and Lemma \ref{lem:cases_2_intersections} that:
\begin{equation}\label{eq:intersection2}
|\alpha \cap \beta| \leq kl.
\end{equation}
In order to get Equation \eqref{eq:intersection_general}, we either improve the inequality on the intersections (Equation \eqref{eq:intersection2}) or the inequality on the lengths $l(\alpha)l(\beta) > kl \cdot l_0^2$. For this purpose, we distinguish two cases:
\begin{enumerate}
\item If $\alpha$ is not contained in the cylinder defined by $e_{\sigma_{\alpha}(0)}, e_{\sigma_{\alpha}(1)}$ as in Figure \ref{case_1_octogone} (in the example of the octagon), then we can find two consecutive segments $\alpha_i$ and $\alpha_{i+1}$ which are not contained in the cylinder. As explained in Figure \ref{case_1_octogone}, the length $l(\alpha_i) + l(\alpha_{i+1})$ should be at least $\Phi^2 l_0$. Since $\Phi^2 > 3$, we have $l(\alpha_i) + l(\alpha_{i+1}) > 3 l_0$, so that $l(\alpha) > (k+1)l_0$, and $$l(\alpha)l(\beta) > (kl+1)l_0^2.$$

Using Equation \eqref{eq:intersection2} we have directly that Equation \eqref{eq:intersection_general} holds (the inequality being strict in this case). By symmetry, the same argument holds if $\beta$ is not contained in the cylinder defined by $e_{\sigma_{\beta}(0)}, e_{\sigma_{\beta}(1)}$.

%\textcolor{blue}{Dans le cas général convexe à angle obtus, cet argument ne marche plus. On doit pouvoir faire autrement...}
\begin{figure}[h]
\center
\definecolor{ccqqqq}{rgb}{0.8,0,0}
\definecolor{wwwwww}{rgb}{0.4,0.4,0.4}
\begin{tikzpicture}[line cap=round,line join=round,>=triangle 45,x=1cm,y=1cm]
\clip(-2.5,-1) rectangle (5.9,5.5);
\fill[line width=1pt,color=wwwwww,fill=wwwwww,fill opacity=0.2] (-2.4142135623730954,1.4142135623730956) -- (2.414213562373095,1.414213562373095) -- (1,0) -- (-1,0) -- cycle;
\fill[line width=1pt,fill=black,fill opacity=0.05] (-1,4.82842712474619) -- (1,4.82842712474619) -- (2.414213562373095,3.4142135623730945) -- (-2.414213562373095,3.4142135623730954) -- cycle;
\fill[line width=1pt,color=wwwwww,fill=wwwwww,fill opacity=0.4] (2.414213562373095,1.414213562373095) -- (1,0) -- (-1,0) -- cycle;
\fill[line width=1pt,color=wwwwww,fill=wwwwww,fill opacity=0.2] (1,0) -- (2.414213562373095,1.414213562373095) -- (4.414213562373095,1.414213562373095) -- (5.82842712474619,0) -- cycle;
\fill[line width=1pt,color=wwwwww,fill=wwwwww,fill opacity=0.4] (1,0) -- (2.414213562373095,1.414213562373095) -- (4.414213562373095,1.414213562373095)-- cycle;
\draw [line width=1pt] (-1,0)-- (1,0);
\draw [line width=1pt] (1,0)-- (2.414213562373095,1.414213562373095);
\draw [line width=1pt] (2.414213562373095,1.414213562373095)-- (2.414213562373095,3.4142135623730945);
\draw [line width=1pt] (2.414213562373095,3.4142135623730945)-- (1,4.82842712474619);
\draw [line width=1pt] (1,4.82842712474619)-- (-1,4.82842712474619);
\draw [line width=1pt] (-1,4.82842712474619)-- (-2.414213562373095,3.4142135623730954);
\draw [line width=1pt] (-2.414213562373095,3.4142135623730954)-- (-2.4142135623730954,1.4142135623730956);
\draw [line width=1pt] (-2.4142135623730954,1.4142135623730956)-- (-1,0);
\draw [line width=1pt,color=wwwwww] (2.414213562373095,1.414213562373095)-- (4.414213562373095,1.414213562373095);
\draw [line width=1pt,color=wwwwww] (4.414213562373095,1.414213562373095)-- (5.82842712474619,0);
\draw [line width=1pt,dash pattern=on 3pt off 3pt] (-1,0)-- (2.414213562373095,1.414213562373095);
\draw [line width=1pt,dash pattern=on 3pt off 3pt] (1,0)-- (4.414213562373095,1.414213562373095);
\draw [line width=1pt,color=ccqqqq] (1.55,0.55)-- (5.044213562373095,0.7842135623730951);
\draw [line width=1pt,color=ccqqqq] (-1.7842135623730955,0.7842135623730958)-- (2.0405859638932142,1.0405859638932142);
\draw [line width=1pt] (-0.96,-0.2)-- (5.86,-0.2);
\draw [line width=1pt] (-0.96,-0.2)-- (-0.86,-0.1);
\draw [line width=1pt] (-0.96,-0.2)-- (-0.86,-0.3);
\draw [line width=1pt] (5.76,-0.1)-- (5.86,-0.2);
\draw [line width=1pt] (5.76,-0.3)-- (5.86,-0.2);
\draw (1.62,-0.26) node[anchor=north west] {$\Phi^2 l_0 > 3l_0$};
\draw (-0.05,0.45) node[anchor=north west] {$e_0$};
\draw (1.7,1) node[anchor=north west] {$e_1$};
\draw [color=ccqqqq](-0.5,1.4) node[anchor=north west] {$\alpha_{i+1}$};
\draw [color=ccqqqq](3,0.7) node[anchor=north west] {$\alpha_i$};
\begin{scriptsize}
\draw [fill=ccqqqq] (1.55,0.55) circle (2.5pt);
\draw [color=ccqqqq] (5.044213562373095,0.7842135623730951)-- ++(-2.5pt,-2.5pt) -- ++(5pt,5pt) ++(-5pt,0) -- ++(5pt,-5pt);
\draw [color=ccqqqq] (-1.7842135623730955,0.7842135623730958)-- ++(-2pt,-2pt) -- ++(4pt,4pt) ++(-4pt,0) -- ++(4pt,-4pt);
\draw [fill=ccqqqq] (2.0405859638932142,1.0405859638932142) ++(-2.5pt,0 pt) -- ++(2.5pt,2.5pt)--++(2.5pt,-2.5pt)--++(-2.5pt,-2.5pt)--++(-2.5pt,2.5pt);
\end{scriptsize}
\end{tikzpicture}
\caption{If $\alpha$ is not contained in the cylinder defined by $e_0$ and $e_1$, then there are two consecutive segments $\alpha_i$ and $\alpha_{i+1}$ with $l(\alpha_i)+l(\alpha_{i+1}) > 3 l_0$.}
\label{case_1_octogone}
\end{figure}
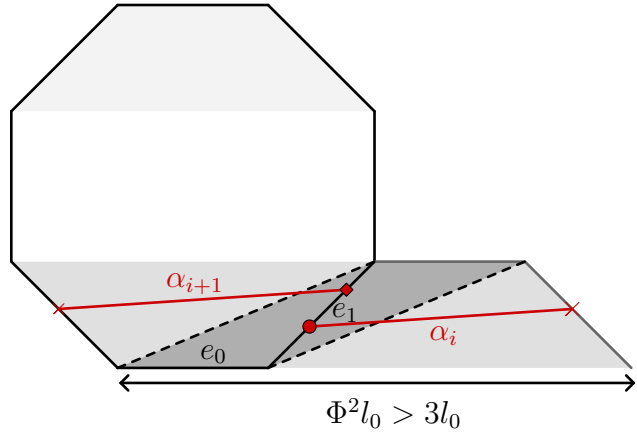

\item Otherwise, we can assume that $\alpha$ is contained in the cylinder defined by $e_{\sigma_ {\alpha}(0)}$ and $e_{\sigma_{\alpha}(1)}$ while $\beta$ is contained in the cylinder defined by $e_{\sigma_ {\beta}(0)}$ and $e_{\sigma_{\beta}(1)}$. In this case, unless $\beta$ is a diagonal, we see that $\alpha_1$ cannot intersect both $\beta_1$ (which lies in a region $S_0$ as in Figure \ref{fig:case_2_octogone}) and $\beta_l$ (which lie in a symmetric region $S_1$), so that $$|\alpha \cap \beta| \leq kl-1.$$
In particular $|\alpha \cap \beta|+1 \leq kl$ and using that $l(\alpha)l(\beta) > kl \cdot l_0^2$, we deduce that Equation~\eqref{eq:intersection_general} holds (and the inequality is strict in this case). 

\item Finally, if both $\alpha$ and $\beta$ are diagonals, we have $l(\alpha) ,l(\beta) \geq \Phi l_0$ and hence $l(\alpha)l(\beta) \geq \Phi^2 l_0^2 > 2 l_0^2$, so that Equation \eqref{eq:intersection_general} holds as well and the inequality is strict.

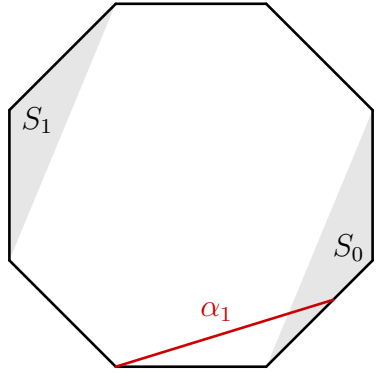
\begin{figure}[h]
\center
\definecolor{ccqqqq}{rgb}{0.8,0,0}
\begin{tikzpicture}[line cap=round,line join=round,>=triangle 45,x=1cm,y=1cm]
\clip(-3,-0.5) rectangle (3,6.3);
\fill[line width=1pt,fill=black,fill opacity=0.1] (2.414213562373095,3.4142135623730945) -- (1,0) -- (2.414213562373095,1.414213562373095) -- cycle;
\fill[line width=1pt,fill=black,fill opacity=0.1] (-2.4142135623730954,1.4142135623730956) -- (-1,4.82842712474619) -- (-2.414213562373095,3.4142135623730954) -- cycle;
\draw [line width=1pt] (-1,0)-- (1,0);
\draw [line width=1pt] (1,0)-- (2.414213562373095,1.414213562373095);
\draw [line width=1pt] (2.414213562373095,1.414213562373095)-- (2.414213562373095,3.4142135623730945);
\draw [line width=1pt] (2.414213562373095,3.4142135623730945)-- (1,4.82842712474619);
\draw [line width=1pt] (1,4.82842712474619)-- (-1,4.82842712474619);
\draw [line width=1pt] (-1,4.82842712474619)-- (-2.414213562373095,3.4142135623730954);
\draw [line width=1pt] (-2.414213562373095,3.4142135623730954)-- (-2.4142135623730954,1.4142135623730956);
\draw [line width=1pt] (-2.4142135623730954,1.4142135623730956)-- (-1,0);
\draw [line width=1pt,color=ccqqqq] (-1,0)-- (1.89,0.89);
\draw [color=ccqqqq](-0.02,1) node[anchor=north west] {$\alpha_1$};
\draw (1.74,1.9) node[anchor=north west] {$S_0$};
\draw (-2.4,3.6) node[anchor=north west] {$S_1$};
\end{tikzpicture}
\caption{In case 2, $\alpha_1$ cannot intersect both $\beta_1$ (which lies in a region $S_0$) and $\beta_l$ (which lies in the symetric region $S_1$). In this example $\beta$ has direction in $\Sigma_2$.}
\label{fig:case_2_octogone}
\end{figure}
\end{enumerate}

It remains to investigate cases $(i)$ and $(ii)$.

\paragraph{\underline{Case $(i)$: $\exists i,j$, $\alpha_i$ and $\beta_j$ are in a configuration of type $(i)$,}} that is $\alpha_i$ is a sandwiched segment of type $e' \to e \to e'$ and $\beta_j$ is a sandwiched segment of type $e \to e' \to e$. This case corresponds to case $(3)$ of \cite[\S 6.3]{BLM22}. Similarly, we consider the maximal sequence of sandwiched segments $\alpha_{i_0} \cup \cdots \cup \alpha_{i_0+p}$ (resp. $\beta_{j_0} \cup \cdots \cup \beta_{j_0+q}$) containing $\alpha_i$ (resp. $\beta_j$), which is maximal in the sense that both $\alpha_{i_0-1}$ and $\alpha_{i_0+p+1}$ (resp. $\beta_{j_0-1}$ and $\beta_{j_0+q+1}$) are non-sandwiched, and we can show:

\begin{Lem}
$\alpha_{i_0} \cup \cdots \cup \alpha_{i_0+p} \cup \alpha_{i_0+p+1}$ and $\beta_{j_0} \cup \cdots \cup \beta_{j_0+q} \cup \beta_{j_0+q+1}$ intersect at most $(p+3)(q+2)$ times while there are $(p+3)(q+3)$ pairs of segments.
\end{Lem}

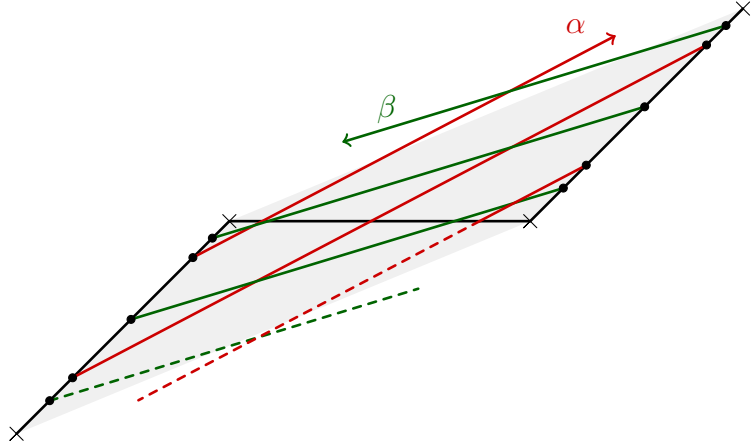
\begin{figure}[h]
\center
\definecolor{qqwuqq}{rgb}{0,0.39215686274509803,0}
\definecolor{ccqqqq}{rgb}{0.8,0,0}
\definecolor{wwwwww}{rgb}{0.4,0.4,0.4}
\begin{tikzpicture}[line cap=round,line join=round,>=triangle 45,x=2cm,y=2cm]
\clip(-1.6,-1.6) rectangle (3.6,1.6);
\fill[line width=1pt,color=wwwwww,fill=wwwwww,fill opacity=0.1] (-1.4142135623730951,-1.4142135623730951) -- (2,0) -- (3.414213562373095,1.4142135623730951) -- (0,0) -- cycle;
\draw [line width=1pt] (-1.4142135623730951,-1.4142135623730951)-- (0,0);
\draw [line width=1pt] (0,0)-- (2,0);
\draw [line width=1pt] (2,0)-- (3.414213562373095,1.4142135623730951);
\draw [line width=1pt,color=ccqqqq] (-1.0418512544474752,-1.0418512544474752)-- (3.1719500830398006,1.1719500830398009);
\draw [line width=1pt,color=ccqqqq, -to] (-0.24226347933329428,-0.24226347933329428)-- (2.5689549765445614,1.2346640450436432);
\draw [line width=1pt,color=qqwuqq,to-] (0.7505168852057164,0.5259057698140497)-- (3.301349526645943,1.3013495266459434);
\draw [line width=1pt,color=qqwuqq] (-0.11286403572715198,-0.11286403572715176)-- (2.7606840230160055,0.760684023016005);
\draw [line width=1pt,color=qqwuqq] (-0.6535295393570895,-0.6535295393570901)-- (2.2200185193860666,0.22001851938606617);
\draw [line width=1pt,dash pattern=on 3pt off 3pt,color=qqwuqq] (-1.1941950429870283,-1.194195042987029)-- (1.2557240594359484,-0.44942863148130907);
\draw [color=ccqqqq](2.1645860786740556,1.417733499785561) node[anchor=north west] {$\alpha$};
\draw [color=qqwuqq](0.9112065662816632,0.9115610043963248) node[anchor=north west] {$\beta$};
\draw [line width=1pt,dash pattern=on 3pt off 3pt,color=ccqqqq] (1.663599166176075,0)-- (-0.6044902242670345,-1.1915842546810187);
\draw [line width=1pt,color=ccqqqq] (1.663599166176075,0)-- (2.37236230792562,0.3723623079256199);
\begin{scriptsize}
\draw [color=black] (0,0)-- ++(-2.5pt,-2.5pt) -- ++(5pt,5pt) ++(-5pt,0) -- ++(5pt,-5pt);
\draw [color=black] (2,0)-- ++(-2.5pt,-2.5pt) -- ++(5pt,5pt) ++(-5pt,0) -- ++(5pt,-5pt);
\draw [color=black] (3.414213562373095,1.4142135623730951)-- ++(-2.5pt,-2.5pt) -- ++(5pt,5pt) ++(-5pt,0) -- ++(5pt,-5pt);
\draw [color=black] (-1.4142135623730951,-1.4142135623730951)-- ++(-2.5pt,-2.5pt) -- ++(5pt,5pt) ++(-5pt,0) -- ++(5pt,-5pt);
\draw [fill=black] (2.37236230792562,0.3723623079256199) circle (1.5pt);
\draw [fill=black] (-1.0418512544474752,-1.0418512544474752) circle (1.5pt);
\draw [fill=black] (3.1719500830398006,1.1719500830398009) circle (1.5pt);
\draw [fill=black] (-0.24226347933329428,-0.24226347933329428) circle (1.5pt);
\draw [fill=black] (3.301349526645943,1.3013495266459434) circle (1.5pt);
\draw [fill=black] (-0.11286403572715198,-0.11286403572715176) circle (1.5pt);
\draw [fill=black] (2.7606840230160055,0.760684023016005) circle (1.5pt);
\draw [fill=black] (-0.6535295393570895,-0.6535295393570901) circle (1.5pt);
\draw [fill=black] (2.2200185193860666,0.22001851938606617) circle (1.5pt);
\draw [fill=black] (-1.1941950429870283,-1.194195042987029) circle (1.5pt);
\end{scriptsize}
\end{tikzpicture}
\caption{Exemple of maximal sequences of sandwiched segments. In this example, $p=1$ and $q=1$. There are $5$ intersections.}
\end{figure}

In particular, summing the intersections of pairs of maximal sequences of sandwiched segments and intersections of pairs of segments which are not already counted in such maximal sequences (and as such intersect at most once), we get:
\[ \sum_{i,j} |\alpha_i \cap \beta_j| < kl \]
so that $|\alpha \cap \beta| + 1 \leq kl$ and $\cfrac{|\alpha \cap \beta| +1}{l(\alpha)l(\beta)} < \cfrac{1}{l_0^2}$, as required.

\paragraph{\underline{Case $(ii)$: $\exists i,j$, $\alpha_i$ and $\beta_j$ are in a configuration of type $(ii)$.}}
In this case, we count the intersection of $\alpha$ with the maximal sequence of non-adjacent segments $\beta_j$ contained in the big cylinder spanned by $e_{\sigma_{\beta}(n/2-1)}$ and $e_{\sigma_{\beta}(n/2-2)}$. Examples are depicted in Figure \ref{fig:non_escape_case} and Figure \ref{fig:case_escape_cylinder} in the case of the octagon.

We distinguish two cases: one where we can get a better estimation of the length of $\beta$ while in the other we can get a better lower bound on the number of intersections between $\alpha$ and $\beta$.
\begin{enumerate}
\item If $\beta$ contains only non-sandwiched segments contained in the big cylinder, then in particular $\beta$ start and end at the vertices which are the common endpoints to the sides of label $e_{\sigma_{\beta}(n/2-1)}$ and $e_{\sigma_{\beta}(n/2-2)}$, see Figure \ref{fig:non_escape_case}). In this case, $\beta = \beta_1 \cup \cdots \cup \beta_l$ could have $l+1$ intersections with $\alpha_i$ (instead of $l$). However, we can compensate with a better estimation of the lengths. A convenient way to estimate the length is to unfold the big cylinder as in Figure \ref{fig:non_escape_case} and get:
\begin{itemize}
\item If $l$ is odd, then $l(\beta) \geq (\frac{l+1}{2} \Phi^2 -1)l_0$.
\item If $l$ is even, then $l(\beta) \geq \frac{l}{2}\Phi^2 l_0$.
\end{itemize}
%\textcolor{blue}{(ne marche pas dans le cas général).}
Hence, in both cases, we have $\cfrac{|\alpha_i \cap \beta|}{l(\alpha_i)l(\beta)} < \cfrac{1}{l_0^2}$. Further, the same holds if $\alpha_i$ is a non-sandwiched segment: in this case $\alpha_i$ intersect $\beta$ at most $l$ times, so that $\cfrac{|\alpha_i \cap \beta| + 1}{l(\alpha_i)l(\beta)} < \cfrac{1}{l_0^2}$. Since $\alpha_1$ is non-sandwiched, this allows to count the singular intersection, so that:
\[
\frac{(\sum_{i} |\alpha_i \cap \beta|) + 1}{(\sum_i l(\alpha_i))l(\beta)} < \frac{1}{l_0^2}.
\]
As required.

\item Else, the maximal sequence $B = \beta_{j_0} \cup \cdots \cup \beta_{j_0+p-1}$ of non-sandwiched segments in the big cylinder containing $\beta_j$ is not contained big cylinder, as in Figure \ref{fig:case_escape_cylinder}. In this case, given that $p$ is the number of non-sandwiched segments in the maximal sequence, there are at most
\begin{itemize}
\item[$\bullet$] $p$ intersections, if $\beta$ either start or end at the central singularity. That is we can assume up to changing the orientation of $\beta$ that the maximal sequence of non-sandwiched segments containing $\beta_j$ is $\beta_1 \cup \cdots \cup \beta_p$, and hence we have 
$$\cfrac{|\alpha_i \cap B|}{l(\alpha_i)l(B)} < \cfrac{p}{p l_0^2} = \cfrac{1}{l_0^2}.$$
Further, using that in this case, $l(\beta_1) \geq (\Phi^2-1)l_0 > 2l_0$, we deduce that in fact $l(B) = l(\beta_1) + \cdots + l(\beta_p)> 2l_0+(p-1)l_0 = (p+1)l_0$,
%\textcolor{blue}{(ne marche pas en général)}
so that $$\cfrac{|\alpha_i \cap B| + 1}{l(\alpha_1)l(B)} < \cfrac{p+1}{(p+1)l_0^2} = \frac{1}{l_0^2}.$$ 
Summing all intersections of this kind with the other intersections, we finally get $\cfrac{\Int(\alpha,\beta)}{l(\alpha)l(\beta)} < \cfrac{1}{l_0^2}$ as required.
\item[$\bullet$] $(p-1)$ intersections, if $\beta$ does not start and does not end at the central singularity, so that
$$\cfrac{|\alpha_i \cap B| +1}{l(\alpha_i)l(B)} < \cfrac{(p-1) + 1}{pl_0^2} = \frac{1}{l_0^2}.$$
Similarly, we conclude that $\cfrac{\Int(\alpha,\beta)}{l(\alpha)l(\beta)} < \cfrac{1}{l_0^2}$ as required.
\end{itemize}

\end{enumerate}

\begin{figure}[h]
\center
\definecolor{ccqqqq}{rgb}{0.8,0,0}
\definecolor{qqwuqq}{rgb}{0,0.39215686274509803,0}
\definecolor{wwwwww}{rgb}{0.4,0.4,0.4}
\definecolor{uuuuuu}{rgb}{0.26666666666666666,0.26666666666666666,0.26666666666666666}
\begin{tikzpicture}[line cap=round,line join=round,>=triangle 45,x=2cm,y=2cm]
\clip(-0.8,-0.8) rectangle (1.8,2.7);
\fill[line width=1pt,color=wwwwww,fill=wwwwww,fill opacity=0.1] (-0.7071067811865475,1.7071067811865477) -- (0,0) -- (1,0) -- (1.7071067811865475,0.7071067811865475) -- (1,2.414213562373095) -- (0,2.414213562373095) -- cycle;
\draw [line width=1pt] (0,0)-- (1,0);
\draw [line width=1pt] (1,0)-- (1.7071067811865475,0.7071067811865475);
\draw [line width=1pt] (1.7071067811865475,0.7071067811865475)-- (1.7071067811865475,1.7071067811865472);
\draw [line width=1pt] (1.7071067811865475,1.7071067811865472)-- (1,2.414213562373095);
\draw [line width=1pt] (1,2.414213562373095)-- (0,2.414213562373095);
\draw [line width=1pt] (0,2.414213562373095)-- (-0.7071067811865475,1.7071067811865477);
\draw [line width=1pt] (-0.7071067811865475,1.7071067811865477)-- (-0.7071067811865477,0.7071067811865478);
\draw [line width=1pt] (-0.7071067811865477,0.7071067811865478)-- (0,0);
\draw [line width=1pt,color=qqwuqq] (0.6545454545454531,0)-- (1.2212121212121203,0.22121212121212033);
\draw [line width=1pt,color=qqwuqq] (0.6545454545454533,2.414213562373095)-- (-0.4191387559808574,1.9950748063922372);
\draw [color=qqwuqq](-0.35,2.05) node[anchor=north west] {$\alpha_i$};
\draw [color=qqwuqq](0.55,0.3) node[anchor=north west] {$\alpha_i$};
\draw [color=ccqqqq](0.648484848484847,1.4666666666666643) node[anchor=north west] {$\beta$};
\draw [line width=1pt,color=ccqqqq] (1,0)-- (-0.19195526603503485,2.22225829633806);
\draw [line width=1pt,color=ccqqqq] (-0.5196434023679111,1.8945701600051836)-- (0.4965493970075056,0);
\draw [line width=1pt,color=ccqqqq] (1.5151515151515125,0.5151515151515123)-- (0.49654939700750583,2.414213562373095);
\draw [line width=1pt,color=ccqqqq] (1.1874633788186364,0.18746337881863595)-- (0,2.414213562373095);
\begin{scriptsize}
\draw [fill=black] (1,0) circle (2.5pt);
\draw [fill=uuuuuu] (0,2.414213562373095) circle (2.5pt);
\draw [color=qqwuqq] (0.6545454545454531,0)-- ++(-2.5pt,-2.5pt) -- ++(5pt,5pt) ++(-5pt,0) -- ++(5pt,-5pt);
\draw [fill=qqwuqq,shift={(1.2212121212121203,0.22121212121212033)}] (0,0) ++(0 pt,3.75pt) -- ++(3.2475952641916446pt,-5.625pt)--++(-6.495190528383289pt,0 pt) -- ++(3.2475952641916446pt,5.625pt);
\draw [color=qqwuqq] (0.6545454545454533,2.414213562373095)-- ++(-2.5pt,-2.5pt) -- ++(5pt,5pt) ++(-5pt,0) -- ++(5pt,-5pt);
\draw [fill=qqwuqq] (-0.4191387559808574,1.9950748063922372) circle (2.5pt);
\draw [fill=qqwuqq] (1.28796802520569,0.28796802520568954) circle (2.5pt);
\end{scriptsize}
\end{tikzpicture}
\definecolor{ccqqqq}{rgb}{0.8,0,0}
\definecolor{zzttqq}{rgb}{0.6,0.2,0}
\begin{tikzpicture}[line cap=round,line join=round,>=triangle 45,x=0.8cm,y=0.8cm]
\clip(-4,-0.3) rectangle (5,9.5);
\fill[line width=1pt,color=wwwwww,fill=wwwwww,fill opacity=0.10000000149011612] (0,0) -- (1,0) -- (1.7071067811865475,0.7071067811865475) -- (1.7071067811865475,1.7071067811865472) -- (1,2.414213562373095) -- (0,2.414213562373095) -- (-0.7071067811865475,1.7071067811865477) -- (-0.7071067811865477,0.7071067811865478) -- cycle;
\draw [line width=1pt] (0,0)-- (1,0);
\draw [line width=1pt] (1,0)-- (1.7071067811865475,0.7071067811865475);
\draw [line width=1pt] (1.7071067811865475,0.7071067811865475)-- (1.7071067811865475,1.7071067811865472);
\draw [line width=1pt] (1.7071067811865475,1.7071067811865472)-- (1,2.414213562373095);
\draw [line width=1pt] (1,2.414213562373095)-- (0,2.414213562373095);
\draw [line width=1pt] (0,2.414213562373095)-- (-0.7071067811865475,1.7071067811865477);
\draw [line width=1pt] (-0.7071067811865475,1.7071067811865477)-- (-0.7071067811865477,0.7071067811865478);
\draw [line width=1pt] (-0.7071067811865477,0.7071067811865478)-- (0,0);
\draw [line width=1pt, dash pattern=on 3pt off 3pt, domain=-6.9999999999999964:0] plot(\x,{(-0--1.7071067811865477*\x)/-0.7071067811865475});
\draw [line width=1pt,dash pattern=on 3pt off 3pt, domain=-6.9999999999999964:1.7071067811865475] plot(\x,{(-3.4142135623730945--1.7071067811865475*\x)/-0.7071067811865475});
\draw [line width=1pt] (-3.414213562373095,8.242640687119286)-- (-2.4142135623730945,8.242640687119286);
\draw [line width=1pt] (-2.4142135623730945,8.242640687119286)-- (-1.707106781186547,8.949747468305834);
\draw [line width=1pt] (-2.414213562373095,5.828427124746191)-- (-1.707106781186547,6.535533905932738);
\draw [line width=1pt] (-1.707106781186547,6.535533905932738)-- (-0.7071067811865472,6.535533905932738);
\draw [line width=1pt] (-1.7071067811865472,4.121320343559643)-- (-0.7071067811865472,4.121320343559643);
\draw [line width=1pt] (-0.7071067811865472,4.121320343559643)-- (0,4.82842712474619);
\draw [line width=1pt,color=ccqqqq] (-3.414213562373095,8.242640687119286)-- (1,0);
\draw [to-to,line width=0.5pt] (0.5,5.02)-- (1.48,2.62);
\draw [to-to,line width=0.5pt] (-0.18,6.66)-- (0.5,5.02);
\draw [line width=0.5pt] (-1.707106781186547,6.535533905932738)-- (-2.4142135623730945,8.242640687119286);
\draw [line width=0.5pt] (0,2.414213562373095)-- (-0.7071067811865472,4.121320343559643);
\draw (0.45,6.3) node[anchor=north west] {$l \geq \Phi l_0$};
\draw (1,4.5) node[anchor=north west] {$l \geq (\Phi^2-1) l_0$};
\end{tikzpicture}
\caption{In the case where the curve $\beta$ stays in the big cylinder, there could be one intersection more (here $5$) than non sandwiched segments in $\beta$ (here $4$). However, unfolding the trajectory of the curve $\beta$ allows to estimate precisely its length, given the lengths of the long diagonals which can be expressed using $\Phi$.}
\label{fig:non_escape_case}
\end{figure}
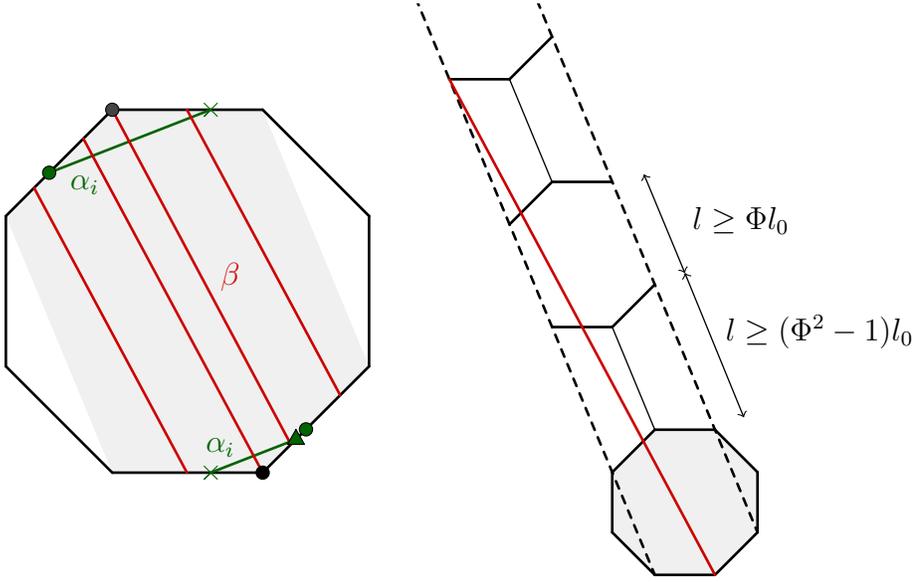

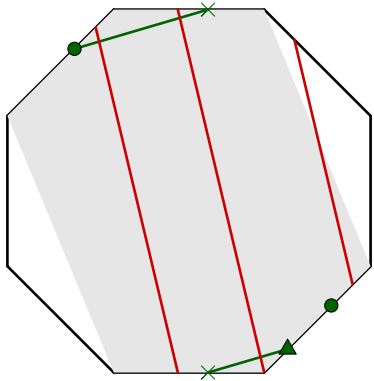
\begin{figure}[h]
\center
\definecolor{ccqqqq}{rgb}{0.8,0,0}
\definecolor{qqwuqq}{rgb}{0,0.39215686274509803,0}
\begin{tikzpicture}[line cap=round,line join=round,>=triangle 45,x=2cm,y=2cm]
\clip(-0.8,-0.2) rectangle (1.8,2.7);
\draw [line width=1pt] (0,0)-- (1,0);
\draw [line width=1pt] (1,0)-- (1.7071067811865475,0.7071067811865475);
\draw [line width=1pt] (1.7071067811865475,0.7071067811865475)-- (1.7071067811865475,1.7071067811865472);
\draw [line width=1pt] (1.7071067811865475,1.7071067811865472)-- (1,2.414213562373095);
\draw [line width=1pt] (1,2.414213562373095)-- (0,2.414213562373095);
\draw [line width=1pt] (0,2.414213562373095)-- (-0.7071067811865475,1.7071067811865477);
\draw [line width=1pt] (-0.7071067811865475,1.7071067811865477)-- (-0.7071067811865477,0.7071067811865478);
\draw [line width=1pt] (-0.7071067811865477,0.7071067811865478)-- (0,0);
\fill[gray!20] (0,0)--(1,0)--(1.7071067811865475,0.7071067811865475)--(1,2.414213562373095)-- (0,2.414213562373095)--(-0.7071067811865475,1.7071067811865477)--(0,0);
\draw [line width=1pt,color=qqwuqq] (0.6266666666666643,0)-- (1.155555555555555,0.155555555555555);
\draw [line width=1pt,color=qqwuqq] (-0.26111111111110763,2.1531024512619874)-- (0.6266666666666645,2.414213562373095);
\draw [line width=1pt,color=ccqqqq] (1,0)-- (0.42583832171184927,2.414213562373095);
\draw [line width=1pt,color=ccqqqq] (0.42583832171184904,0)-- (-0.11982573479760647,2.294387827575488);
\draw [line width=1pt,color=ccqqqq] (1.587281046388941,0.5872810463889404)-- (1.200465666653539,2.2137478957195564);
\begin{scriptsize}
\draw [color=qqwuqq] (0.6266666666666643,0)-- ++(-2.5pt,-2.5pt) -- ++(5pt,5pt) ++(-5pt,0) -- ++(5pt,-5pt);
\draw [fill=qqwuqq,shift={(1.155555555555555,0.155555555555555)}] (0,0) ++(0 pt,3.75pt) -- ++(3.2475952641916446pt,-5.625pt)--++(-6.495190528383289pt,0 pt) -- ++(3.2475952641916446pt,5.625pt);
\draw [color=qqwuqq] (0.6266666666666645,2.414213562373095)-- ++(-2.5pt,-2.5pt) -- ++(5pt,5pt) ++(-5pt,0) -- ++(5pt,-5pt);
\draw [fill=qqwuqq] (-0.26111111111110763,2.1531024512619874) circle (2.5pt);
\draw [fill=qqwuqq] (1.44599567007544,0.4459956700754397) circle (2.5pt);
\end{scriptsize}
\end{tikzpicture}
\caption{In the case where the curve $\beta$ leaves the big cylinder, there are at most as many non-sandwiched segments as intersections with the segment $\alpha_i$.}
\label{fig:case_escape_cylinder}
\end{figure}

\subsection{Conclusion}
As shown in the previous section, for any two distinct saddle connections $\alpha$ and $\beta$, we have
\begin{equation}\label{eq:sc_case}
\tag{$4$}
\frac{|\alpha \cap \beta| + 1}{l(\alpha)l(\beta)} \leq \frac{1}{l_0^2}
\end{equation}
and equality occurs if and only if $\alpha$ and $\beta$ are sides of the regular $n$-gon.

Now, any closed curve $\eta$ (resp. $\xi$) on the regular $n$-gon is homologous to a union of saddle connections $\eta = \eta_1 \cup \cdots \cup \eta_k$ (resp. $\xi = \xi_1 \cup \cdots \cup \xi_l)$. In this case, we have
\[
\Int(\eta,\xi) \leq (\sum_{i,j} |\eta_{i} \cap \xi_j|) + s
\]
where $s$ is the number singular intersection points. It should be noted that we set $|\eta_i \cap \xi_j| = 0$ if $\eta_i = \xi_j$. Further, we can assume without loss of generality that $\eta$ and $\xi$ are simple closed curves (see Lemma 3.1 of \cite{MM}) so that there are no multiple intersections at the singularities, and hence $s \leq \min(k,l)$. Further, $s \leq 1$ for $n \equiv 0 \mod 4$ and $s \leq 2$ for $n \equiv 2 \mod 4$. Using Equation \eqref{eq:sc_case}, we get:
\[ 
\Int(\eta,\xi) \leq (\sum_{i,j} l(\eta_{i}) l(\xi_j)) \times \frac{1}{l_0^2} + s - kl = \frac{l(\eta)l(\xi)}{l_0^2} + s-kl
\]    
with equality if and only if each $\eta_i$ (resp. $\xi_j$) is a side of the regular $n$-gon. In particular:
\begin{itemize}
\item For $n \equiv 0 \mod 4$, $s \leq 1$ and we get directly Theorem \ref{theo:KVol_4m}, by noticing that distinct sides of the $n$-gon are indeed intersecting once at the singularity.
\item For $n \equiv 2 \mod 4$, if $s \leq 1$ then at least one of the $\eta_i$ (or $\xi_j$) is not a side of the $n$-gon so that $l(\eta_i) > l_0$ (or $l(\xi_j) > l_0$) and
\[ 
\Int(\eta,\xi) < \frac{l(\eta)l(\xi)}{l_0^2}
\]   
Else, if $s=2$ then $k,l \geq 2$ so that $kl \geq 4$ and 
\[
\Int(\eta,\xi) \leq \frac{l(\eta)l(\xi)}{l_0^2} -2 < \frac{l(\eta)l(\xi)}{l_0^2}
\]
This gives Theorem \ref{theo:4m+2}.
\end{itemize}

\section{KVol as a supremum over pairs of directions}\label{sec:directions}
In this section, we extend the study of KVol as a function over the Teichm\"uller disk. We assume $n \equiv 0 \mod 4$ so that $\XX$ has a single singularity (as well as all the translation surfaces in its Teichm\"uller disk). The study follows the method of \cite{BLM22}, but the case of the regular $n$-gon requires more precise estimates.

We first give a consistent name for saddle connections as the surface varies in the Teichm\"uller disk, as well as the direction of such saddle connections. This is done in \S \ref{subsec:directions} by choosing a base surface in the Teichm\"uller disk, namely $\SS$, the staircase model. Next, for each pair of distinct periodic directions we define a quantity $K(d,d')$ which can be computed on the base surface $\SS$, and allows for a more convenient expression of KVol (see Proposition \ref{prop:sup_directions}). Finally, we provide in Proposition \ref{prop:etude_K} precise estimates on $K(d,d')$. These estimates are one of the main ingredients in the proof of Theorem \ref{theo:main}, and they differ from the case of the double $n$-gon.

\subsection{Directions in the Teichm\"uller disk}\label{subsec:directions}
Following \cite[\S 4]{BLM22}, we consider the plane template of $\SS$ as our base surface and define the direction of a saddle connection $\alpha$ in $X = M \cdot \SS$ as the direction (in $\RR P^1$) of the preimage saddle connection $M^{-1} \cdot \alpha$ in $\SS$ it corresponds to. More precisely:

\begin{Def}\label{def:directions}
For $d \in \RR P^1$, we say that a saddle connection in $\SS$  has direction $d$ if it has direction $d$ in the plane template of Figure \ref{staircase_model}. For $M \in \mathrm{GL}_2^+ (\RR)$ we say that a saddle connection $\alpha$ in $M\cdot \SS$ has direction $d$ if $M^{-1}\cdot \alpha$ has direction $d$ in $\SS$.
\end{Def}
This is a bit counter-intuitive because $\alpha$ may not have direction $d$ in a plane template for $M\cdot \SS$, but it allows for a consistent choice of the notion of direction along the Teichm\"uller space. Moreover, we have:\medskip

\begin{Prop}\cite[\S 4]{BLM22}\label{prop:directions}
Using the identifications 
\[ d = [x:y] \in \RR P^1 \mapsto -\frac{x}{y} \in \RR \cup \{ \infty \} \equiv \partial \HH \]
and for $M = \begin{pmatrix} a & b \\ c & d \end{pmatrix} \in SL_2(\RR)$
\[ M \cdot \SS \in \TT 
\mapsto \frac{di+b}{ci+a} \in \HH, \]
the locus of surfaces in $ \TT$ where the directions $d$ and $d'$ make an (unoriented) angle $\theta \in ]0, \frac{\pi}{2}]$ is the banana neighborood
\[
\gamma_{d,d',r}= \{ z \in \Hyp^2: \mathrm{dist}_{\HH}(z, \gamma_{d,d'} )=r \}
\]
where $\cosh r = \cfrac{1}{\sin \theta}$.\newline

In particular, the locus of surfaces in $\TT$ where the directions $d$ and $d'$ are orthogonal is the hyperbolic geodesic with endpoints $d$ and $d'$. \newline
\end{Prop}

In the rest of the paper, we use the following 

\begin{Nota}\label{nota:sinus}
Given $X = M \cdot \SS \in \TT$, and $d,d'$ distinct periodic directions, we define
$\theta(X,d,d') \in ]0, \frac{\pi}{2}]$ as the (unoriented) angle between the directions $d$ and $d'$ in the surface $X$. With this notation, we have by Proposition \ref{prop:directions}:
\begin{equation}\label{eq:lien_sin_cosh}
\sin \theta(X,d,d') = \cfrac{1}{\cosh(\mathrm{dist}_{\HH}(X ,\gamma_{d,d'}))}.
\end{equation}
\end{Nota}

\subsection{KVol as a supremum over pairs of directions}
With the above choice of a consistent name for saddle connections and directions along the Teichm\"uller disk, we have the following proposition, which is already stated in \cite{BLM22} in the case of the double $n$-gon but can be extended with the same proof to the case of translation surfaces $S$ for which saddle connections in the same direction do not intersect. 

\begin{Prop}\cite[Proposition 5.1]{BLM22}\label{prop:sup_directions}
Let $S$ be a translation surface with a single singularity such that saddle connections in the same direction are non-intersecting. Let $\mathcal{P}$ be the set of directions of saddle connections of $S$. Define the notion of direction for $X$ in the Teichm\"uller disk of $S$ using Definition \ref{def:directions}. Then, for any surface $X$ in the Teichm\"uller disk of $S$, we have:
\begin{equation}\label{eq:reformulation}
\KVol(X) = \mathrm{Vol}(X) \cdot \sup_{
 \begin{scriptsize}
 \begin{array}{c}
d, d' \in \mathcal{P} \\
d \neq d'
\end{array}
\end{scriptsize}} K(d,d')\cdot \sin \theta(X,d,d'),
\end{equation}
where
$K(d,d') = \sup_{
 \begin{scriptsize}
 \begin{array}{c}
\alpha\subset \SS \ {\mathrm saddle\ connection\ in\ direction\ } d \\
\beta\subset \SS \ {\mathrm saddle\ connection\ in\ direction\ } d'
\end{array}
\end{scriptsize}
}
\frac{\mathrm{Int} (\alpha,\beta)}{\alpha\wedge \beta}
$ and $\theta(X,d,d')$ is the angle given by Notation \ref{nota:sinus}.
\end{Prop}
The fact that saddle connections in the same direction on the regular $n$-gon are non-intersecting can be checked using the horizontal and vertical cylinder decomposition of the staircase model.

\begin{Rema}\label{rk:invariance_K}
Notice that the Veech group $\Gamma_n^{\pm}$ of the staircase model $\SS$ acts on $\SS$ while preserving the intersection form, and acts linearly on $\RR^2$, hence preserving the wedge product. In particular, $K(d,d') = K(g \cdot d, g \cdot d')$ for any element $g \in \Gamma_n^{\pm}$.
\end{Rema}

From this result and Theorem \ref{theo:KVol_4m}, we deduce

\begin{Cor}\label{cor:K_regular}
For $n \geq 8, n \equiv 0 \mod 4$, we have:
\[
\KVol(\XX) = \mathrm{Vol}(X) \cdot K(\infty, \frac{1}{\Phi}) \cdot \sin \theta(\XX,\infty, \frac{1}{\Phi})
\]
In particular, 
\[ \forall (d,d') \text{, } K(d, d') \cdot \sin \theta(\XX,d,d') \leq K(\infty, \frac{1}{\Phi}) \cdot \sin \theta(\XX,\infty, \frac{1}{\Phi}).
\]
\end{Cor}
\begin{proof}
By Theorem \ref{theo:KVol_4m}, $\KVol$ is achieved on the regular $n$-gon by pairs of distinct sides. In particular, the sides of the $n$-gon $\alpha$ and $\beta$ corresponding to directions $\infty$ and $\frac{1}{\Phi}$ achieve the supremum:
\begin{align*}
\KVol(\XX) &= \frac{\Int(\alpha,\beta)}{l(\alpha)l(\beta)} \text{ by Theorem \ref{theo:KVol_4m}.} \\ 
& = \frac{\Int(\alpha,\beta)}{\alpha \wedge \beta} \sin angle(\alpha,\beta)\\
 &= K(\infty, \frac{1}{\Phi}) \sin \theta(\XX,\infty,\frac{1}{\Phi}) \text{ by definition of } K(d,d') \text{ and } \theta(X,d,d').
\end{align*}
\end{proof}

Next, we provide precise estimates on $K(d,d')$ using its invariance under the diagonal action of the Veech group. These estimates are one of the main ingredients in the proof of Theorem \ref{theo:main}.

\begin{Prop}\label{prop:etude_K}
For any pair of distinct periodic directions $(d,d')$, we have, with the notations of Figure \ref{staircase_model}:
\begin{itemize}
\item[\ding{171}] If there exist $k \in \NN^{*} \cup \{ \infty \}$ and $g \in \Gamma_n^{\pm}$ such that $(d,d') = (g \cdot \infty, \pm g \cdot \frac{1}{k\Phi})$, we have
\[ 
K(d,d') = \frac{1}{l(\alpha_1)l(\alpha_m)} = \frac{1}{\Phi l_m^2}
\]
with $l_m := l(\alpha_m)$.
\item[\ding{169}] If there exist $g \in \Gamma_n^{\pm}$ such that $(d,d') = (g \cdot \infty, \pm g \cdot \frac{\Phi^2 -1}{\Phi^3 - 2 \Phi})$, then
\[
K(d,d') = \frac{1}{(\Phi^3-2\Phi)l_m^2} = \frac{1}{\Phi^2-2}K(\infty, \frac{1}{\Phi}).
\]
This is the case of $(d,d') =(\frac{1}{\Phi}, \Phi - \frac{1}{\Phi})$, image of $(\infty, \frac{\Phi^2 -1}{\Phi^3 - 2 \Phi})$ by the element $g = T_V R \in \Gamma^{\pm}$.
\item In the other cases, 
\[ K(d,d') < \frac{1}{(\Phi^3-2\Phi)l_m^2}. \]
\end{itemize}
\end{Prop}

Recall that the $\alpha_i$ (resp. $\beta_j$) are the horizontal (resp. vertical) saddle connection sorted by decreasing length, and define $l_i := l(\alpha_i)$ and $h_j := l(\beta_j)$ for $1 \leq i,j \leq m$. Further, we denote $C_1, \dots, C_m$ (resp. $Z_1, \dots, Z_m$) the vertical (resp. horizontal) cylinders, as in Figure \ref{staircase_model}. Using that the moduli of both horizontal and vertical cylinders are all $1/\Phi$ (except the horizontal cylinder $Z_1$), one can compare the lengths $h_m, h_{m-1}$ and $l_{m-1}$ to $l_m$, as in Figure \ref{lengths_staircase}.

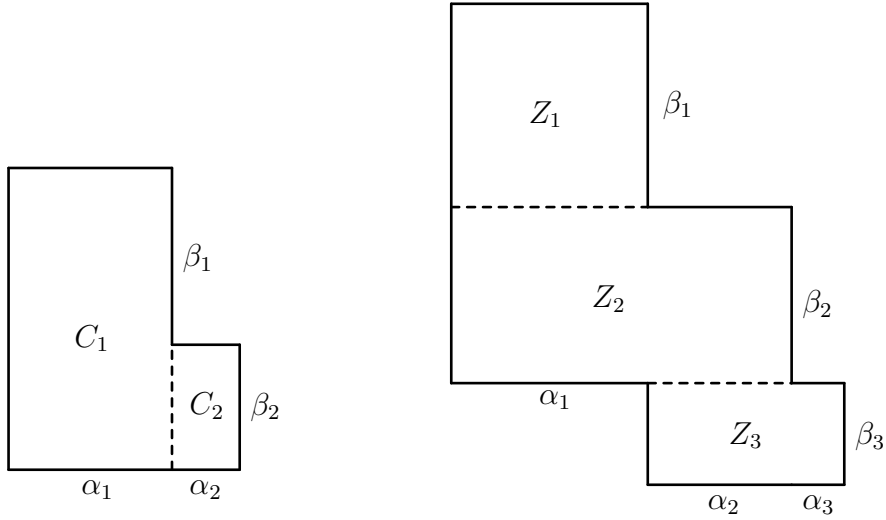
\begin{figure}[h]
\center
\begin{tikzpicture}[line cap=round,line join=round,>=triangle 45,x=1cm,y=1cm, scale=0.9]
\clip(-3.5,-1) rectangle (2,5);
\draw [line width=1pt] (0,0)-- (1,0);
%\draw [line width=1pt, color = red] (-2.414213562373095,1.8477590650225735)-- (0,4.460884994775325);
%\draw (-1.9,3.7) node[anchor=north west, color = red] {$\gamma_1$};
\draw [line width=1pt] (1,0)-- (1,1.8477590650225735);
\draw [line width=1pt] (1,1.8477590650225735)-- (0,1.8477590650225735);
\draw [line width=1pt] (0,1.8477590650225735)-- (0,4.460884994775325);
\draw [line width=1pt] (0,4.460884994775325)-- (-2.414213562373095,4.460884994775325);
\draw [line width=1pt] (-2.414213562373095,4.460884994775325)-- (-2.414213562373095,0);
\draw [line width=1pt] (-2.414213562373095,0)-- (0,0);
\draw [line width=1pt,dash pattern=on 3pt off 3pt] (0,0)-- (0,1.8477590650225735);
%\draw [line width=1pt,dash pattern=on 3pt off 10pt] (-2.414213562373095,1.8477590650225735)-- (0,1.8477590650225735);
\draw (0.1,0) node[anchor=north west] {$\alpha_2$};
\draw (-1.5,0) node[anchor=north west] {$\alpha_1$};
\draw (1,1.3) node[anchor=north west] {$\beta_2$};
\draw (0,3.5) node[anchor=north west] {$\beta_1$};
\draw (0.1,1.3) node[anchor=north west] {$C_2$};
\draw (-1.6,2.3) node[anchor=north west] {$C_1$};
\end{tikzpicture}
\begin{tikzpicture}[line cap=round,line join=round,>=triangle 45,x=1cm,y=1cm, scale = 0.7]
\clip(-9,-1) rectangle (2,10);
%\draw [line width=1pt, color = red] (-2.7320508075688776, 1.9318516525781366)-- (0,5.277916867529369);
%\draw (-2.1,4.3) node[anchor=north west, color = red] {$\gamma_1$};
\draw [line width=1pt] (0,0)-- (1,0);
\draw [line width=1pt] (1,0)-- (1,1.9318516525781366);
\draw [line width=1pt] (1,1.9318516525781366)-- (0,1.9318516525781366);
\draw [line width=1pt] (0,1.9318516525781366)-- (0,5.277916867529369);
\draw [line width=1pt] (0,5.277916867529369)-- (-2.7320508075688776,5.277916867529369);
\draw [line width=1pt] (-2.7320508075688776,0)-- (0,0);
%\draw [line width=1pt,dash pattern=on 3pt off 3pt] (0,0)-- (0,1.9318516525781366);
\draw (-1.8,0) node[anchor=north west] {$\alpha_2$};
\draw (-5,2) node[anchor=north west] {$\alpha_1$};
\draw (0,3.9) node[anchor=north west] {$\beta_2$};
\draw (-2.65,7.7) node[anchor=north west] {$\beta_1$};
%\draw [line width=1pt,dash pattern=on 3pt off 10pt] (0,1.9318516525781366)-- (-2.7320508075688776,1.9318516525781368);
\draw [line width=1pt] (-2.7320508075688776,1.9318516525781368)-- (-6.464101615137755,1.9318516525781366);
\draw [line width=1pt] (-6.464101615137755,1.9318516525781366)-- (-6.464101615137755,9.141620172685643);
\draw [line width=1pt] (-6.464101615137755,9.141620172685643)-- (-2.7320508075688767,9.141620172685643);
\draw [line width=1pt] (-2.7320508075688767,9.141620172685643)-- (-2.7320508075688776,5.277916867529369);
\draw [line width=1pt] (-2.7320508075688776,1.9318516525781368)-- (-2.7320508075688776,0);
\draw [line width=1pt,dash pattern=on 3pt off 3pt] (-6.4641,5.2779)-- (-2.7320508075688776,5.277916867529369);
\draw [line width=1pt,dash pattern=on 3pt off 3pt] (-2.7320508075688776,1.9318516525781368)-- (0,1.9318);
%\draw [line width=1pt,dash pattern=on 3pt off 10pt] (-6.464101615137756,5.277916867529369)-- (-2.7320508075688776,5.277916867529369);
\draw (0,0) node[anchor=north west] {$\alpha_3$};
\draw (1,1.4) node[anchor=north west] {$\beta_3$};
\draw (-1.4,1.4) node[anchor=north west] {$Z_3$};
\draw (-4,4) node[anchor=north west] {$Z_2$};
\draw (-5.2,7.5) node[anchor=north west] {$Z_1$};
\end{tikzpicture}
\caption{The staircase models associated to the $n$-gon for $n=8$ on the left and $n=12$ on the right, and the cylinders.}
\label{staircase_model}
\end{figure}

\begin{figure}[h]
\center
\definecolor{ccqqqq}{rgb}{0.8,0,0}
\begin{tikzpicture}[line cap=round,line join=round,>=triangle 45,x=1cm,y=1cm]
\clip(-8,-1.3) rectangle (6,6);
\draw [line width=1pt] (0,0)-- (1,0);
\draw [line width=1pt] (1,0)-- (1,1.9318516525781366);
\draw [line width=1pt] (1,1.9318516525781366)-- (0,1.9318516525781366);
\draw [line width=1pt] (0,1.9318516525781366)-- (0,5.277916867529369);
\draw [line width=1pt] (0,5.277916867529369)-- (-2.7320508075688776,5.277916867529369);
\draw [line width=1pt] (-2.7320508075688776,0)-- (0,0);
\draw [line width=1pt,dash pattern=on 3pt off 3pt] (0,0)-- (0,1.9318516525781366);
\draw [line width=1pt,dash pattern=on 3pt off 3pt] (0,1.9318516525781366)-- (-2.7320508075688776,1.9318516525781368);
\draw [line width=1pt] (-2.7320508075688776,1.9318516525781368)-- (-6.5,1.9318516525781366);
\draw [line width=1pt] (-6.5,1.9318516525781366)--(-6.5,6);
\draw [line width=1pt] (-2.7320508075688767,6)-- (-2.7320508075688776,5.277916867529369);
\draw [line width=1pt,dash pattern=on 3pt off 3pt] (-2.7320508075688776,1.9318516525781368)-- (-2.7320508075688776,5.277916867529369);
\draw [line width=1pt] (-2.7320508075688776,1.9318516525781368)-- (-2.7320508075688776,0);
\draw [line width=1pt,dash pattern=on 3pt off 3pt] (-6.5,5.277916867529369)-- (-2.7320508075688776,5.277916867529369);
\draw[line width=1pt, color = ccqqqq] (0,1.9318516525781366)--(-2.7320508075688776,0);
\draw[line width=1pt, color = ccqqqq] (1,1.9318516525781366)--(0,0);
\draw [line width=1pt,color=ccqqqq] (-2.732,1.931)--(0,5.277);
\draw [line width=0.5pt] (-2.7,-0.3)-- (0,-0.3);
\draw [line width=0.5pt] (1,-0.3)-- (0,-0.3);
\draw [line width=0.5pt] (1.5,0)-- (1.5,1.9);
\draw [line width=0.5pt] (1.5,1.9)-- (1.5,5.3);
\draw [line width=0.5pt] (-2.7,-0.3)-- (-2.55,-0.15);
\draw [line width=0.5pt] (-2.7,-0.3)-- (-2.55,-0.45);
\draw [line width=0.5pt] (0,-0.3)-- (-0.15,-0.15);
\draw [line width=0.5pt] (0,-0.3)-- (-0.15,-0.45);
\draw [line width=0.5pt] (0,-0.3)-- (0.15,-0.45);
\draw [line width=0.5pt] (0,-0.3)-- (0.15,-0.15);
\draw [line width=0.5pt] (1,-0.3)-- (0.85,-0.15);
\draw [line width=0.5pt] (1,-0.3)-- (0.85,-0.45);
\draw [line width=0.5pt] (1.5,0)-- (1.65,0.15);
\draw [line width=0.5pt] (1.5,0)-- (1.35,0.15);
\draw [line width=0.5pt] (1.5,1.9)-- (1.65,1.75);
\draw [line width=0.5pt] (1.5,1.9)-- (1.35,1.75);
\draw [line width=0.5pt] (1.5,1.9)-- (1.35,2.05);
\draw [line width=0.5pt] (1.5,1.9)-- (1.65,2.05);
\draw [line width=0.5pt] (1.5,5.3)-- (1.35,5.15);
\draw [line width=0.5pt] (1.5,5.3)-- (1.65,5.15);
\draw (1.6,4) node[anchor=north west] {$h_{m-1} = (\Phi^3- 2 \Phi) l_m$};
\draw (1.6,1.4) node[anchor=north west] {$h_m = \Phi l_m$};
\draw (0.25,-0.3) node[anchor=north west] {$l_m$};
\draw (-3,-0.65) node[anchor=north west] {$l_{m-1} = (\Phi^2 - 1)l_m$};
%\draw (-2.5,2.5) node[anchor=north west, color = ccqqqq] {co-slope};
\draw (-1.5,1.75) node[anchor=north west, color = ccqqqq] {$\gamma$};
\draw (0.05,1.55) node[anchor=north west, color = ccqqqq] {$\gamma'$};
\draw (-2.1,4) node[anchor=north west, color = ccqqqq] {$\gamma''$};
\end{tikzpicture}
\caption{The lengths $h_m, h_{m-1}$ and $l_{m-1}$ expressed using $l_m$. The diagonal curves $\gamma, \gamma'$ and $\gamma''$ have respective co-slope $\Phi - \frac{1}{\Phi}$, $\frac{1}{\Phi}$, and $\frac{\Phi^2-1}{\Phi^3-2\Phi}$.}
\label{lengths_staircase}
\end{figure}
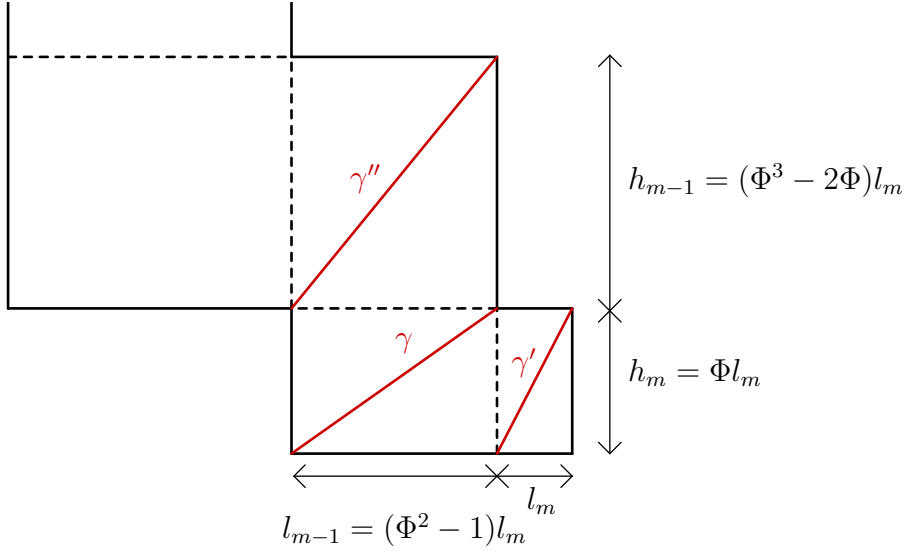

\begin{proof}
The proof goes as follow: given a pair of distinct periodic directions $(d,d')$ and a pair of saddle connections $(\alpha, \beta)$ in respective directions $d,d'$, we first notice that it is possible to assume that $\alpha$ is either horizontal ($d=\infty$, i.e $\alpha = \alpha_i$ for $i \in \llbracket 1 ,m\rrbracket$) or vertical ($d=0$, i.e $\alpha=\beta_j$ for $j \in \llbracket 1,m \rrbracket$). This is because $K(d,d')$ is by invariant under the action of the Veech group and any periodic direction is either the image of the horizontal or the vertical by an element of the Veech group. We study the cases $d=\infty$ and $d=0$ separately, identify the configurations \ding{171} and \ding{169} and show that $K(d,d')$ is smaller in the other cases. \newline

\paragraph{\underline{Case 1: $d$ represents the horizontal cusp.}} Up to the action by an element of the Veech group, we can assume $\alpha$ is one of the $\alpha_i$. Let us first study the intersection with $\alpha_i$ for $i<m$.
\paragraph{\underline{Case 1.1: $\alpha = \alpha_i$ for $i<m$.}}
\begin{Lem}\label{cas_alpha_i}
For any $i<m$, and any saddle connection $\beta$, we have 
$$
\frac{\mathrm{Int}(\alpha_i, \beta)}{\alpha_i \wedge \beta} < \frac{1}{\Phi^3-2\Phi}\cdot \frac{1}{l_m^2}.
$$
\end{Lem}
\begin{proof}
First notice that if a saddle connection $\beta$ does not intersect the core curves of the cylinders $Z_1, \dots, Z_m$, we have $\mathrm{Int}(\alpha_i, \beta) = 0$ since the core curves are respectively homologous to $\alpha_1, \alpha_1 + \alpha_2, \cdots, \alpha_{i-1}+\alpha_i$. In particular, a saddle connection $\beta$ having non-zero intersection with $\alpha$ has a vertical length at least $h_i$. Further, any non singular intersection of $\beta$ with $\alpha_i$ for $i<m$ requires a vertical length at least $h_i+h_{i+1}$ (see Figure \ref{fig:lemme_4_7}). This gives:
$$
\alpha_i \wedge \beta \geq l_i \mathrm{max}((h_i+h_{i+1})(\mathrm{Int}(\alpha_i, \beta)-1), h_i)
$$

\begin{figure}
\center
\definecolor{qqwuqq}{rgb}{0,0.39215686274509803,0}
\definecolor{ccqqqq}{rgb}{0.8,0,0}
\begin{tikzpicture}[line cap=round,line join=round,>=triangle 45,x=1cm,y=1cm, scale = 0.8]
\clip(-4,-0.5) rectangle (1,6);
\draw [line width=1pt,color=ccqqqq] (-3,0)-- (0,0);
\draw [line width=1pt,dash pattern=on 3pt off 3pt] (1,2)-- (0,2);
\draw [line width=1pt] (0,2)-- (0,5.58);
\draw [line width=1pt,color=ccqqqq] (0,5.58)-- (-2.973333333333335,5.59);
\draw [line width=1pt,dash pattern=on 3pt off 3pt] (-2.973333333333335,5.59)-- (-4,5.59);
\draw [line width=1pt,dash pattern=on 3pt off 3pt] (-4,2)-- (-3,2);
\draw [line width=1pt] (-3,2)-- (-3,0);
\draw [line width=1pt,color=qqwuqq] (-1.8000135734152098,-0.004035828788961915)-- (-1.2147591025791868,1.3466186656662082);
\draw [line width=1pt,dash pattern=on 3pt off 3pt,color=qqwuqq] (-1.2147591025791868,1.3466186656662082)-- (-1.0021966769673205,1.8371717666016807);
\draw [line width=1pt,color=qqwuqq] (-1.7733469067485448,5.585964171211038)-- (-2.3438978186802184,4.269242656183222);
\draw [line width=1pt,dash pattern=on 3pt off 3pt,color=qqwuqq] (-2.3438978186802184,4.269242656183222)-- (-2.6266666666666683,3.61666666666666);
\draw [color=ccqqqq](-1,0) node[anchor=north west] {$\alpha_i$};
\draw [color=qqwuqq](-2.053333333333335,4.976666666666658) node[anchor=north west] {$\beta$};
\draw (0.15,4.136666666666659) node[anchor=north west] {$h_i$};
\draw (-0.1,1.4) node[anchor=north west] {$h_{i+1}$};
\draw [-to,line width=0.5pt] (0.5173091738313147,1.2565125882099142) -- (0.5110264706788689,1.90991371606422);
\draw [-to,line width=0.5pt] (0.49525623247357337,3.5500184894148097) -- (0.508720224338232,2.1497633354904315);
\draw [-to,line width=0.5pt] (0.49051431388862954,4.043178022248922) -- (0.4766682074512304,5.483173091738302);
\draw [-to,line width=0.5pt] (0.5220498597886045,0.7634812486518121) -- (0.5285815537271586,0.08418507904224826);
\begin{scriptsize}
\draw [fill=qqwuqq] (-1.7733469067485448,5.585964171211038) circle (2.5pt);
\draw [fill=qqwuqq] (-1.8000135734152098,-0.004035828788961915) circle (2.5pt);
\end{scriptsize}
\end{tikzpicture}
\caption{A non singular intersection with $\alpha_i$, $i<m$, requires a vertical length at least $h_i + h_{i+1}$.}
\label{fig:lemme_4_7}
\end{figure}
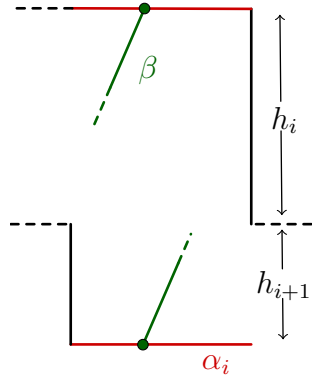

In particular,
\begin{enumerate}[label =(\roman*)]
\item If $\Int(\alpha_i,\beta) \leq 1$, we use $\alpha_i \wedge \beta \geq l_i h_i$ to get
\[  \frac{\mathrm{Int}(\alpha_i,\beta)}{\alpha_i \wedge \beta} \leq \frac{1}{l_i h_i},\]
and, since $i<m$, we get
\[ \frac{1}{l_i h_i} \leq \frac{1}{l_{m-1} h_{m-1}} \leq \frac{1}{(\Phi^2-1)(\Phi^3-2\Phi) l_m^2} < \frac{1}{(\Phi^3-2\Phi) l_m^2}. \]
\item Otherwise, we use $\alpha_i \wedge \beta \geq l_i (h_i+h_{i+1})(\Int(\alpha_i,\beta)-1)$ and obtain
\[ \frac{\mathrm{Int}(\alpha_i,\beta)}{\alpha_i \wedge \beta} \leq \frac{2}{l_i(h_i+h_{i+1})}. \]
But, since $i<m$, we have $l_i \geq l_{m-1} = (\Phi^2-1)l_m$ and $h_i+h_{i+1} \geq h_{m-1} + h_m = (\Phi^3 - \Phi)l_m$ and in particular 
$$
\frac{2}{l_i (h_i+h_{i+1})} \leq \frac{2}{(\Phi^3-\Phi)(\Phi^2-1)l_m^2} < \frac{1}{(\Phi^3-2\Phi)l_m^2},
$$
where the last inequality comes from $\Phi^3- \Phi > \Phi^3 - 2\Phi$ and $\Phi^2-1 > 2$.
\end{enumerate}

\end{proof}
 
\paragraph{\underline{Case 1.2: $\alpha = \alpha_m$.}}
In this case, notice that either $\beta$ is contained in the horizontal cylinder $Z_m$ (see Figure \ref{staircase_model}) and, up to a horizontal twist, its co-slope is $\pm \frac{1}{k\Phi}$ for $k \in \NN^* \cup \{ \infty \}$ so that it corresponds to a geodesic of case \ding{171}, or $\beta$ is not contained in $Z_m$.\newline

\underline{\textbf{Case 1.2.i.}} If the direction of $\beta$ is, up to an horizontal twist, $d' = \pm \frac{1}{k\Phi}$ as in \ding{171}, and $\beta$ is contained in $Z_m$, we see that $\alpha$ and $\beta$ intersect $k+1$ times and $\alpha \wedge \beta = (k+1) l(\beta_m)l(\alpha_m) = (k+1) \Phi l_m^2$ (see Figure \ref{fig:plane_template_direction_3Phi} for the case $k=3$). This gives directly that 
\[ \frac{\Int(\alpha_m, \beta)}{l(\alpha) l(\beta)} = \frac{k+1}{(k+1) \Phi l_m^2} = \frac{1}{\Phi l_m^2}.
\]
In particular, $K(\infty, \frac{1}{k\Phi}) \geq \frac{1}{\Phi l_m^2}$, and the reversed inequality holds since the other saddle connections of direction $d' = \frac{1}{k\Phi}$ do not intersect $\alpha_m$, and as seen in Case 1.1 the intersection of any $\beta$ with $\alpha_i$, $i<m$, gives a lower ratio. Consequently,
$$\text{(\ding{171}) } \forall k \in \NN^* \cup \{ \infty \}, \text{ } K(\infty, \frac{1}{k\Phi}) = \frac{1}{\Phi l_m^2}.$$
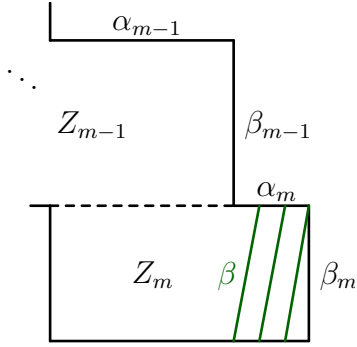
\begin{figure}
\center
\definecolor{qqwuqq}{rgb}{0,0.39215686274509803,0}
\begin{tikzpicture}[line cap=round,line join=round,>=triangle 45,x=1cm,y=1cm]
\clip(-3,-0.1) rectangle (1.6,4.5);
\draw [line width=1pt] (0,0)-- (1,0);
\draw [line width=1pt] (1,0)-- (1,1.8);
\draw [line width=1pt] (1,1.8)-- (0,1.8);
\draw [line width=1pt] (0,1.8)-- (0,4);
\draw [line width=1pt] (0,4)-- (-2.44,4);
\draw [line width=1pt] (-2.44,4)--(-2.44,4.5);
\draw [line width=1pt] (0,0)-- (-2.44,0);
\draw [line width=1pt,color=qqwuqq] (0,0)-- (0.34214876033057856,1.8);
\draw [line width=1pt,color=qqwuqq] (0.34214876033057856,0)-- (0.6842975206611571,1.8);
\draw [line width=1pt,color=qqwuqq] (0.6842975206611571,0)-- (1,1.8);
\draw [line width=1pt] (-2.44,1.8)-- (-2.44,0);
\draw [line width=1pt, dash pattern=on 3pt off 3pt] (-2.44,1.8)--(0,1.8);
\draw [line width=1pt] (-2.7,1.8)--(-2.44,1.8);
\draw (-1.15,3.9) node[above] {$\alpha_{m-1}$};
\draw (-1.5,1.2) node[anchor=north west] {$Z_m$};
\draw (-2.5,3.2) node[anchor=north west] {$Z_{m-1}$};
\draw (-3.2,4) node[anchor=north west] {$\ddots$};
\draw (0.16,2.25) node[anchor=north west] {$\alpha_m$};
\draw (1,1.2) node[anchor=north west] {$\beta_m$};
\draw (0,3.2) node[anchor=north west] {$\beta_{m-1}$};
\draw [color=qqwuqq](-0.36,1.2) node[anchor=north west] {$\beta$};
\end{tikzpicture}
\caption{Example of $\beta$ with co-slope $\frac{1}{3\Phi}$. Here, $\Int(\alpha_m,\beta) = 3$ with an intersection at the singularity.}
\label{fig:plane_template_direction_3Phi}
\end{figure}

\underline{\textbf{Case 1.2.ii.}} In the second case, the geodesic $\beta$ is not contained in the horizontal cylinder $Z_m$. As such, $\beta$ has a vertical length at least the length of $\beta_{m-1}$, that is $h_{m-1} = (\Phi^3 - 2\Phi) l_m$.
\begin{itemize}
\item[(1.2.ii.a)] \underline{If $\Int(\alpha_m, \beta) = 0$ or $1$}, we directly have 
$$
\frac{\Int(\alpha_m,\beta)}{\alpha_m \wedge \beta} \leq \frac{1}{(\Phi^3-2\Phi)l_m^2},
$$
with equality if and only if $\Int(\alpha_m,\beta) = 1$ and the vertical length of $\beta$ is exactly $h_{m-1}$. In particular, $\beta$ is contained in the horizontal cylinder $Z_{m-1}$ and up to an horizontal twist, its direction is $d' = 0$ or $d'=\pm \frac{\Phi^2-1}{\Phi^3-2\Phi}$ (see Figure \ref{lengths_staircase}). The case $d'=0$ is an instance of case \ding{171} and has already been dealt with. The case $d'=\pm \frac{\Phi^2-1}{\Phi^3-2\Phi}$ is exactly case \ding{169}. Further, taking the diagonal saddle connections $\gamma$ and $\gamma'$ of Figure \ref{lengths_staircase} of respective holonomy vectors $\begin{pmatrix} \Phi^2-1 \\ \Phi \end{pmatrix}l_m$ and $\begin{pmatrix} 1 \\ \Phi \end{pmatrix}l_m$ which intersect once at the singularity, we get:
\[
\frac{\Int(\gamma,\gamma')}{\gamma \wedge \gamma'} = \frac{1}{(\Phi^3-2\Phi)l_m^2} = K(\infty, \frac{\Phi^2-1}{\Phi^3-2\Phi}),
\]
This is due to the fact that the pair of directions $(\frac{1}{\Phi}, \Phi - \frac{1}{\Phi})$ is the image of the pair $(\infty, \frac{\Phi^2-1}{\Phi^3-2\Phi})$ by the diagonal action of the matrix $T_V R$ which belongs to the (unoriented) Veech group. In particular:
\[ K(\frac{1}{\Phi}, \Phi - \frac{1}{\Phi}) = K(\infty, \frac{\Phi^2-1}{\Phi^3-2\Phi}). \]
\item[(1.2.ii.b)] Otherwise, $\beta$ intersects $\alpha_m$ outside the singularity and is not contained in the horizontal cylinder $Z_m$. In this case, we show:
\end{itemize}

\begin{Lem}\label{lem:cas12iib}
Assume $\beta$ intersects $\alpha_m$ outside the singularity and is not contained in the horizontal cylinder $Z_m$. Then 
$$
\frac{\Int(\alpha_m,\beta)}{\alpha_m \wedge \beta} < \frac{1}{(\Phi^3-2\Phi)l_m^2}.
$$
\end{Lem}

\begin{proof}
Let $c$ denote the co-slope of $\beta$. The assumption ensures that $c \neq \frac{1}{k\Phi}$. Up to an horizontal twist, we can also assume that $c \in ]-\frac{1}{\Phi},\Phi - \frac{1}{\Phi}[  $. Finally, up to a symmetry with respect to the vertical axis, we can further assume $c >0$, giving:
\[ c \in ]0,\Phi - \frac{1}{\Phi}[ \backslash \{\frac{1}{k \Phi}, k \in \NN^* \cup \{\infty \} \} \]
\begin{enumerate}
\item Let us first study the case $\frac{1}{\Phi} < c < \Phi - \frac{1}{\Phi}$. Under this assumption, $\beta$ has to vertically go through the small horizontal cylinder $Z_m$ before any non singular intersection with $\alpha_m$ and vertically go through $Z_{m-1}$ after each non-singular intersection with $\alpha_m$ (see Figure \ref{fig:case_k=0}), so that if we denote by $p$ the number of non singular intersections between $\alpha_m$ and $\beta$ we get
\[ l(\beta) \sin \theta \geq p(h_m +(h_m + h_{m-1})) \]
where $\theta$ is the angle between the horizontal and the direction of $\beta$, so that $l(\beta) \sin \theta$ is the vertical length of $\beta$. \newline
Now, adding the possible singular intersection, we deduce that
\begin{align*}
\frac{\Int(\alpha_m, \beta)}{\alpha_m \wedge \beta} &\leq \frac{p+1}{p(h_m +(h_m + h_{m-1}))l_m} \\
& = \frac{p+1}{p\Phi^3} \cdot \frac{1}{l_m^2}.
\end{align*}
(Recall that $h_m = \Phi l_m$ and $h_{m-1} = (\Phi^3 - 2\Phi)l_m$.)\newline
The last quantity is maximal for $p=1$, so that
\[ \frac{\Int(\alpha_m, \beta)}{\alpha_m \wedge \beta} \leq \frac{2}{\Phi^3} \cdot \frac{1}{l_m^2} < \frac{1}{\Phi^3 - 2\Phi} \cdot \frac{1}{l_m^2}\]
where the last inequality comes from $\Phi < 2$. 

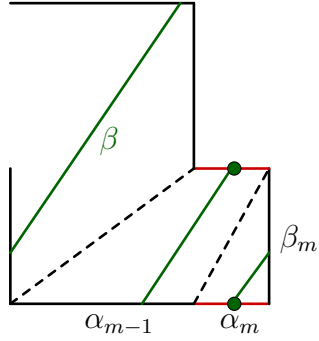
\begin{figure}
\center
\definecolor{qqwuqq}{rgb}{0,0.39215686274509803,0}
\definecolor{ccqqqq}{rgb}{0.8,0,0}
\begin{tikzpicture}[line cap=round,line join=round,>=triangle 45,x=1cm,y=1cm]
\clip(-3,-0.5) rectangle (1.7,4.1);
\draw [line width=1pt,color=ccqqqq] (0,0)-- (1,0);
\draw [line width=1pt] (1,0)-- (1,1.8);
\draw [line width=1pt,color=ccqqqq] (1,1.8)-- (0,1.8);
\draw [line width=1pt] (0,1.8)-- (0,4);
\draw [line width=1pt] (0,4)-- (-2.44,4);
\draw [line width=1pt] (0,0)-- (-2.44,0);
\draw [line width=1pt] (-2.44,1.8)-- (-2.44,0);
\draw (-1.6,0) node[anchor=north west] {$\alpha_{m-1}$};
\draw (0.2,0) node[anchor=north west] {$\alpha_m$};
\draw (1,1.2) node[anchor=north west] {$\beta_m$};
\draw (-1.4,2.5) node[anchor=north west,color=qqwuqq] {$\beta$};
\draw [line width=1pt,dash pattern=on 3pt off 3pt] (-2.44,0)-- (0,1.8);
\draw [line width=1pt,dash pattern=on 3pt off 3pt] (0,0)-- (1,1.8);
\draw [line width=1pt,color=qqwuqq] (-0.69,0)-- (0.5,1.8);
\draw [line width=1pt,color=qqwuqq] (-2.44,0.68)-- (-0.18,4);
\draw [line width=1pt,color=qqwuqq] (0.5,0)-- (1,0.68);
\begin{scriptsize}
\draw [fill=qqwuqq] (0.5368166581254012,1.8) circle (2.5pt);
\draw [fill=qqwuqq] (0.5368166581254012,0) circle (2.5pt);
\end{scriptsize}
\end{tikzpicture}
\caption{If the co-slope of $\beta$ lies between $\frac{1}{\Phi}$ and $\Phi- \frac{1}{\Phi}$, then a non-singular intersection with $\alpha_m$ requires a vertical length at least $h_m+(h_m+h_{m-1})$.}
\label{fig:case_k=0}
\end{figure}

\item Now, assume $0 < c < \frac{1}{\Phi}$ and let $k \geq 0$ be such that $\frac{1}{(k+1) \Phi} < c < \frac{1}{k \Phi}$. Let $p \geq 1 $ denote the number of times the curve $\beta$ crosses the small vertical cylinder $C_m$, that is the number of connected components of $\beta \cap C_m$. Then $\beta$ crosses $C_{m-1}$ at most once fewer, say $p' \geq 1$ times with $p' \geq p-1$. Moreover, the assumption on the co-slope gives that for each crossing of $C_m$ the curve $\beta$ intersects at most $k+1$ times $\alpha_m$ and for each crossing of $C_{m-1}$ it intersects at least $k$ times $\alpha_{m-1}$ (see Figure \ref{fig:case_k=2}). This gives that :
$$\Int(\alpha_m, \beta) \leq s_{max}^\circ := p(k+1) +1$$
(where the added intersection stands for the possible singular intersection). In fact, this estimate can be improved in the cases $p'=p$ and $p'=p-1$ as follows:
\begin{itemize}
\item[$\bigstar$] \underline{If $p' = p$}, the curve $\beta$ has to either start or end at a vertex of $C_m$, so that either the first or the last crossing of $C_m$ has only $k$ intersections with $\alpha_m$ instead of $k+1$. Hence, 
$$ \Int(\alpha_m, \beta) \leq s_{max}^\star := k(p+1) = (k+1)p-1 +1.$$
\item[\ding{115}] \underline{If $p' = p-1$} (so that $p \geq 2$), the curve $\beta$ has to both start and end at a vertex of $C_m$. The same argument shows:
$$ \Int(\alpha_m, \beta) \leq s_{max}^\Delta := (k+1)p-1 = (k+1)p-2+1.$$
\end{itemize}

\begin{figure}
\center
\definecolor{wwwwww}{rgb}{0.4,0.4,0.4}
\definecolor{qqwuqq}{rgb}{0,0.39215686274509803,0}
\begin{tikzpicture}[line cap=round,line join=round,>=triangle 45,x=1cm,y=1cm]
\clip(-2.5,-0.5) rectangle (1.7,4.2);
\fill[line width=1pt,color=wwwwww,fill=wwwwww,fill opacity=0.1] (-2.088447652072008,4) -- (0,4) -- (0.002478854036820562,0) -- (-2.1018276355316923,0) -- cycle;
\fill[line width=1pt,fill=black,fill opacity=0.15] (0,1.8) -- (1,1.8) -- (1,0) -- (0.002478854036820562,0) -- cycle;
\draw [line width=1pt] (0,0)-- (1,0);
\draw [line width=1pt] (1,0)-- (1,1.8);
\draw [line width=1pt] (1,1.8)-- (0,1.8);
\draw [line width=1pt] (0,1.8)-- (0,4);
\draw [line width=1pt] (0,4)-- (-2.1,4);
\draw [line width=1pt] (0,0)-- (-2.1018276355316923,0);
\draw [line width=1pt] (-2.1,1.8)-- (-2.1018276355316923,0);
\draw (-1.5848737291348032,0.0400925798775776) node[anchor=north west] {$\alpha_{m-1}$};
\draw (0.24331855631113586,0.0400925798775776) node[anchor=north west] {$\alpha_m$};
\draw (1.0972047734655266,1.2990273872205922) node[anchor=north west] {$\beta_m$};
\draw [line width=1pt,color=qqwuqq] (0.16547137981842694,0)-- (0.5644381593435552,1.8);
\draw [line width=1pt,color=qqwuqq] (0.5644381593435552,0)-- (0.9634049388686834,1.8);
\draw [line width=1pt,color=qqwuqq] (0.9634049388686834,0)-- (1,0.16510424781424984);
\draw [line width=1pt,color=qqwuqq] (-2.1018276355316923,0.16510424781424984)-- (-1.2518298532738348,4);
\draw [line width=1pt,color=qqwuqq] (-1.253657488805527,0)-- (-0.36706464541635303,4);
\draw [line width=1pt,color=qqwuqq] (-0.36889228094804527,0)-- (-0.10377430444030677,1.1961205348523836);
\draw [line width=1pt,color=qqwuqq] (-0.03563780616452043,0.892664970250974)-- (0.16547137981842694,1.8);
\draw [line width=1pt,dash pattern=on 3pt off 3pt,color=qqwuqq] (-0.2334953997067013,0)-- (-0.03563780616452043,0.892664970250974);
\draw [line width=1pt,dash pattern=on 3pt off 3pt,color=qqwuqq] (-0.10377430444030677,1.1961205348523836)-- (0.03007449857708295,1.8);
\draw (-1.7,2.3) node[anchor=north west] {$C_{m-1}$};
\draw (0.25,1.3) node[anchor=north west] {$C_m$};
\draw [color=qqwuqq](-0.9389854366718665,1.4741835343291856) node[anchor=north west] {$\beta$};
\end{tikzpicture}
\caption{A portion of the curve $\beta$ with co-slope $\frac{1}{3\Phi} < c < \frac{1}{2\Phi}$. Each time $\beta$ crosses the cylinder $C_m$, it gives at most $k+1 = 3$ intersections with $\alpha_m$ and it is followed by a crossing of $C_{m-1}$, giving at least $k=2$ intersections with $\alpha_{m-1}$.}
\label{fig:case_k=2}
\end{figure}
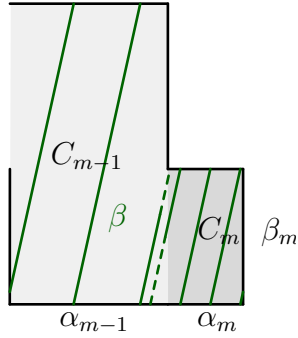

Concerning the lengths, cutting $\beta$ at each intersection with $\alpha_m$ and $\alpha_{m-1}$ helps us notice that:
\begin{enumerate}
\item[$(i)$] Before any non-singular intersection with $\alpha_m$, there is a vertical length at least $h_m$.
\item[$(ii)$] Before any non-singular intersection with $\alpha_{m-1}$, there is a vertical length at least $h_m+h_{m-1}$, except maybe the first intersection with $\alpha_{m-1}$ for which there is a vertical length a least $h_{m-1}$ (but note that this case arises only if $\alpha$ starts at a bottom vertex of $Z_{m-1}$, in particular this cannot happen for $p' = p-1$).
\item[$(iii)$] We further have to count a vertical length $h_m$ for $\beta$ to reach the singularity after the last intersection with either $\alpha_m$ or $\alpha_{m-1}$.
\end{enumerate}

In particular, letting $s := Int(\alpha_m, \beta) \geq 2$ (so that there are at least $s-1$ non-singular intersections), we have :
\begin{align*}
l(\beta) \sin \theta & \geq (s-1)h_m + (kp'-1) (h_m+h_{m-1}) + h_{m-1} + h_m \\
& = kp'(h_m+h_{m-1}) + (s-1) h_m,
\end{align*}
and for $p' = p-1$, we can replace $kp'-1$ by $kp'$ as remarked in $(ii)$.
\begin{align*}
l(\beta) \sin \theta & \geq (s-1)h_m + kp' (h_m+h_{m-1}) + h_m \\
&= k(p-1)(h_m+h_{m-1}) + s h_m.
\end{align*}

In particular, the following inequalities holds:
\begin{align*}
& \text{ If } p' \geq p \text{, then }
 \frac{\Int(\alpha_m, \beta)}{\alpha_m \wedge \beta} \leq \frac{s}{p'k(h_m+h_{m-1}) + (s-1) h_m}\cdot \frac{1}{l_m} & & \\
& (\text{\ding{115}}) \text{ If } p' =p-1 \text{, then }
\frac{\Int(\alpha_m, \beta)}{\alpha_m \wedge \beta} \leq \frac{s}{(p-1)k(h_m+h_{m-1}) + s h_m}\cdot \frac{1}{l_m}. & &
\end{align*}

Now, since $k(h_m+h_{m-1}) > h_m$, for fixed $k$ and $p$, the right part is maximal when $s$ is maximal. Distinguishing cases \ding{108} ($p' \geq p+1$), $\bigstar$ ($p'=p$) and \ding{115} ($p'=p-1$), we have:
\begin{itemize}
\item[\ding{108}] \underline{If $p' \geq p+1$}, then $s_{max}=s_{max}^\circ = p(k+1)+1$ and 
\begin{align*}
\frac{\Int(\alpha_m, \beta)}{\alpha_m \wedge \beta} & \leq \frac{p(k+1)+1}{p'k(\Phi^3-\Phi) + p(k+1) \Phi} \cdot \frac{1}{l_m^2}\\
& \leq \frac{p(k+1)+1}{(p+1)k(\Phi^3-\Phi) + p(k+1) \Phi} \cdot \frac{1}{l_m^2}\\
& = \frac{pk+p+1}{k((p+1)\Phi^3-\Phi) + p\Phi} \cdot \frac{1}{l_m^2}
\end{align*}
(given that $h_m = \Phi l_m$ and $h_{m-1} = (\Phi^3 - 2\Phi)l_m$.) \newline

To give an upper bound for this last quantity, it is convenient to use the following lemma which will be used several times along the proof:
\begin{Lem}\label{lemme_fractions}
Let $a,b,c,d \in \RR$. Assume the m\"obius transformation $f:x \mapsto \frac{ax+b}{cx+d}$ is defined for any $x \geq 1$. Then,
\begin{itemize}
\item if $ad-bc \geq 0$, then $\forall x \geq 1$, $f(x) \leq \lim_{x \to \infty}f(x) = \frac{a}{c}$.
\item Otherwise, $\forall x \geq 1$, $f(x) \leq f(1) = \frac{a+b}{c+d}$.
\end{itemize}
\end{Lem}
\begin{proof}
$f$ is differentiable of derivative map $f'(x) = \frac{ad-bc}{(cx+d)^2}$.
\end{proof}
Since for any fixed $p \geq 1$ we have $p^2\Phi - (p+1)((p+1) \Phi^3 -\Phi) <(p+1)^2(\Phi - \Phi^3) <0$, we can use Lemma \ref{lemme_fractions} to conclude that the last quantity is maximal for $k=1$, and it gives
\begin{align*}
\frac{\Int(\alpha_m, \beta)}{\alpha_m \wedge \beta}& \leq  \frac{2p+1}{(p+1)(\Phi^3-\Phi)+2p\Phi}\cdot \frac{1}{l_m^2} \\
& = \frac{2p+1}{p(\Phi^3+\Phi)+(\Phi^3-\Phi)}\cdot \frac{1}{l_m^2}
\end{align*}
Similarly, given the inequality $2(\Phi^3-\Phi)-(\Phi^3+\Phi) < 0$, coming from $\Phi > 2\cos(\frac{\pi}{6}) = \sqrt{3}$, Lemma \ref{lemme_fractions} shows that the last quantity is maximal for $p \to \infty$, and we have
\[
\frac{\Int(\alpha_m, \beta)}{\alpha_m \wedge \beta} = \frac{2}{\Phi^3+\Phi} \cdot \frac{1}{l_m^2} < \frac{1}{\Phi^3- 2\Phi} \cdot \frac{1}{l_m^2},
\]
where the last inequality comes from $\Phi < 2$.

\item[$\bigstar$] \underline{If $p' = p$}, then $s_{max} = s_{max}^\star=p(k+1)$ and:
\begin{align*}
\frac{\Int(\alpha_m, \beta)}{\alpha_m \wedge \beta} & \leq \frac{p(k+1)}{pk(\Phi^3-\Phi) + (p(k+1)-1) \Phi} \cdot \frac{1}{l_m^2}\\
& = \frac{p(k+1)}{p(k\Phi^3+\Phi) -\Phi} \cdot \frac{1}{l_m^2}
\end{align*}
Similarly, for fixed $k$ the last quantity is maximal for $p=1$, giving
\begin{align*}
\frac{\Int(\alpha_m, \beta)}{\alpha_m \wedge \beta}& \leq  \frac{k+1}{k(\Phi^3-\Phi)+k\Phi}\cdot \frac{1}{l_m^2} \\
& = \frac{k+1}{k \Phi^3} \cdot \frac{1}{l_m^2} < \frac{2}{\Phi^3} \cdot \frac{1}{l_m^2} < \frac{1}{\Phi^3- 2\Phi} \cdot \frac{1}{l_m^2}.
\end{align*}

\item[\ding{115}] Finally, \underline{if $p' = p-1$}, we have $s_{max} = s_{max}^\Delta=p(k+1)-1$ and:
\begin{align*}
\frac{\Int(\alpha_m, \beta)}{\alpha_m \wedge \beta} &\leq \frac{p(k+1)-1}{(p-1)k(\Phi^3-\Phi) + (p(k+1)-1) \Phi} \cdot \frac{1}{l_m^2}\\
& = \frac{p(k+1)-1}{p(k\Phi^3+\Phi) - k\Phi^3 + (k-1)\Phi} \cdot \frac{1}{l_m^2}
\end{align*}
Using the inequality 
$
(k+1)((k-1)\Phi -k\Phi^3) + k\Phi^3 + \Phi = k^2(\Phi-\Phi^3) < 0
$
and Lemma \ref{lemme_fractions}, we see that for fixed $k \geq 1$, the last quantity is maximal for $p$ minimal, that is $p=2$ since $p' =p-1$ should be at least one. In particular
\begin{align*}
\frac{\Int(\alpha_m, \beta)}{\alpha_m \wedge \beta} & \leq \frac{2k+1}{k(\Phi^3-\Phi) + (2k+1) \Phi}\cdot \frac{1}{l_m^2}\\
& = \frac{2k+1}{k(\Phi^3+\Phi) + \Phi}\cdot \frac{1}{l_m^2}
\end{align*}
which, using again Lemma \ref{lemme_fractions}, is shown to be minimal for $k=1$, so that
\[
\frac{\Int(\alpha_m, \beta)}{\alpha_m \wedge \beta} < \frac{3}{\Phi^3+2\Phi} \cdot \frac{1}{l_m^2} < \frac{1}{\Phi^3- 2\Phi} \cdot \frac{1}{l_m^2},
\]
where the last inequality comes from $\Phi < 2$.
\end{itemize}
\end{enumerate}
%Hence, in all cases we have $K(d,d') \leq \frac{1}{\Phi^3-2\Phi} \cdot \frac{1}{l_m^2}$ for $(d,d') \notin \Gamma_n \cdot \mathcal{G}_{max}$.\newline
This concludes the proof of Lemma \ref{lem:cas12iib}.
\end{proof}

\paragraph{\underline{Case 2: $d$ represents the vertical cusp.}} Up to the action by an element of the Veech group, we can assume that $\alpha$ is one of the $\beta_j$, $j \in \llbracket 1, m \rrbracket$. We first study the case $\alpha = \beta_j$ for $j<m$.
\paragraph{\underline{Case 2.1: $\alpha = \beta_j$ for $j<m$.}}
\begin{Lem}
Assume $d' \notin \{\frac{1}{k\Phi}, k \in \NN \cup \{\infty \} \}$. For any $j < m$, 
$$
\frac{\mathrm{Int}(\beta_j,\beta)}{\beta_j \wedge \beta} < \frac{1}{(\Phi^3-2\Phi)}\cdot \frac{1}{l_m^2}
$$
\end{Lem}
\begin{proof}
First notice that if $\beta$ has horizontal length $l_m$, then its direction is $\frac{1}{k\Phi}$ for a given $k$. In particular the assumption on $d'$ ensures that $\beta$ has horizontal length at least than $l_{m-1}$.
Further, any non-singular intersection with $\beta_j$ for $j<m$ requires an horizontal length $l_j+l_{j-1}$ (see Figure \ref{fig:lem410}). This gives, for $j <m$,
$$
\beta_j \wedge \beta \geq h_j \mathrm{max}((l_j+l_{j-1})(\mathrm{Int}(\beta_j, \beta)-1), l_{m-1}).
$$

In particular,
\begin{enumerate}[label =(\roman*)]
\item If $\Int(\beta_j,\beta) \leq 1$, we use $\beta_j \wedge \beta \geq h_j l_{m-1}$ to get
\[  \frac{\mathrm{Int}(\beta_j,\beta)}{\beta_j \wedge \beta} \leq \frac{1}{l_{m-1} h_j} \leq \frac{1}{l_{m-1}h_{m-1}}.\]
Then, given that $l_{m-1} = (\Phi^2-1) l_m$ and $h_{m-1} = (\Phi^3-2\Phi) l_m$, we draw:
$$
\frac{\mathrm{Int}(\beta_j,\beta)}{\beta_j \wedge \beta} \leq \frac{1}{(\Phi^3-2\Phi)(\Phi^2-1)l_m^2} < \frac{1}{(\Phi^3-2\Phi)}\cdot \frac{1}{l_m^2}
$$

\item Otherwise, we use $\beta_j \wedge \beta \geq  h_j (l_j+l_{j-1})(\mathrm{Int}(\beta_j, \beta)-1)$ to obtain:
\[ \frac{\mathrm{Int}(\beta_j,\beta)}{\beta_j \wedge \beta} \leq \frac{2}{h_j(l_j+l_{j-1})}. \]
Since $2 l_{m-1} \leq 2l_j < (\Phi^2-1)l_j = l_{j-1}$, we deduce that
\[ \frac{\mathrm{Int}(\beta_j,\beta)}{\beta_j \wedge \beta} \leq \frac{1}{h_j l_{m-1}} \leq \frac{1}{l_{m-1}h_{m-1}} < \frac{1}{(\Phi^3-2\Phi)}\cdot \frac{1}{l_m^2}. \]
\end{enumerate}

\begin{figure}
\center
\definecolor{qqwuqq}{rgb}{0,0.39215686274509803,0}
\definecolor{ccqqqq}{rgb}{0.8,0,0}
\begin{tikzpicture}[line cap=round,line join=round,>=triangle 45,x=1cm,y=1cm, scale = 1.2]
\clip(-3.5,-0.6) rectangle (2.3,2.6);
\draw [line width=1pt] (-3,0)-- (0,0);
\draw [line width=1pt,dash pattern=on 3pt off 3pt] (0,0)-- (0,-1);
\draw [line width=1pt] (2,2)-- (0,2);
\draw [line width=1pt,color=ccqqqq] (-3,2)-- (-3,0);
\draw [line width=1pt,color=ccqqqq] (2,2)-- (2,0);
\draw [line width=1pt,dash pattern=on 3pt off 3pt] (2,0)-- (2,-1);
\draw [line width=1pt,dash pattern=on 3pt off 3pt] (0,2)-- (0,3);
\draw [line width=1pt,color=qqwuqq] (-3,1)-- (-1.4199582045705343,1.3251239848287564);
\draw [line width=1pt,dash pattern=on 3pt off 3pt,color=qqwuqq] (-1.4199582045705343,1.3251239848287564)-- (-0.3192737955430047,1.5516109689940372);
\draw [line width=1pt,color=qqwuqq] (2,1)-- (0.9604206507739941,0.7860865569861867);
\draw [line width=1pt,dash pattern=on 3pt off 3pt,color=qqwuqq] (0.9604206507739941,0.7860865569861867)-- (0.2666666666666651,0.6433333333333321);
\draw [color=ccqqqq](-3.5,1) node[anchor=north west] {$\beta_j$};
\draw [color=qqwuqq](1.1,0.9) node[anchor=north west] {$\beta$};
\draw [-to,line width=0.5pt] (-1.4,-0.3) -- (-0.10667,-0.3);
\draw [-to,line width=0.5pt] (-1.78667,-0.3) -- (-2.93333,-0.3);
\draw [-to,line width=0.5pt] (0.61333,-0.3) -- (0.10667,-0.3);
\draw [-to,line width=0.5pt] (1.28,-0.3) -- (1.89333,-0.3);
\draw (-1.8,-0.05) node[anchor=north west] {$l_j$};
\draw (0.6,-0.05) node[anchor=north west] {$l_{j+1}$};
\begin{scriptsize}
\draw [fill=qqwuqq] (2,1) circle (2.5pt);
\draw [fill=qqwuqq] (-3,1) circle (2.5pt);
\end{scriptsize}
\end{tikzpicture}
\caption{Any non-singular intersection with $\beta_j$ requires a horizontal length at least $l_j + l_{j+1}$.}
\label{fig:lem410}
\end{figure}
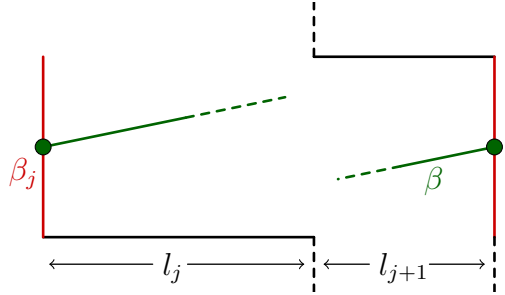
\end{proof}

\paragraph{\underline{Case 2.2: $\alpha = \beta_m$.}}
In this case $\alpha$ is homologous to a non singular geodesic so there is no non singular intersection. The first intersection with $\beta_m$ requires a horizontal length at least $l_m$ while all the following intersections require a length $l_{m-1}+l_m$. In particular, either $\beta$ stays in the vertical cylinder $C_m$ and the co-slope is $\frac{1}{k\Phi}$ for $k \in \NN \cup \{ \infty \}$ (case \ding{171}), or there are at least two intersections and
$$
\frac{\mathrm{Int}(\beta_m,\beta)}{\beta_m \wedge \beta} \leq \frac{2}{h_m (2 l_m+l_{m-1})} = \frac{2}{\Phi(\Phi^2+1)}\cdot \frac{1}{l_m^2} < \frac{1}{\Phi^3-2\Phi}\cdot \frac{1}{l_m^2}.
$$
where the last inequality comes from $\Phi < 2$. This concludes the proof of Proposition \ref{prop:etude_K}. \newline
\end{proof}

\section{Proof of Theorem \ref{theo:main}}\label{sec:extension_teichmuller}
In this section, we finally prove our main result. More precisely, we show

\begin{Theo}\label{theo:main_reformule}
For any $X \in \TT$, we have
\begin{equation}\label{Formule_2}
\KVol(X) = \mathrm{Vol}(X) K(\infty,\frac{1}{\Phi}) \cdot \frac{1}{\cosh(\mathrm{dist}_{\HH}(X, \Gamma_n \cdot \mathcal{G}_{max}))}
\end{equation}
\end{Theo}

Recall that $\mathcal{G}_{max} = \bigcup_{k \in \NN^* \cup \{ \infty \}} \gamma_{\infty, \pm \frac{1}{k\Phi}}$, and notice that the constant $K_0$ given in the statement of Theorem \ref{theo:main} is now explicitely determined by $K_0 = \Vol(X) K(\infty, \frac{1}{\Phi})$.

We begin this section by giving an overview of the geometric ideas behind the proof.

\subsection{Geometric interpretation of Proposition \ref{prop:sup_directions}}
Following \S 7.2 of \cite{BLM22}, we first give a geometric interpretation of Equation~\eqref{eq:reformulation} (see Proposition \ref{prop:sup_directions}), which we restate here for the reader's convenience:

\[
\KVol(X) = \mathrm{Vol}(X) \cdot \sup_{
 \begin{scriptsize}
 \begin{array}{c}
d, d' \in \mathcal{P} \\
d \neq d'
\end{array}
\end{scriptsize}} K(d,d')\cdot \sin \theta(X,d,d').
\]
Recall that $\mathcal{P}$ is the set of periodic directions on the staircase model $\SS$.\newline

Given a pair of distinct periodic directions $(d,d')$, its associated geodesic $\gamma_{d,d'}$ on $\HH$ gives a projected geodesic on the hyperbolic surface $\HH / \Gamma_n$. In the fundamental domain $\TT$, it gives a set of geodesics formed by all the images of $\gamma_{d,d'}$ by an element of $\Gamma_n$. By Remark \ref{rk:invariance_K}, all pairs of directions in $\Gamma_n^{\pm} \cdot (d,d')$ give the same value for $K(\cdot,\cdot)$. This amounts to saying that the corresponding geodesic trajectory on the surface $\HH / \Gamma_n$ has a well defined associated constant, $K(d,d')$. Now, given a point $X \in \HH / \Gamma_n$, we can look at the minimal distance $r$ from $X$ to the geodesic trajectory associated with $(d,d')$. Proposition \ref{prop:directions} gives that for a pair of saddle connections $\alpha$ and $\beta$ on $X$ of respective directions $d$ and $d'$ (in the sense of Definition~\ref{def:directions}), we have the sharp inequality:
\[
\frac{\mathrm{Int}(\alpha, \beta)}{l(\alpha) l(\beta)} \leq K(d,d') \times \frac{1}{\cosh r}.
\]
Examples of geodesic trajectories for $d=\infty$ and $d'= \frac{1}{\Phi}, \frac{1}{2\Phi}$ and $\frac{1}{3\Phi}$ are depicted in Figure \ref{geodesics_kPhi} of the introduction, as well as in Figure \ref{geodesics_autre} for $d'= \frac{\Phi}{2\Phi^2-1}$. \newline 

Using this interpretation, we can prove Theorem \ref{theo:main_reformule} by showing that for every pair of distinct periodic directions $(d,d')$ and any $X \in \TT$ (or, by symmetry, for any $X \in \TT_{+} = \{ X = x+iy \in \TT, x \geq 0 \}$) we have: 

\begin{equation}\label{eq:trefle}
\tag{\ding{168}}
 K(d,d') \sin \theta(X, d,d') \leq \max_{k \in \NN^{*}} \{ K(\infty, \frac{1}{k \Phi}) \sin \theta (X,\infty, \frac{1}{k \Phi})\}
\end{equation}

The constant $K(\infty, \frac{1}{k\Phi})$ being maximal among all possible constants $K(d,d')$, Equation \eqref{eq:trefle} holds for surfaces $X$ close to $\Gamma_n \cdot \mathcal{G}_{max}$. More precisely, in \S 5.2, we use Proposition \ref{prop:etude_K} to show that \eqref{eq:trefle} holds in the domain $\mathcal{R}_+ = \{X = x+iy \in \TT_+,y \geq x-\frac{1}{\Phi} \}$ (Lemma \ref{lem:R_1_case} and Figure \ref{Bonne_geodesique}).

Then, it remains to deal with surfaces outside $\mathcal{R}_+$ (which are, in some sense, surfaces close to the regular $n$-gon $\XX$). This latter case is explained in \S5.3 and requires use of the machinery of \cite[\S 7]{BLM22}, along with Corollary \ref{cor:K_regular} and Proposition \ref{prop:etude_K}.

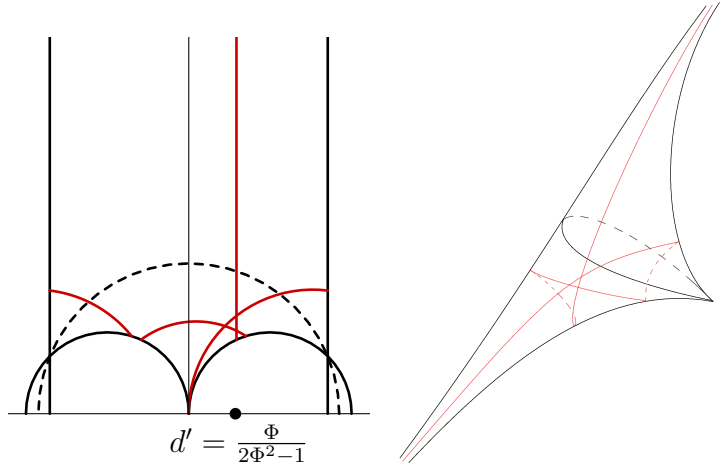
\begin{figure}
\center
\definecolor{ccqqqq}{rgb}{0.8,0,0}
\begin{tikzpicture}[line cap=round,line join=round,>=triangle 45,x=2cm,y=2cm]
\clip(-1.2,-0.5) rectangle (1.2,2.5);
\draw [shift={(0.541196100146197,0)},line width=1pt]  plot[domain=0:3.141592653589793,variable=\t]({1*0.541196100146197*cos(\t r)+0*0.541196100146197*sin(\t r)},{0*0.541196100146197*cos(\t r)+1*0.541196100146197*sin(\t r)});
\draw [shift={(0,0)},line width=1pt,dash pattern=on 3pt off 3pt]  plot[domain=0:3.141592653589793,variable=\t]({1*1*cos(\t r)+0*1*sin(\t r)},{0*1*cos(\t r)+1*1*sin(\t r)});
\draw [line width=1pt] (0.9238795325112867,0) -- (0.9238795325112867,3.0472427983539028);
\draw [line width=1pt] (-0.9238795325112867,0) -- (-0.9238795325112867,3.0472427983539028);
\draw [shift={(-0.541196100146197,0)},line width=1pt]  plot[domain=0:3.141592653589793,variable=\t]({1*0.541196100146197*cos(\t r)+0*0.541196100146197*sin(\t r)},{0*0.541196100146197*cos(\t r)+1*0.541196100146197*sin(\t r)});
\draw [shift={(-1.0238408779149524,0)},line width=1pt,color=ccqqqq]  plot[domain=0.6763286696673477:1.4493964179776777,variable=\t]({1*0.825431473300645*cos(\t r)+0*0.825431473300645*sin(\t r)},{0*0.825431473300645*cos(\t r)+1*0.825431473300645*sin(\t r)});
\draw [shift={(0.04897119341563786,0)},line width=1pt,color=ccqqqq]  plot[domain=1.0008387825741267:2.2098093501643175,variable=\t]({1*0.6136726703570244*cos(\t r)+0*0.6136726703570244*sin(\t r)},{0*0.6136726703570244*cos(\t r)+1*0.6136726703570244*sin(\t r)});
\draw [shift={(0.8239181871076211,0)},line width=1pt,color=ccqqqq]  plot[domain=1.4493964179776777:3.141592653589793,variable=\t]({1*0.8254314733006451*cos(\t r)+0*0.8254314733006451*sin(\t r)},{0*0.8254314733006451*cos(\t r)+1*0.8254314733006451*sin(\t r)});
\draw [line width=1pt,color=ccqqqq] (0.3170253355622144,0.49258571550470803) -- (0.3170253355622144,3.0472427983539028);
\draw (-0.2,0) node[anchor=north west] {$d' = \frac{\Phi}{2\Phi^2-1}$};
\draw [line width=0.4pt,domain=-2.2700411522633757:2.93292181069959] plot(\x,{(-0-0*\x)/1});
\draw [line width=0.4pt] (0,0) -- (0,3.0472427983539028);
\begin{scriptsize}
\draw [fill=black] (0.31,0) circle (2pt);
\end{scriptsize}
\end{tikzpicture}
\includegraphics[width=5cm]{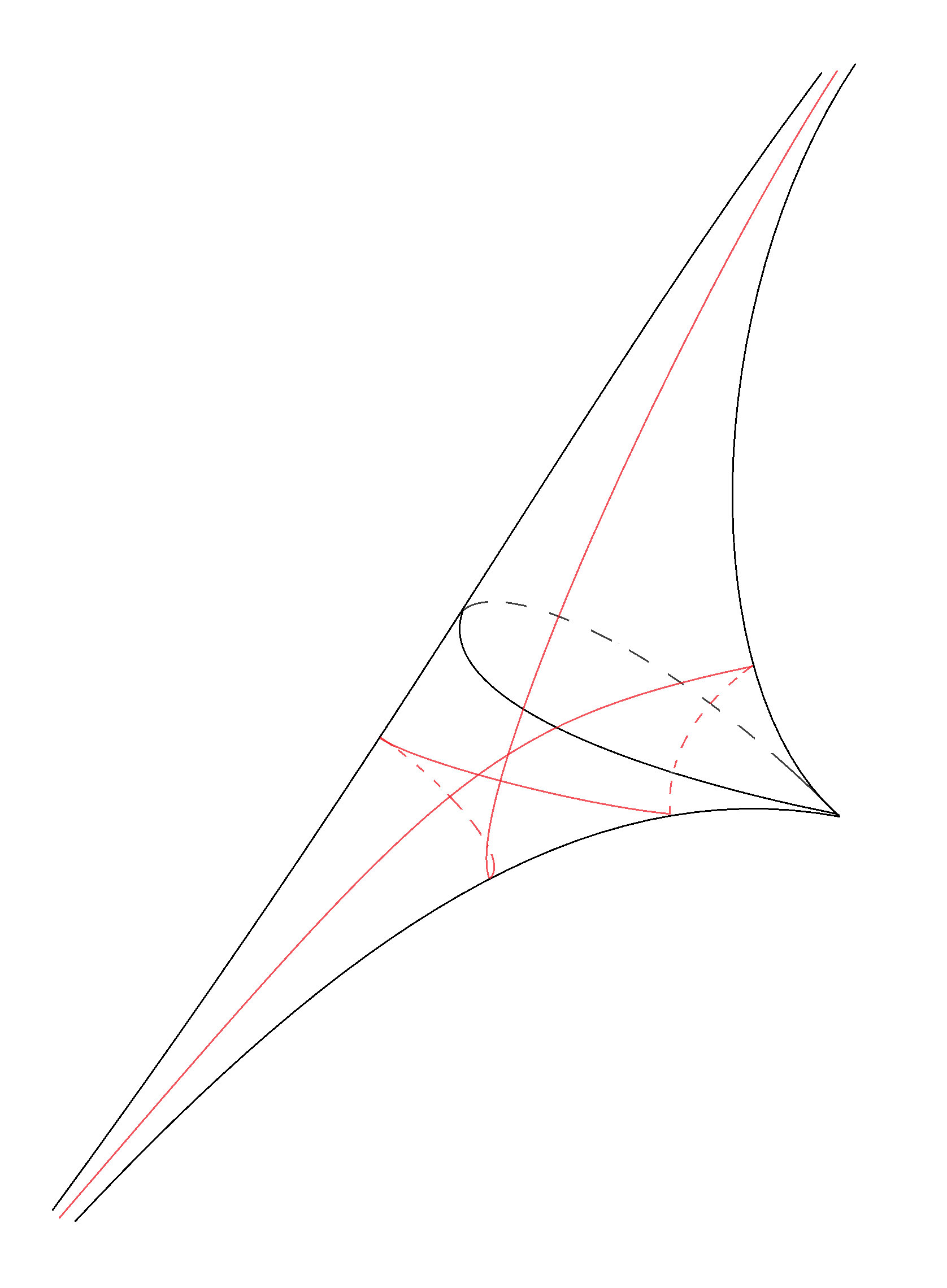}
\caption{The geodesic $\gamma_{\infty, \frac{\Phi}{2\Phi^2-1}}$ its images by the Veech group intersecting the fundamental domain $\TT$. On the right, the same geodesics on the surface $\HH / \Gamma_n$.}
\label{geodesics_autre}
\end{figure}

\subsection{The case of surfaces close to $\Gamma_n \cdot \mathcal{G}_{max}$}

With the above geometric interpretation, we deduce that for surfaces close to the geodesic trajectories associated to the directions $(\infty, \frac{1}{k\Phi})$, $k \in \NN^* \cup \{\infty\}$ (having maximal $K(\cdot,\cdot)$), KVol is achieved by pairs of saddle connections in directions $d = \infty$ and $d' = \frac{1}{k \Phi}$, and Equation \eqref{eq:trefle} holds. More precisely, we have:

\begin{Lem}\label{lem:R_1_case}
For any $X \in \mathcal{R}_+ = \{X = x+iy \in \TT_+$ with $y \geq x-\frac{1}{\Phi} \}$, Equation \eqref{eq:trefle} holds.
\end{Lem}
Recall that because of the symmetry of KVol with respect to the reflection on the vertical axis, we can assume $X \in \TT_{+} := \{ X = x+iy \in \TT, x \geq 0\}$.

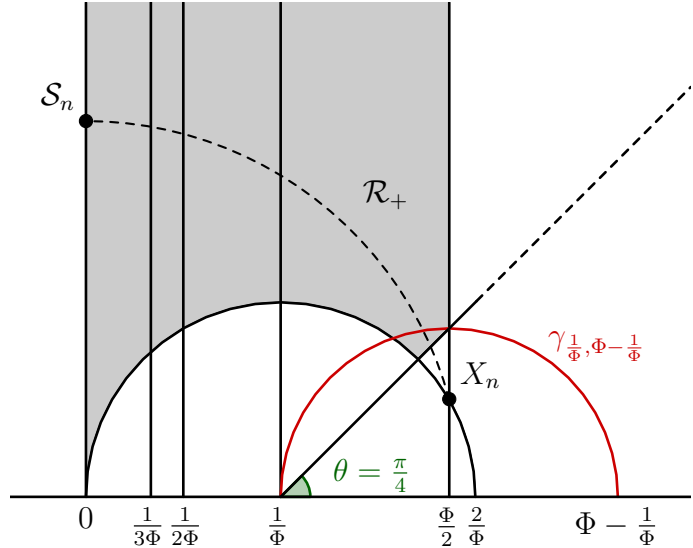
\begin{figure}[h]
\center
\definecolor{qqwuqq}{rgb}{0,0.39215686274509803,0}
\definecolor{ccqqqq}{rgb}{0.8,0,0}
\begin{tikzpicture}[line cap=round,line join=round,>=triangle 45,x=1cm,y=1cm,scale=5]
\clip(-0.2,-0.25) rectangle (1.6082304526748996,1.3155555555555518);
\def\h{1.5}
\coordinate (b) at (0,0);
\coordinate (bdr) at (0.96592582628,0.451);
\coordinate (mil) at (0.88,0.365);
%\fill (dpmbis) circle[radius=2pt];
%\fill[color=gray!50] (a) to (b) to (2,3) to cycle;
\filldraw [draw=black, fill=gray!40](0,\h) to (b) to (mil) to (bdr) -- (0.965925826289,\h);
\filldraw[fill=white] (1.03527618041,0) -- (1.03527618041,0) arc (0:180:0.5176380902050415) (0,0) -- cycle;
\draw [shift={(0.5176380902050415,0)},line width=1pt,color=qqwuqq,fill=qqwuqq,fill opacity=0.3] (0,0) -- (0:0.07901234567901236) arc (0:45:0.07901234567901236) -- cycle;
\draw [line width=1pt] (0.9659258262890683,0) -- (0.9659258262890683,1.3155555555555518);
\draw [line width=1pt] (0.5176380902050415,0) -- (0.5176380902050415,1.3155555555555518);
\draw [line width=1pt] (0.2588190451025207,0) -- (0.2588190451025207,1.3155555555555518);
\draw [line width=1pt] (0.1725460300683471,0) -- (0.1725460300683471,1.32);
\draw [shift={(0.5176380902050415,0)},line width=1pt]  plot[domain=0:3.141592653589793,variable=\t]({1*0.5176380902050415*cos(\t r)+0*0.5176380902050415*sin(\t r)},{0*0.5176380902050415*cos(\t r)+1*0.5176380902050415*sin(\t r)});
\draw [shift={(0.9659258262890683,0)},line width=1pt,color=ccqqqq]  plot[domain=0:3.141592653589793,variable=\t]({1*0.44828773608402683*cos(\t r)+0*0.44828773608402683*sin(\t r)},{0*0.44828773608402683*cos(\t r)+1*0.44828773608402683*sin(\t r)});
\draw [shift={(0,0)},line width=0.8pt,dash pattern=on 3pt off 3pt]  plot[domain=0.2617993877991492:1.5707963267948966,variable=\t]({1*1*cos(\t r)+0*1*sin(\t r)},{0*1*cos(\t r)+1*1*sin(\t r)});
\draw (0.45,0) node[anchor=north west] {$\frac{1}{\Phi}$};
\draw (0.19,0) node[anchor=north west] {$\frac{1}{2\Phi}$};
\draw (0.09,0) node[anchor=north west] {$\frac{1}{3\Phi}$};
\draw (0.9,0) node[anchor=north west] {$\frac{\Phi}{2}$};
\draw (0.98,0) node[anchor=north west] {$\frac{2}{\Phi}$};
\draw (1.27,0.0) node[anchor=north west] {$\Phi - \frac{1}{\Phi}$};
\draw [color = ccqqqq] (1.2, 0.4) node[right] {$ \gamma_{\frac{1}{\Phi}, \Phi - \frac{1}{\Phi}}$};
\draw [line width=1pt] (0,0) -- (0,1.3155555555555518);
\draw [line width=1pt,domain=-0.7041975308641953:1.6082304526748996] plot(\x,{(-0-0*\x)/1});
\draw (0.7074897119341587,0.8704526748971166) node[anchor=north west] {$\mathcal{R}_+$};
\draw (-0.05,0) node[anchor=north west] {$0$};
\draw [color=qqwuqq](0.6284773662551463,0.13) node[anchor=north west] {$\theta = \frac{\pi}{4}$};
\draw [line width=1pt] (0.5176380902050415,0)-- (1.038572131522274,0.5209340413172325);
\draw [line width=1pt,dash pattern=on 3pt off 3pt,domain=1.038572131522274:1.6082304526748996] plot(\x,{(-0.19563737532687653--0.37794238683127546*\x)/0.37794238683127546});
\draw [fill=black] (0,1) circle (0.5pt);
\draw [fill=black] (0.9659,0.26) circle (0.5pt);
\draw (0,1) node[anchor=south east] {$\SS$};
\draw (0.9659,0.26) node[anchor=south west] {$\XX$};
\end{tikzpicture}
\caption{The region $\mathcal{R}_+$ in the right part of the fundamental domain $\TT^+$ and the geodesic $(\frac{1}{\Phi}, \Phi - \frac{1}{\Phi})$ for the case $n=12$.}
\label{Bonne_geodesique}
\end{figure}

\begin{proof}
 By Proposition \ref{prop:etude_K}, Equation \eqref{eq:trefle} holds for any $X \in \TT^+$ such that there is a $k \in \NN^* \cup \{ \infty \}$ with 
\begin{equation}\label{eq:star}
\sin \theta(X,\infty, \frac{1}{k\Phi}) \geq \frac{1}{\Phi^2-2}.
\end{equation}

Since $\cfrac{1}{\Phi^2-2} \leq \cfrac{1}{\Phi_8^2-2} = \cfrac{\sqrt{2}}{2} = \sin \cfrac{\pi}{4}$, we directly have that the above condition is verified for $X \in \{Y \in \TT^+, \exists k \in \NN^* \cup \{\infty \}, \theta(Y,\infty, \frac{1}{k\Phi}) \geq \frac{\pi}{4} \}$. One verifies that this set is exactly $\mathcal{R}_+$.
\end{proof}
\begin{Rema}
In fact, $\mathcal{R}_+ = \TT^+$ for $m=2$, which finishes the proof in the case of the octagon. However, this is no longer the case for $m \geq 3$.
\end{Rema}

\subsection{The case of surfaces close to $\XX$}
It remains to deal with surfaces away from geodesics associated with directions $\infty$ and $\frac{1}{k \Phi}$. In a way, these surfaces are close to the regular $n$-gon $\XX$. To this end we make the following definition (which can also be found in \cite[\S 7.2]{BLM22}):
\begin{Def}
Given a pair of distinct periodic directions $(d,d')$ and its associated geodesic $\gamma_{d,d'}$ on $\HH$, we denote by $V(d,d')$ the connected component of $\HH \backslash (\Gamma_n^{\pm} \cdot \gamma_{d,d'})$ containing $\XX$.
\end{Def}

As an example, the domain $V(\infty,d')$ with $d' = \frac{\Phi}{2\Phi^2-1}$ is drawn in Figure \ref{fig:example_voronoi_domain}.

\begin{Rema}
If a geodesic of $\Gamma_n^{\pm} \cdot \gamma_{d,d'}$ passes through $\XX$, then $V(d,d')$ is not well defined. For convenience, we set $V(d,d') = \{ \XX \}$ as in this case we automatically have:
\[ K(d,d') \leq K(\infty, \frac{1}{\Phi}) \sin \theta( \XX, \infty, \frac{1}{\Phi}) \]
and since $\XX$ is the furthest point away from the set of geodesics $\mathcal{G}_{max}$, we have for any surface $X \in \TT$:
\[ K(d,d') \sin \theta(X,d,d') \leq K(d,d') \leq K(\infty, \frac{1}{\Phi}) \sin \theta( \XX, \infty, \frac{1}{\Phi}). \]
\end{Rema}

\begin{figure}
\center
\definecolor{ccqqqq}{rgb}{0.8,0,0}
\begin{tikzpicture}[line cap=round,line join=round,>=triangle 45,x=3cm,y=3cm]
\clip(-2.2,-0.5) rectangle (2.2,1.5);
\filldraw[fill=gray!20] (1.848,0) -- (1.848,0) arc (0:180:0.825431473300645) (0.198,0) -- cycle;
%\filldraw[fill=gray!20] (-0.198,0) -- (-0.198,0) arc (0:180:0.825431473300645) (-1.848,0) -- cycle;
\filldraw[fill=white] (1.082,0) -- (1.082,0) arc (0:180:0.541196100146197) (0,0) -- cycle;
\filldraw[fill=white, color=white] (0.9238795325112867,0) rectangle (2.5,4); 
\filldraw[fill=white] (0,0) -- (0,0) arc (0:180:0.541196100146197) (-1.082,0) -- cycle;
\filldraw[fill=white] (-0.9238795325112867,0) rectangle (-2.5,4); 
\draw [shift={(0.541196100146197,0)},line width=1pt]  plot[domain=0:3.141592653589793,variable=\t]({1*0.541196100146197*cos(\t r)+0*0.541196100146197*sin(\t r)},{0*0.541196100146197*cos(\t r)+1*0.541196100146197*sin(\t r)});
%\draw [shift={(0,0)},line width=1pt,dash pattern=on 3pt off 3pt]  plot[domain=0:3.141592653589793,variable=\t]({1*1*cos(\t r)+0*1*sin(\t r)},{0*1*cos(\t r)+1*1*sin(\t r)});
\draw [line width=1pt] (0.9238795325112867,0) -- (0.9238795325112867,3.0472427983539028);
\draw [line width=1pt] (-0.9238795325112867,0) -- (-0.9238795325112867,3.0472427983539028);
\draw [shift={(-0.541196100146197,0)},line width=1pt]  plot[domain=0:3.141592653589793,variable=\t]({1*0.541196100146197*cos(\t r)+0*0.541196100146197*sin(\t r)},{0*0.541196100146197*cos(\t r)+1*0.541196100146197*sin(\t r)});
\draw [shift={(-1.0238408779149524,0)},line width=1pt,color=ccqqqq]  plot[domain=0.6763286696673477:1.4493964179776777,variable=\t]({1*0.825431473300645*cos(\t r)+0*0.825431473300645*sin(\t r)},{0*0.825431473300645*cos(\t r)+1*0.825431473300645*sin(\t r)});
\draw [shift={(0.04897119341563786,0)},line width=1pt,color=ccqqqq]  plot[domain=1.0008387825741267:2.2098093501643175,variable=\t]({1*0.6136726703570244*cos(\t r)+0*0.6136726703570244*sin(\t r)},{0*0.6136726703570244*cos(\t r)+1*0.6136726703570244*sin(\t r)});
\draw [shift={(0.8239181871076211,0)},line width=1pt,color=ccqqqq]  plot[domain=1.4493964179776777:3.141592653589793,variable=\t]({1*0.8254314733006451*cos(\t r)+0*0.8254314733006451*sin(\t r)},{0*0.8254314733006451*cos(\t r)+1*0.8254314733006451*sin(\t r)});
\draw [line width=1pt,color=ccqqqq] (0.3170253355622144,0.49258571550470803) -- (0.3170253355622144,3.0472427983539028);

\draw [shift={(1.0238408779149524,0)},line width=1pt,color=ccqqqq]  plot[domain=1.69229:2.465262,variable=\t]({1*0.825431473300645*cos(\t r)+0*0.825431473300645*sin(\t r)},{0*0.825431473300645*cos(\t r)+1*0.825431473300645*sin(\t r)});
\draw [shift={(-0.8239181871076211,0)},line width=1pt, color=ccqqqq]  plot[domain=0:1.69229,variable=\t]({1*0.8254314733006451*cos(\t r)+0*0.8254314733006451*sin(\t r)},{0*0.8254314733006451*cos(\t r)+1*0.8254314733006451*sin(\t r)});
\draw [shift={(-0.04897119341563786,0)},line width=1pt,color=ccqqqq]  plot[domain=0.938:2.1415,variable=\t]({1*0.6136726703570244*cos(\t r)+0*0.6136726703570244*sin(\t r)},{0*0.6136726703570244*cos(\t r)+1*0.6136726703570244*sin(\t r)});
\draw [line width=1pt,color=ccqqqq] (-0.3170253355622144,0.49258571550470803) -- (-0.3170253355622144,3.0472427983539028);

\draw (0,-0.1) node[anchor=north west] {$d' = \frac{\Phi}{2\Phi^2-1}$};
\draw [line width=0.4pt,domain=-2.2700411522633757:2.93292181069959] plot(\x,{(-0-0*\x)/1});
%\draw [line width=0.4pt] (0,0) -- (0,3.0472427983539028);
\begin{scriptsize}
\draw [fill=black] (0.31,0) circle (2pt);
\draw [fill=black] (0,1) circle (2pt);
\draw [fill=black] (0.9238,0.38) circle (2pt);
\end{scriptsize}
\draw (0,1) node[anchor=south east] {$\SS$};
\draw (0.9238,0.38) node[anchor=south west] {$\XX$};
\draw (0,0) node[below] {$0$};
\draw (0.9238,0) node[below] {$\frac{\Phi}{2}$};
\draw (-0.9238,0) node[below] {$-\frac{\Phi}{2}$};
\draw [line width=0.5pt,to-] (0.8,0.6)--(1.2,1);
\draw (1.05,1.1) node[right] {$V(\infty,\frac{\Phi}{2\Phi^2-1}) \cap \TT$};
%\draw [line width=0.5pt,to-] (-0.6,0.6)--(1.2,1);
\end{tikzpicture}
\caption{The geodesics of $\Gamma_n^{\pm} \cdot \gamma_{\infty,\frac{\Phi}{2\Phi^2-1}}$ intersecting the fundamental domain $\TT$, and the domain $V(\infty, \frac{\Phi}{2\Phi^2-1}) \cap \TT$.}
\label{fig:example_voronoi_domain}
\end{figure}
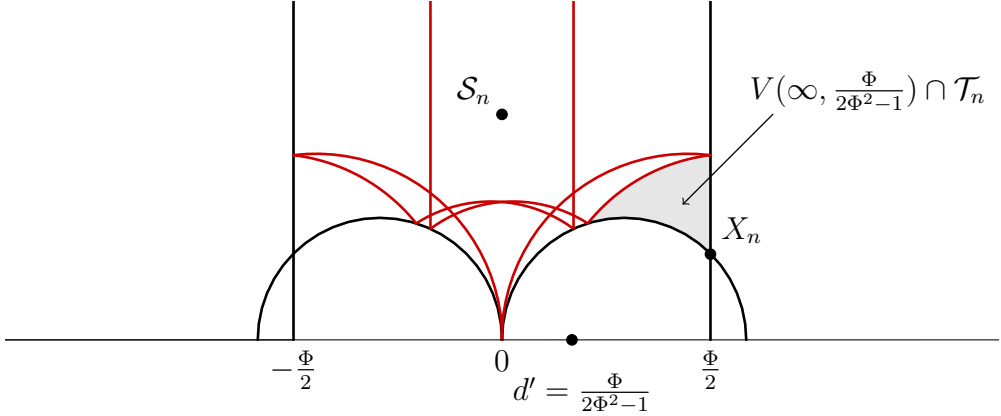

With this definition, the proof of Lemmas 7.6 and 7.7 of \cite[\S 7.2]{BLM22} generalizes to our setting and gives:
\begin{Lem}
Let $(d_0,d_0')$ be a pair of distinct periodic directions. Assume that Equation \eqref{eq:trefle} holds for any pair of directions $(d,d')$ whose associated geodesic lie in the boundary of the domain $V(d_0,d_0')$. Then the same inequality is true for $(d_0,d_0')$.\newline
Furthermore, a pair of directions $(d,d')$, $d<d'$ whose associated geodesic lies in the boundary of the domain $V(d_0,d_0') \cap \TT_{+}$ has to satisfy $d'+d \geq \Phi$.
\end{Lem}

In particular, it suffices to prove Equation \eqref{eq:trefle} for pairs of directions $(d,d')$ with $d+d' \geq \Phi$. This relies on the sinus comparison techniques of \cite[\S 7.3]{BLM22}.

\begin{Prop}\label{prop:extension_case_1}
Let $(d,d')$  with $d+d' \geq \Phi$ and $X=x+iy \in \TT$ inside the half disk defined by the geodesic $\gamma_{\frac{1}{\Phi}, \Phi - \frac{1}{\Phi}}$. We have
\[
 K(d,d') \sin \theta(X,d,d') \leq K(\infty,\frac{1}{\Phi}) \sin \theta(X,\infty,\frac{1}{\Phi}).
\]
This condition is verified in particular for $X \in \TT \backslash \mathcal{R}_+$.
\end{Prop}

\begin{proof}[Proof of Proposition \ref{prop:extension_case_1}]
We distinguish two cases:
\begin{enumerate}
\item In the case $d \geq \frac{1}{\Phi}$, we use Proposition 7.8 of \cite{BLM22} which can be stated more generally as:

\begin{Prop}\label{prop:extension_Prop_BLM}
Let $a \in \RR$, $b > 0$, and $c \geq b$. Let $\mathcal{D} \subset \HH$ be the domain enclosed by the geodesics $\gamma_{a, \infty}$, $\gamma_{a + b, \infty}$ and $\gamma_{a-c, a+c}$, and $X_0$ be the right corner of the domain $\mathcal{D}$, as in Figure \ref{fig:prop_BLM}. Then for any $(d,d')$ with $a \leq d \leq a + b \leq d'$ such that $\gamma_{d,d'}$ intersect the domain $\mathcal{D}$, the function
\[
F_{(d,d')} : X \in \mathcal{D} \mapsto \frac{\sin \theta(X,\infty,a)}{\sin \theta(X,d,d')}
\]
is minimal at $X_0$ on the domain $\mathcal{D}$.
\end{Prop}

\begin{figure}
\center
\definecolor{qqwuqq}{rgb}{0,0.39215686274509803,0}
\definecolor{zzttqq}{rgb}{0.6,0.2,0}
\definecolor{uuuuuu}{rgb}{0.26666666666666666,0.26666666666666666,0.26666666666666666}
\begin{tikzpicture}[line cap=round,line join=round,>=triangle 45,x=1cm,y=1cm, scale=0.8]
\clip(-4.2,-1) rectangle (9,7);
\fill[line width=1pt,color=gray!20,fill=gray!20] (0,0) -- (3,0) -- (3,8.64) -- (0,8.64) -- cycle;
\draw [shift={(0,0)},line width=1pt,color=white,fill=white]  (0,0) --  plot[domain=0:1.5707963267948966,variable=\t]({1*3.56*cos(\t r)+0*3.56*sin(\t r)},{0*3.56*cos(\t r)+1*3.56*sin(\t r)}) -- cycle ;
\draw [line width=1pt] (0,0) -- (0,8.5);
\draw [shift={(0,0)},line width=1pt]  plot[domain=0:3.141592653589793,variable=\t]({1*3.56*cos(\t r)+0*3.56*sin(\t r)},{0*3.56*cos(\t r)+1*3.56*sin(\t r)});
\draw [line width=1pt] (3,0) -- (3,8.5);
\draw [line width=1pt,domain=-4.893333333333335:12.426666666666671] plot(\x,{(-0-0*\x)/3.56});
\draw (0,-0.1) node[below] {$a$};
\draw (2.826666666666668,0.05) node[below] {$a+b$};
\draw (3.5,-0.4) node[below] {$a+c$};
\draw (-3.5,-0.1) node[below] {$a-c$};
\draw (0.7666666666666672,0) node[below] {$d$};
\draw (7.906666666666671,0) node[below] {$d'$};
\draw [shift={(4.5,0)},line width=1pt]  plot[domain=0:3.141592653589793,variable=\t]({1*3.6*cos(\t r)+0*3.6*sin(\t r)},{0*3.6*cos(\t r)+1*3.6*sin(\t r)});
\draw (3.4,2.4) node[anchor=north west] {$X_0$};
\draw (6.8,3.8) node[left] {$\gamma_{d,d'}$};
\draw (1.2066666666666672,6.02) node[anchor=north west] {$\mathcal{D}$};
\begin{scriptsize}
\draw [fill=uuuuuu] (3,1.9166637681137504) circle (2pt);
\end{scriptsize}
\end{tikzpicture}
\caption{The domain $\mathcal{D}$ of Proposition \ref{prop:extension_Prop_BLM}.}
\label{fig:prop_BLM}
\end{figure}
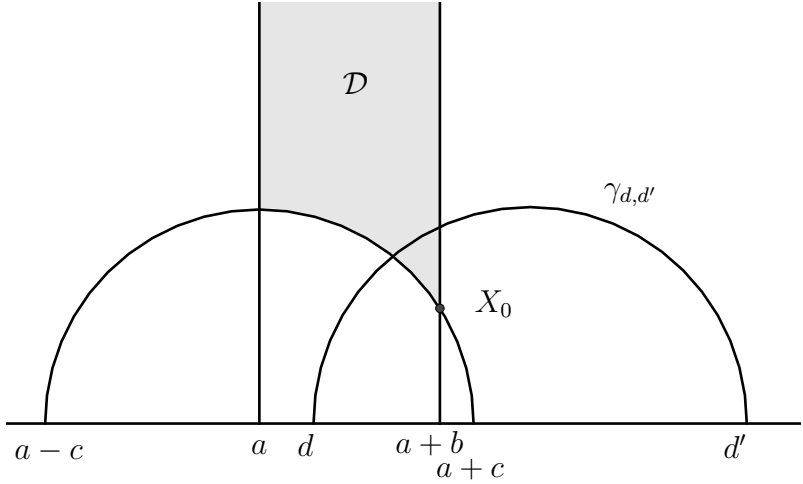

In our setting, we set $a = \frac{1}{\Phi}$, $b = \frac{\Phi}{2} - \frac{1}{\Phi}$ and $c=\frac{1}{\Phi}$ and it gives that for any $(d,d')$ with $\frac{1}{\Phi} \leq d \leq \frac{\Phi}{2} \leq d'$, the function

\[
F_{(d,d')} : X \mapsto \frac{\sin \theta(X,\infty,\frac{1}{\Phi})}{\sin \theta(X,d,d')}
\]
is minimal at $\XX$ on the domain $\mathcal{D} = \{X=x+iy \in \TT$ with $x \geq \frac{1}{\Phi} \}$.\newline
In particular, since by Corollary \ref{cor:K_regular}
\[
\forall (d,d'), \text{ } K(d,d') \sin \theta(\XX,d,d') \leq K(\infty, \frac{1}{\Phi})\sin \theta(\XX, \infty, \frac{1}{\Phi})
\]

we deduce that for all $X \in \{X=x+iy \in \TT$ with $x \geq \frac{1}{\Phi} \}$ and $(d,d')$ with $\frac{1}{\Phi} \leq d \leq \frac{\Phi}{2} \leq d'$, we have
\[
 K(d,d') \sin \theta(X,d,d') \leq K(\infty,\frac{1}{\Phi}) \sin \theta(X,d,d').
\]

\item Else, $d < \frac{1}{\Phi}$ so that $d' > \Phi - \frac{1}{\Phi}$. In this case, for any pair of directions $(d,d')$ not in $\Gamma \cdot \mathcal{G}_{max}$ and such that $d \leq \frac{1}{\Phi}$ and $d' \geq \Phi - \frac{1}{\Phi}$, and for all $X \in \mathcal{R}_+$, we have
$\sin \theta (X,d,d') \leq \sin \theta (X,\frac{1}{\Phi}, \Phi - \frac{1}{\Phi})$
so that
\begin{align*}
K(d,d') \sin \theta(X,d,d') & \leq K(d,d') \sin \theta(X,\frac{1}{\Phi},\Phi - \frac{1}{\Phi}) \ & \\
& \leq K(\frac{1}{\Phi},\Phi - \frac{1}{\Phi}) \sin \theta(X,\frac{1}{\Phi},\Phi - \frac{1}{\Phi}) & \text{ (by Proposition \ref{prop:etude_K})}\\
& \leq K(\infty,\frac{1}{\Phi}) \sin \theta(X,\infty,\frac{1}{\Phi}) &\text{ (by case 1.)}
\end{align*}
This concludes the proof.
\end{enumerate}
\end{proof}

\section{Proof of Theorem \ref{theo:main2}}\label{sec:4m+2}
In the case where $n \equiv 2 \mod 4$, saddle connections are not necessarily closed anymore, and it is no longer possible to use techniques from Sections 4 and 5. However, we can still show that $\KVol$ is bounded on the Teichm\"uller disk of $\XX$. In \S6.1, we extend a boundedness criterion of \cite[\S 3]{BLM22} to the case of multiple singularities. In \S6.2, we use this criterion to show that KVol is bounded on the Teichm\"uller disk of the regular $n$-gon.

\subsection{Boundedness criterion for multiple singularities}
In this paragraph, we show Theorem \ref{theo:boundedness_criterion_0} which extends the boundedness criterion of \cite{BLM22} to the case of translation surfaces with multiple singularities.
\paragraph{Theorem 1.5.}
$\KVol$ is bounded on the Teichm\"uller disk of a Veech surface $X$ if and only if there are no intersecting closed curves $\alpha$ and $\beta$ on $X$ such that $\alpha = \alpha_1 \cup \cdots \cup \alpha_k$ and $\beta = \beta_1 \cup \cdots \cup \beta_l$ are unions of parallel saddle connections (that is all saddle connections $\alpha_1, \dots , \alpha_k,\beta_1, \dots, \beta_l$ have the same direction).

\begin{proof}
First, notice that if there exist a pair $(\alpha,\beta)$ of parralel closed curves, then applying the Teichm\"uller geodesic flow in their common direction make the length of both $\alpha$ and $\beta$ go to zero, while the interserction remain unchanged. In particular

\begin{equation*} 
\KVol(g_t \cdot X) \to + \infty \text{ as } t \to +\infty. 
\end{equation*}

Conversely, let $X$ be a Veech surface, which we assume to be of unit area. Recall that for such a surface, the "no small triangle condition" of \cite{Vorobets} (see also \cite{SW}) gives a constant $A > 0$ such that for any two saddle connections $\alpha$ and $\beta$, the inequality $|\alpha \wedge \beta| \geq A$ holds. Further, the constant $A$ does not depend on the choice of the surface in the Teichm\"uller disk of $X$.

Now, let $\alpha$ be a saddle connection in a periodic direction. The Veech surface $X$ decomposes into cylinders in the direction of $\alpha$. Let $h$ be the smallest height of the cylinders. For any saddle connection $\beta$ which is not parallel to $\alpha$, $\beta$ has at least a vertical length $h$ between each non-singular intersection with $\alpha$, plus a vertical length at least $h$ before the first non-singular intersection, and after the last non-singular intersection with $\alpha$. In particular, for any saddle connection $\beta$ having at least one non-singular intersection with $\alpha$, we have:
\[ h(\beta) \geq h(|\alpha \cap \beta|+1) \]
where $h(\beta) = l(\beta) \sin angle (\alpha,\beta)$. In fact, if $\beta$ does not intersect $\alpha$ non-singularly then the above inequality still holds, as it becomes $h(\beta) \geq h$ which is true as long as $\beta$ is not parallel to $\alpha$. Finally, since $h(\beta) \leq l(\beta)$, we conclude that :

\begin{equation*}
\frac{ |\alpha \cap \beta| + 1}{l(\alpha)l(\beta)} \leq \frac{1}{l(\alpha) h} \leq \frac{1}{A}. \end{equation*}
The last inequalty comes from the fact that $l(\alpha)h = \alpha \wedge \beta_0 \geq A$ for a saddle connection $\beta_0$ which stays inside the cylinder of height $h$.\newline

Next, take $\alpha$ and $\beta$ two simple closed curves decomposed as an union of saddle connections $\alpha = \alpha_1 \cup \cdots \cup \alpha_k$ and $\beta = \beta_1 \cup \cdots \cup \beta_l$, and assume that $\alpha_1, \dots, \alpha_k,\beta_1 , \dots ,\beta_l$ are not all parallel. Then, 
\[ \Int(\alpha,\beta) \leq (\sum_{
\begin{scriptsize}
 \begin{array}{c}
1 \leq i \leq k \\
1 \leq j \leq l
\end{array}
\end{scriptsize}} |\alpha_i \cap \beta_j|) + s \]
where $s \leq \min(k,l)$ denotes the number of common singularities between $\alpha$ and $\beta$. It should be noted that if $\alpha_i = \beta_j$, we set $|\alpha_i \cap \beta_j| = 0$.\newline
Now, since saddle connections are not all parallel, there is at least $\min(k,l)$ pairs $(i,j)$ such that $\alpha_i$ and $\beta_j$ are not parallel, and we get
\begin{align*}
\Int(\alpha,\beta) & \leq (\sum_{
\begin{scriptsize}
 \begin{array}{c}
1 \leq i \leq k \\
1 \leq j \leq l
\end{array}
\end{scriptsize}} |\alpha_i \cap \beta_j|) + s \\ 
& \leq \sum_{
\begin{scriptsize}
 \begin{array}{c}
i,j \\
\alpha_i  \text{ and } \beta_j \text{ non-parallel}
\end{array}
\end{scriptsize}} (|\alpha_i \cap \beta_j| + 1)\\
& \leq  \sum_{
\begin{scriptsize}
 \begin{array}{c}
i,j \\
\alpha_i  \text{ and } \beta_j \text{ non-parallel}
\end{array}
\end{scriptsize}} \frac{1}{A}\times l(\alpha_i)l(\beta_j) \\
& \leq \frac{1}{A} l(\alpha) l(\beta).
\end{align*}
This gives the required boundedness result
\[ \KVol(X) \leq \frac{1}{A}. \]
\end{proof}

\subsection{Intersection of parallel curves on $\XX$, $n \equiv 2 \mod 4$}
In this paragraph, we go back to the regular $n$-gon for $n \equiv 2 \mod 4$ and we study the intersection $\Int(\alpha,\beta)$ in the case where $\alpha$ and $\beta$ are union of parallel saddle connections. Up to the action of the Veech group, we can assume that $\alpha$ and $\beta$ are either both horizontal or both vertical. We will work in the staircase model $\mathcal{S}$, as in Figure \ref{staircase_model_v2}. Notice that in the horizontal direction, saddle connections go from one singularity to the other while saddle connections are closed curves in the vertical direction. In this latter case, it is easy to show:

\begin{Lem}[Vertical curves]
For every $i,j$, $\Int(\beta_i,\beta_j) = 0$
\end{Lem}
\begin{proof}
First, since $\beta_m$ is homologous to a non-singular curve which do not intersect any of the $\beta_j$ for $j<m$, we have $\Int(\beta_m, \beta_j) = 0$.\newline
Next, for any $j<m$ the curve $\beta_j + \beta_{j+1}$ is homologous to a non-singular curve and for all $i$, $\Int(\beta_j + \beta_{j+1}, \beta_i)=0$ and hence $\Int(\beta_j,\beta_i) = - \Int(\beta_{j+1},\beta_i)$ and by induction $\Int(\beta_j,\beta_i) = \pm \Int(\beta_m,\beta_i) = 0$.
\end{proof}

In the case where $\alpha$ (resp. $\beta$) is an horizontal curve made of two horizontal saddle connections $\alpha_{i_1}$ and $\alpha_{i_2}$ (resp. $\alpha_{j_1}$ and $\alpha_{j_2}$) going from one singularity to the other and oriented such that the resulting curve $\alpha = \alpha_{i_1} \pm \alpha_{i_2}$ has a well defined orientation, we have:
\begin{Lem}[Horizontal curves]
In this setting, $\Int(\alpha,\beta) = 0$.
\end{Lem}
\begin{proof}
First notice that, given the orientation of the saddle connections $\alpha_i$ is from left to right as in Figure \ref{staircase_model_v2}, the curve $\alpha$ has a well defined orientation if and only if $\alpha =\pm [\alpha_{i_1} + (-1)^{i_2-i_1-1}\alpha_{i_2}]$. In this case, we can write (assuming $i_1 < i_2$ for convenience)
\begin{align*}
\alpha &=\pm [ (\alpha_{i_1} + \alpha_{i_1+1}) -  (\alpha_{i_1+1} + \alpha_{i_1+2}) +\cdots + (-1)^{i_2-i_1-1}(\alpha_{i_2-1} + \alpha_{i_2})]\\
& = \pm \sum_{k=0}^{i_2-i_1-1} (-1)^{k}(\alpha_{i_1+k} + \alpha_{i_1+k+1}).
\end{align*}
Now, since the curves $\alpha_i + \alpha_{i+1}$ are closed curves homologous to core curve of horizontal cylinder, they are pairwise non-intersecting. Making the same decomposition for $\beta$ gives directly $\Int(\alpha,\beta)=0$.
\end{proof}

\begin{figure}
\center
\definecolor{ccqqqq}{rgb}{0.8,0,0}
\begin{tikzpicture}[line cap=round,line join=round,>=triangle 45,x=1cm,y=1cm, scale = 0.6]
\clip(-10.1,-1) rectangle (2,10);
\draw [-to, line width=1pt] (0,0)-- (0.6,0);
\draw [line width=1pt] (0.6,0) -- (1,0);
\draw [line width=1pt] (1,0)-- (1,1.9318516525781366);
\draw [line width=1pt] (1,1.9318516525781366)-- (0,1.9318516525781366);
\draw [line width=1pt] (0,1.9318516525781366)-- (0,5.277916867529369);
\draw [line width=1pt] (0,5.277916867529369)-- (-2.7320508075688776,5.277916867529369);
\draw [-to,line width=1pt] (-2.7320508075688776,0)-- (-1.26,0);
\draw [line width=1pt] (-1.26,0) -- (0,0);
%\draw [line width=1pt,dash pattern=on 3pt off 3pt] (-2.7320508075688776,5.277916867529369)-- (-6.464101615137755,5.277916867529369);
\draw (-1.8,0) node[anchor=north west] {$\alpha_3$};
\draw (-5,1.9) node[anchor=north west] {$\alpha_2$};
\draw (0,3.9) node[anchor=north west] {$\beta_2$};
\draw (-2.65,7.7) node[anchor=north west] {$\beta_1$};
%\draw [line width=1pt,dash pattern=on 3pt off 10pt] (0,1.9318516525781366)-- (-2.7320508075688776,1.9318516525781368);
\draw [line width=1pt] (-2.7320508075688776,1.9318516525781368)--(-4.5,1.93185);
\draw [to-,line width=1pt](-4.5,1.93185)-- (-6.464101615137755,1.9318516525781366);
\draw [line width=1pt] (-6.464101615137755,1.9318516525781366)-- (-6.464101615137755,5.277916867529369);
\draw [-to,line width=1pt] (-9.9,5.277916867529369)--(-8.1,5.2779);
\draw [line width=1pt] (-8.1,5.2779) -- (-6.464101615137755,5.277916867529369);
\draw [line width=1pt] (-9.9,5.277916867529369)-- (-9.9,9.141620172685643);
\draw [line width=1pt] (-9.9,9.141620172685643)-- (-2.7320508075688767,9.141620172685643);
\draw [line width=1pt] (-2.7320508075688767,9.141620172685643)-- (-2.7320508075688776,5.277916867529369);
%\draw [line width=1pt,dash pattern=on 3pt off 3pt] (-2.7320508075688776,1.9318516525781368)-- (0,1.9318516525781368);
\draw [line width=1pt] (-2.7320508075688776,1.9318516525781368)-- (-2.7320508075688776,0);
%\draw [line width=1pt,dash pattern=on 3pt off 10pt] (-6.464101615137756,5.277916867529369)-- (-2.7320508075688776,5.277916867529369);
\draw (0,0) node[anchor=north west] {$\alpha_4$};
\draw (1,1.4) node[anchor=north west] {$\beta_3$};
%\draw (-1.5,1.5) node[anchor=north west] {$Z_3$};
%\draw (-4,4) node[anchor=north west] {$Z_2$};
%\draw (-7.2,7.7) node[anchor=north west] {$Z_1$};
\draw (-9,5) node[anchor=north west] {$\alpha_1$};
\draw [color = ccqqqq] (1,0) node {\large x};
\draw [color = ccqqqq] (1,1.9318516525781366) node {\large x};
\draw [color = ccqqqq] (-2.7320508075688776,5.277916867529369) node {\large x};
\draw [color = ccqqqq] (-2.7320508075688776,0) node {\large x};
\draw [color = ccqqqq] (-2.7320508075688776,1.9318516525781368) node {\large x};
\draw [color = ccqqqq] (-2.7320508075688767,9.141620172685643) node {\large x};
\draw [color = ccqqqq] (-9.9,5.277916867529369) node {\large x};
\draw [color = ccqqqq] (-9.9,9.141620172685643) node {\large x};
\draw (0,0) circle (5pt);
\draw (0,1.9318516525781366)  circle (5pt);
\draw (0,5.277916867529369) circle (5pt);
\draw (-6.464101615137755,9.141620172685643) circle (5pt);
\draw (-6.464101615137755,1.9318516525781366)  circle (5pt);
\draw (-6.464101615137755,5.277916867529369) circle (5pt);
\end{tikzpicture}
\caption{The staircase model associated with the $14$-gon. The horizontal and the vertical direction represent the two cuspidal directions. Notice that vertical saddle connections are closed while horizontal saddle connections are not.}
\label{staircase_model_v2}
\end{figure}

As a corollary, we can apply the criterion of Theorem \ref{theo:boundedness_criterion_0} to deduce that KVol is bounded on the Teichm\"uller disk of the regular $n$-gon. In fact, a closer look at the proof of Theorem \ref{theo:boundedness_criterion_0} gives the following explicit bound:
\begin{Cor}\label{cor:boundedness}
Let $n \geq 10$, $n \equiv 2 \mod 4$. For any surface $X$ in the Teichm\"uller disk of the regular $n$-gon, and any closed curves $\alpha$ and $\beta$ on $X$, we have:
\[
\frac{\Int(\alpha, \beta)}{l(\alpha) l(\beta)} \leq  \frac{1}{\Phi l_m^2}.
\]
\end{Cor}

\bibliographystyle{alpha}
\bibliography{KVol_bibli}
\end{document}